\documentclass[12pt]{article}
\usepackage{amsthm,amsfonts,amssymb,epsfig,graphics,amsmath,amsbsy,color,enumerate,tikz,authblk}

\usepackage{blindtext}
\newcommand\blfootnote[1]{%
\begingroup
\renewcommand\thefootnote{}\footnote{#1}%
\addtocounter{footnote}{-1}%
\endgroup
}

\mathchardef\mhyphen="2D

\theoremstyle{plain}
\newtheorem{theorem}{Theorem}
\newtheorem{lemma}{Lemma}
\newtheorem{proposition}{Proposition}
\newtheorem{remark}{Remark}
\newtheorem{definition}{Definition}

\newcommand{\Span}{\operatorname{Span}}

\newcommand{\ess}{\operatorname{ess}}
\newcommand{\tr}{\operatorname{tr}}
\newcommand{\ind}{\operatorname{Ind}}
\newcommand{\sech}{\operatorname{sech}}
\newcommand{\sgn}{\operatorname{sgn}}
\newcommand{\re}{\operatorname{Re}}

\numberwithin{equation}{section}
\numberwithin{lemma}{section}
\numberwithin{theorem}{section}
\numberwithin{remark}{section}
\numberwithin{claim}{section}
\numberwithin{corollary}{section}
\numberwithin{proposition}{section}
\numberwithin{definition}{section}
\numberwithin{condition}{section}
\numberwithin{figure}{section}

\DeclareFontFamily{OT1}{pzc}{}
\DeclareFontShape{OT1}{pzc}{m}{it}{<-> s * [1.100] pzcmi7t}{}
\DeclareMathAlphabet{\mathpzc}{OT1}{pzc}{m}{it}

\title{Oscillation Theory and Instability  
of Nonlinear Waves} 

\author{Peter Howard}
\affil{Texas A\&M University}

\begin{document}

\maketitle

\blfootnote{MSC 2020 subject classes: 34C10, 34L15}

\begin{abstract} 
In recent work, Baird et al. have introduced a generalized
Maslov index which allows oscillation techniques that have 
previously been restricted to eigenvalue problems with 
underlying Hamiltonian structure to be extended to the 
non-Hamiltonian setting
[T. J. Baird, P. Cornwell, G. Cox, C. Jones, and R. Marangell,
{\it Generalized Maslov indices for non-Hamiltonian systems},
SIAM J. Math. Anal. {\bf 54} (2022) 1623-1668]. 
We show that this approach can be implemented in the 
analysis of spectral instability for nonlinear waves, 
taking as our setting a class of equations previously 
investigated by Pego and Weinstein via the Evans function
[R. L. Pego and M. I. Weinstein,
{\it Eigenvalues, and instabilities of solitary waves},
Phil. Trans. R. Soc. Lond. A {\bf 340} (1992) 47-94].
\end{abstract}

\section{Introduction}
\label{introduction}

For values of $\lambda$ in a real open interval $I \subset \mathbb{R}$, 
we consider first-order ODE systems 
\begin{equation} \label{nonhammy}
    \frac{dy}{dx} = A (x; \lambda) y, 
    \quad x \in \mathbb{R}, \quad y(x; \lambda) \in \mathbb{R}^n, 
    \quad n \in \{2, 3, \dots \},
\end{equation}
for which we make the following assumptions, adapted from 
\cite{PW1992}: 

\medskip
\noindent
{\bf (A)} We assume that for some open set $\Omega \subset \mathbb{C}$ containing $I$,
we have $A \in C (\mathbb{R} \times \Omega, \mathbb{C}^{n \times n})$,
with $A (x; \lambda) \in \mathbb{R}^{n \times n}$ for all 
$(x, \lambda) \in \mathbb{R} \times I$, and that 
the map $\lambda \mapsto A (x; \lambda)$ is analytic on $\Omega$ for every 
$x \in \mathbb{R}$. 

\medskip
\noindent
{\bf (B)} For each $\lambda \in \Omega$, the limits 
\begin{equation*}
    A_{\pm} (\lambda) := \lim_{x \to \pm \infty} A(x; \lambda)
\end{equation*}
exist and are obtained uniformly on compact subsets of $\Omega$. 

\medskip
\noindent
{\bf (C)} For each $\lambda \in \Omega$, each of the matrices 
$A_{\pm} (\lambda)$ has a unique (and so necessarily real for $\lambda \in I$)
eigenvalue of largest real part, which is simple. We denote 
this eigenvalue $\mu_{\pm} (\lambda)$ and denote by 
$\mu_{\pm}^* (\lambda)$ the largest real part of any other 
eigenvalue of $A_{\pm} (\lambda)$, so that 
\begin{equation*}
    \textrm{Re}\, \mu_{\pm} (\lambda)
    > \mu_{\pm}^* (\lambda)
    := \max \{\textrm{Re}\, \nu: \nu \ne \mu_{\pm} (\lambda),\,\, \nu \in \sigma (A_{\pm} (\lambda)) \}.
\end{equation*}
Moreover, $\mu_{\pm} (\lambda)$ is analytic on $\Omega$, and 
there exists an analytic (in $\Omega$) choice of left and right eigenvectors 
of $A_{\pm} (\lambda)$, satisfying
\begin{equation*}
    (A_{\pm} (\lambda) - \mu_{\pm} (\lambda) I) v^{\pm} (\lambda) = 0,
    \quad w^{\pm} (\lambda) (A_{\pm} (\lambda) - \mu_{\pm} (\lambda) I) = 0,
    \quad w^{\pm} (\lambda) v^{\pm} (\lambda) = 1.
\end{equation*}

\medskip
\noindent
{\bf (D)} The function 
\begin{equation*}
    E(x; \lambda) := 
    \begin{cases}
    A(x; \lambda) - A_- (\lambda) & x < 0 \\
    A(x; \lambda) - A_+ (\lambda) & x > 0 
    \end{cases}
\end{equation*}
satisfies the following two conditions: (i) $\int_{\mathbb{R}} |E (x; \lambda)| dx$ 
is finite for all $\lambda \in \Omega$; and (ii) this integral converges 
uniformly on compact subsets of $\Omega$. 

\medskip

As with \cite{PW1992}, our analysis is primarily motivated by consideration
of stability for traveling-wave solutions $\bar{u} (x - st)$ arising 
in the context of single higher order nonlinear evolutionary PDE, 
\begin{equation} \label{higher-order}
    u_t + g(u) + f(u)_x 
    = \sum_{k=2}^n (b_k (u) \partial_x^{k-1} u)_x,
    \quad n \in \{2, 3, \dots \}.
\end{equation}
By replacing $f$ with $f - su$, we can regard such solutions as 
stationary solutions $\bar{u} (x)$ to the same equation, and 
upon linearization (with $u = \bar{u} + v$), we obtain 
\begin{equation} \label{perturbation-equation}
    v_t + a_0 (x) v + (a_1 (x) v)_x 
    = \sum_{k = 2}^{n} (a_k (x) \partial^{k-1}_x v)_x,    
\end{equation}
with 
\begin{equation*}
    a_0 (x) = g'(\bar{u} (x)),
    \quad a_1 (x) = f' (\bar{u} (x)) - \sum_{k=2}^n {b_k}' (\bar{u} (x)) \partial^{k-1}_x \bar{u} (x),
\end{equation*}
and 
\begin{equation*}
a_k (x) = b_k (\bar{u} (x)), 
\quad k = 2, 3, \dots, n.    
\end{equation*}
The associated eigenvalue problem can be expressed as 
\begin{equation*} \label{evp}
    L \phi := - \sum_{k = 2}^{n} (a_k (x) \phi^{(k-1)})' + (a_1 (x) \phi)' + a_0 (x) \phi 
    = \lambda \phi,
\end{equation*}
where our sign convention is taken to be consistent with cases in 
which (\ref{higher-order}) is second order and the negative sign 
corresponds with a positive operator. We now obtain (\ref{nonhammy})
by expressing (\ref{evp}) as a first-order system with 
$y = (y_1 \,\, y_2 \,\, \dots \,\, y_n)^T$, 
$y_k = \phi^{(k-1)}$, $k = 1, 2, \dots, n-1$, 
$y_n = a_n (x) \phi^{(n-1)}$. Specific matrices arising 
in this way will be considered in the applications 
discussed in Section \ref{applications-section}. 

Under Assumption {\bf (C)}, there exists a one-dimensional subspace of 
solutions of (\ref{nonhammy}) that decay at the maximal  
exponential rate $e^{\mu_- (\lambda) x}$
as $x$ tends to $-\infty$, and likewise an $(n-1)$-dimensional subspace 
of solutions to (\ref{nonhammy}) that fail to grow at the maximal rate
$e^{\mu_+ (\lambda) x}$ as $x$ tends to $+ \infty$. 
Following \cite{PW1992}, our goal will be to identify values 
$\lambda \in I$ for which there exists a solution 
$y (\cdot; \lambda) \in C^1 (\mathbb{R}, \mathbb{R}^n)$ 
of (\ref{nonhammy}) that lies in the intersection of these two spaces. 
Although we will refer to such values as eigenvalues throughout 
the analysis, we observe here that they are only ``genuinely"
eigenvalues in the event that $\mu_- (\lambda)$ is strictly positive and is 
the only positive eigenvalue of $A_- (\lambda)$, and additionally 
$\mu_+ (\lambda)$ is non-negative and is the only eigenvalue of 
$A_+ (\lambda)$ with non-negative real part. In the usual way, if 
$\lambda$ is an eigenvalue of (\ref{nonhammy}), we will refer to the dimension of the 
space of all associated solutions as the geometric multiplicity of $\lambda$. 
Our main goal is to show that oscillation 
theory can be used to obtain a lower bound on the number of eigenvalues
$\mathcal{N}_{\#} ([\lambda_1, \lambda_2])$, counted {\it without} multiplicity,
that (\ref{nonhammy}) has on a given interval $[\lambda_1, \lambda_2] \subset I$,
$\lambda_1 < \lambda_2$. 
Under our relatively weak assumptions on the dependence of $A (x; \lambda)$
on $\lambda$, it's possible that 
the eigenvalues of (\ref{nonhammy}), as we've defined them, won't comprise a discrete
set on the interval $[\lambda_1, \lambda_2]$. In this case, our convention will 
be to take $\mathcal{N}_{\#} ([\lambda_1, \lambda_2]) = + \infty$, and to view 
our lower bounds on $\mathcal{N}_{\#} ([\lambda_1, \lambda_2])$ as holding by 
convention. We will see that in many cases the value $\lambda_1$ can be taken 
sufficiently negative so that (\ref{nonhammy}) has no eigenvalues below $\lambda_1$,
and in such cases, we will often write 
$\mathcal{N}_{\#} ([\lambda_1, \lambda_2]) = \mathcal{N}_{\#} ((-\infty, \lambda_2])$
to emphasize that it becomes a count of the number of eigenvalues that 
(\ref{nonhammy}) has at or below $\lambda_2$. For stability analyses, we are most 
interested in taking $\lambda_2 = 0$ (since negative eigenvalues correspond with 
instability by our sign conventions), so the count we will be most interested 
in is $\mathcal{N}_{\#} ((-\infty, 0])$. 

In \cite{PW1992}, the authors address a closely related question via the 
Evans function, namely, determining conditions under which evaluation of 
the Evans function along with appropriate derivatives at $\lambda_2 = 0$
can be combined with asymptotic information as $\lambda$ tends toward
$-\infty$ to indicate the existence of at least one eigenvalue on the  
open half-line $(-\infty, 0)$. 
Our analysis puts the results of \cite{PW1992} in a broader context and adds
an additional geometric criterion for the existence of such eigenvalues.  

Our primary tool for this analysis will be a generalization of the Maslov index
introduced in \cite{BCCJM2022}, and for the purposes of this introduction
we will start with a brief, intuitive discussion of this object 
(see Section \ref{maslov-section} for additional details and reference 
\cite{BCCJM2022} for a full development). Precisely, we focus on the hyperplane
setting discussed in Section 3.2 of \cite{BCCJM2022}. 

To begin, for any $n \in \mathbb{N}$ 
we denote by $Gr_n (\mathbb{R}^{2n})$ the Grassmannian comprising
the collection of all $n$-dimensional subspaces of $\mathbb{R}^{2n}$, and we let 
$\mathpzc{g}$ denote an element of $Gr_n (\mathbb{R}^{2n})$. The
space $\mathpzc{g}$ can be
spanned by a choice of $n$ linearly independent vectors in 
$\mathbb{R}^{2n}$, and we will generally find it convenient to collect
these $n$ vectors as the columns of a $2n \times n$ matrix $\mathbf{G}$, 
which we will refer to as a {\it frame} for $\mathpzc{g}$.  
We specify a metric on $Gr_n (\mathbb{R}^{2n})$ in terms of appropriate orthogonal projections. 
Precisely, let $\mathcal{P}_i$ 
denote the orthogonal projection matrix onto $\mathpzc{g}_i \in Gr_n (\mathbb{R}^{2n})$
for $i = 1,2$. I.e., if $\mathbf{G}_i$ denotes a frame for $\mathpzc{g}_i$,
then $\mathcal{P}_i = \mathbf{G}_i (\mathbf{G}_i^* \mathbf{G}_i)^{-1} \mathbf{G}_i^*$.
We take our metric $d$ on $Gr_n (\mathbb{R}^{2n})$ to be defined 
by 
\begin{equation*}
d (\mathpzc{g}_1, \mathpzc{g}_2) := \|\mathcal{P}_1 - \mathcal{P}_2 \|,
\end{equation*} 
where $\| \cdot \|$ can denote any matrix norm. We will say 
that a path of Grassmannian subspaces 
$\mathpzc{g}: [a, b] \to \Lambda (n)$ is continuous provided it is 
continuous under the metric $d$. 

\begin{remark} \label{frames-remark} Here, and throughout, we will
be as consistent as possible with the following notational 
conventions: we will 
express Grassmannian subspaces with script letters such as $\mathpzc{g}$
or $\mathpzc{h}$,
and we will denote a choice of basis elements for $\mathpzc{g}$ 
or $\mathpzc{h}$ respectively
by $\{g_i\}_{i=1}^n$ or $\{h_i\}_{i=1}^n$. 
We will also collect these basis elements 
into associated frames designated with bold capital 
letters,
\begin{equation*}
    \mathbf{G} = (g_1, \,\, g_2, \,\, \dots, \,\, g_n)
    \quad \textrm{or} \quad
    \mathbf{H} = (h_1, \,\, h_2, \,\, \dots, \,\, h_n)
\end{equation*}
\end{remark}

Given a continuous path of Grassmannian subspaces 
$\mathpzc{g}: [a, b] \to Gr_n (\mathbb{R}^{2n})$ 
and a fixed {\it target} space $\mathpzc{q} \in Gr_n (\mathbb{R}^{2n})$,
the generalized Maslov index of \cite{BCCJM2022} 
(under some additional conditions discussed below)
provides a means of counting intersections
between the subspaces $\mathpzc{g} (t)$ and $\mathpzc{q}$
as $t$ increases from $a$ to $b$, counted
with direction, but not with multiplicity. (By multiplicity, we mean the 
dimension of the intersection; direction will be discussed in 
detail in Section \ref{maslov-section}). In order to understand how this
works, we first recall the notion of a kernel for a skew-symmetric
$n$-linear map $\omega$. 

\begin{definition}
For a skew-symmetric $n$-linear map
$\omega: \mathbb{R}^{2n} \times \dots \times \mathbb{R}^{2n} \to \mathbb{R}$
($\mathbb{R}^{2n}$ appearing $n$ times), we define the kernel, $\ker \omega$,
to be the subset of $\mathbb{R}^{2n}$,
\begin{equation*}
    \ker \omega := \{v \in \mathbb{R}^{2n}: \omega (v, v_1, \dots, v_{n-1}) = 0,
    \quad \forall \, v_1, v_2, \dots, v_{n-1} \in \mathbb{R}^{2n}\}.
\end{equation*}
\end{definition}

Given a target space $\mathpzc{q} \in Gr_n (\mathbb{R}^{2n})$, we first identify 
a skew-symmetric $n$-linear map $\omega_1$ so that $\mathpzc{q} = \ker \omega_1$.
For example, if we let $\{q_i\}_{i=1}^n \subset \mathbb{R}^{2n}$ 
denote a basis for $\mathpzc{q}$, then we can set 
\begin{equation*}
    \omega_1 (g_1, \dots, g_n)
    := \det (g_1 \,\, \dots \,\, g_n \,\, q_1 \dots q_n),
\end{equation*}
or if we interpret the elements $\{g_i\}_{i=1}^n$ and 
$\{q_i\}_{i=1}^n$ as 1-forms, 
\begin{equation} \label{what-we-mean}
    g_1 \wedge \dots \wedge g_n \wedge q_1 \wedge \dots \wedge q_n
    = \omega_1 (g_1, \dots, g_n) e_1 \wedge \dots \wedge e_{2n}.
\end{equation}
For notational convenience, we will often write 
\begin{equation*}
    \omega_1 (g_1, \dots, g_n)
    = g_1 \wedge \dots \wedge g_n \wedge q_1 \wedge \dots \wedge q_n
\end{equation*}
when, strictly speaking, we mean (\ref{what-we-mean}). 

Next, we let $\omega_2$ denote any skew-symmetric $n$-linear map for which 
$\ker \omega_2 \ne \mathpzc{q}$, and we set 
\begin{equation*}
    \mathcal{H}_{\omega_i} 
    := \{\mathpzc{g} \in Gr_n (\mathbb{R}^{2n}): \mathpzc{g} \cap \ker \omega_i \ne \{0\}\},
    \quad i = 1, 2.
\end{equation*}
Then according to Definition 1.3 in \cite{BCCJM2022}, the set 
\begin{equation} \label{hyperplane-ma}
    \mathcal{M} := Gr_n (\mathbb{R}^{2n}) \backslash (\mathcal{H}_{\omega_1} \cap \mathcal{H}_{\omega_2})
\end{equation}
is a {\it hyperplane Maslov-Arnold space}.

\begin{definition} \label{invariance-definition}
We say that the flow $t \mapsto \mathpzc{g} (t)$
is {\it invariant} on $[a, b]$ with respect to $\omega_1$
and $\omega_2$ provided the values
\begin{equation*}
    \omega_1 (g_1 (t), \dots, g_n (t))
    \quad \text{and} \quad
    \omega_2 (g_1 (t), \dots, g_n (t))
\end{equation*}
do not simultaneously vanish at any $t \in [a, b]$ (i.e., 
$\mathpzc{g} (t) \in \mathcal{M}$ for all $t \in [a, b]$). For brevity, 
we say that the triple $(\mathpzc{g} (\cdot), \omega_1, \omega_2)$
is invariant on $[a, b]$. (Here, we note that $\mathpzc{q}$ is not
needed in the triple notation, since $\mathpzc{q}$ is determined by $\omega_1$.) 
Likewise, we say that a map 
$\mathpzc{g}: [a,b] \times [c, d] \to Gr_n (\mathbb{R}^{2n})$
is invariant on $[a,b] \times [c, d]$ with respect to $\omega_1$
and $\omega_2$ provided the values 
\begin{equation} \label{quantities}
    \omega_1 (g_1 (s, t), \dots, g_n (s, t))
    \quad \text{and} \quad
    \omega_2 (g_1 (s, t), \dots, g_n (s, t))
\end{equation}
do not simultaneously vanish at any 
$(s, t) \in [a,b] \times [c, d]$ (i.e., 
$\mathpzc{g} (s, t) \in \mathcal{M}$ for all $(s,t) \in [a, b] \times [c, d]$). 
For brevity, we say that the triple $(\mathpzc{g} (\cdot, \cdot), \omega_1, \omega_2)$
is invariant on $[a, b] \times [c, d]$. Finally, we will say 
that a map $\mathpzc{g}: [a,b] \times [c, d] \to Gr_n (\mathbb{R}^{2n})$
is invariant on the boundary of $[a,b] \times [c, d]$ with respect to $\omega_1$
and $\omega_2$ provided the values in (\ref{quantities}) do not 
simultaneously vanish at any point $(s, t)$ on the boundary of 
$[a,b] \times [c, d]$. 
\end{definition}

\begin{remark} The terminology ``invariant" is taken 
from \cite{BCCJM2022}, where it arises naturally as the 
condition that a path in $P (\bigwedge^n (\mathbb{R}^{2n}))$
(i.e., the projective space of all one-dimensional subspaces
of the wedge space $\bigwedge^n (\mathbb{R}^{2n})$) 
associated to the flow $t \mapsto \mathpzc{g} (t)$ lies 
entirely in the {\it Maslov-Arnold space} introduced 
in \cite{BCCJM2022}. While this notion of the Maslov-Arnold
space is critical to the development of \cite{BCCJM2022},
we will only use it indirectly here, and so will omit a precise
definition. 
\end{remark}

In the event that the flow $t \mapsto \mathpzc{g} (t)$
is invariant on $[a, b]$ with respect to $\omega_1$
and $\omega_2$, the generalized Maslov index of \cite{BCCJM2022} can be computed 
as the winding number in projective space $\mathbb{R}P^1$ of the map 
\begin{equation} \label{extended-maslov}
    t \mapsto [\omega_1 (g_1 (t), \dots, g_n (t)) : \omega_2 (g_1 (t), \dots, g_n (t))]
\end{equation}
through $[0 : 1]$ (with appropriate conventions taken for counting arrivals and 
departures; see Section \ref{maslov-section} below). If the Maslov-Arnold 
space is a hyperplane Maslov-Arnold space as in (\ref{hyperplane-ma}) then 
this value is referred to as the {\it hyperplane index}. Since the currently
analysis will be entirely in the hyperplane setting, we will henceforth
refer to such counts as hyperplane indices. 

Following the convention of \cite{BCCJM2022}, we denote 
the hyperplane index by $\ind(\cdots)$, though our specific notation 
is adapted from \cite{HS2018, HS2022}, leading to
$\ind (\mathpzc{g} (\cdot), \mathpzc{q}; [a, b])$; i.e., 
$\ind (\mathpzc{g} (\cdot), \mathpzc{q}; [a, b])$ is a directed count 
of the number of times the subspace $\mathpzc{g} (t)$ has non-trivial 
intersection with $\mathpzc{q}$, counted without multiplicity, as $t$ 
increases from $a$ to $b$. The value $\ind (\mathpzc{g} (\cdot), \mathpzc{q}; [a, b])$
clearly depends on the choices of $\omega_1$ and $\omega_2$, but we 
will generally suppress this in the notation, taking it to generally 
be the case that these choices are clear from context. For 
$\omega_1$, we will always use (\ref{what-we-mean}), but there 
is considerable flexibility in the choice of $\omega_2$. In cases
in which two possibilites for $\omega_2$ are considered, we will 
denote the alternative choice as $\omega_3$, and in order to distinguish the 
resulting hyperplane indices, we will denote them respectively 
$\ind_{\omega_2} (\mathpzc{g} (\cdot), \mathpzc{q}; [a, b])$
and $\ind_{\omega_3} (\mathpzc{g} (\cdot), \mathpzc{q}; [a, b])$.

For many applications, we would like to compute the hyperplane index
associated with a pair of evolving spaces $\mathpzc{g}, \mathpzc{h}: [a, b] \to Gr_n (\mathbb{R}^{2n})$,
or more generally (as in the current setting) a pair of evolving 
spaces $\mathpzc{g}: [a, b] \to Gr_m (\mathbb{R}^{n})$ and 
$\mathpzc{h}: [a, b] \to Gr_{n-m} (\mathbb{R}^{n})$, where
$m \in \{1, 2, \dots, n-1\}$. Following the approach of Section 3.5
in \cite{Furutani2004} (as developed for the current setting 
in \cite{Howard2022a}), we can proceed by letting $\mathbf{G} (t)$
and $\mathbf{H} (t)$ respectively denote frames for $\mathpzc{g} (t)$
and $\mathpzc{h} (t)$, and specifying an evolving 
subspace $\mathpzc{f}: [a, b] \to Gr_{n} (\mathbb{R}^{2n})$ 
with frame 
\begin{equation} \label{F-defined}
    \mathbf{F} (t)
    := 
    \begin{pmatrix}
        \mathbf{G} (t) & \mathbf{0}_{n \times (n - m)} \\
        \mathbf{0}_{n \times m} & \mathbf{H} (t)
    \end{pmatrix}.
\end{equation}
Subsequently, we fix a target space $\tilde{\Delta} \in Gr_n (\mathbb{R}^{2n})$ 
with frame $\mathbf{\tilde{\Delta}} = \genfrac{(}{)}{0pt}{2}{-I_n}{I_n}$,
and we define the hyperplane index for the 
pair $\mathpzc{g}, \mathpzc{h}: [a, b] \to Gr_n (\mathbb{R}^{2n})$
to be 
\begin{equation} \label{extended-pairs}
    \ind (\mathpzc{g} (\cdot), \mathpzc{h} (\cdot); [a, b])
    := \ind (\mathpzc{f} (\cdot), \tilde{\Delta}; [a, b]),
\end{equation}
where the right-hand side is computed precisely as 
described above (i.e., as in \cite{BCCJM2022}), depending as always
on specified choices of $\omega_1$ and $\omega_2$. 

From (\ref{what-we-mean}), $\omega_1$ is taken to be 
\begin{equation} \label{omega1specified-intro}
    \omega_1 (f_1, f_2, \dots, f_n)
    := f_1 \wedge \dots \wedge f_n 
    \wedge \tilde{\delta}_1 \wedge \dots \wedge \tilde{\delta}_n,
\end{equation}
where the vectors $\{\tilde{\delta}_i\}_{i=1}^n$ comprise the columns
of $\tilde{\mathbf{\Delta}} = \genfrac{(}{)}{0pt}{2}{-I_n}{I_n}$. 
In order to take advantage of the additional flexibility 
with $\omega_2$, we proceed by fixing a 
(judiciously chosen) invertible $n \times n$ matrix $M$
and setting 
\begin{equation} \label{omega2specified-intro}
    \omega_2 (f_1, f_2, \dots, f_n)
    := f_1 \wedge \dots \wedge f_n
    \wedge \tilde{\sigma}_1 \wedge \dots \wedge \tilde{\sigma}_n,
\end{equation}
where the vectors $\{\tilde{\sigma}_i\}_{i=1}^n$ comprise the columns
of $\tilde{\mathbf{\Sigma}} = \genfrac{(}{)}{0pt}{2}{-M}{I_n}$. 

Returning to (\ref{nonhammy}), we will show in Section
\ref{ev-problems-section} that under our Assumptions 
{\bf (A)} through {\bf (D)} there exists a unique 
solution $\eta^- (x; \lambda)$ to (\ref{nonhammy})
so that 
\begin{equation} \label{eta-minus}
    \lim_{x \to -\infty} e^{- \mu_- (\lambda) x}
    \eta^- (x; \lambda) = v^- (\lambda),
\end{equation}
where $\mu_- (\lambda)$ and $v^- (\lambda)$ are as described 
in Assumption {\bf (C)}.
We will let $\mathpzc{g} (x; \lambda) \in Gr_1 (\mathbb{R}^{n})$ 
denote the path of Grassmannian subspaces with 
frame $\eta^- (x; \lambda)$. For the space $\mathpzc{h} (x; \lambda)$,
we begin by observing that under 
our Assumptions {\bf (A)} through {\bf (D)} we can take 
a linearly independent collection of solutions to 
(\ref{nonhammy}) $\{y_i^+ (x; \lambda)\}_{i=1}^n$, selected 
so that 
\begin{equation} \label{v-plus-specified}
    \lim_{x \to + \infty} e^{-\mu_+ (\lambda) x} y_n^+ (x; \lambda)
    = v^+ (\lambda),
\end{equation}
where $\mu_+ (\lambda)$ and $v^+ (\lambda)$ are as described 
in Assumption {\bf (C)}. With $y_n^+ (x; \lambda)$ distinguished in this way,
we take $\mathpzc{h} (x; \lambda) \in Gr_{n-1} (\mathbb{R}^{n})$ to be
the space spanned by the collection $\{y_i^+ (x; \lambda)\}_{i=1}^{n-1}$.
In addition, we characterize this space by designating
the wedge product 
\begin{equation} \label{mathcal-Y}
    \mathcal{Y}^+ (x; \lambda)
    := y_1^+ (x; \lambda) \wedge \dots \wedge y_{n-1}^+ (x; \lambda),
\end{equation}
about which we will establish some general notation in the following 
remark. 

\begin{remark} \label{wedge-not-wedge}
In many cases, we will associate a vector $v \in \mathbb{R}^n$ with a 
corresponding element of $\bigwedge^1 (\mathbb{R}^n)$, 
$v = v_1 e_1 + v_2 e_2 + \dots + v_n e_n$, and 
likewise we will associate a vector $\mathcal{V} \in \mathbb{R}^n$
with a corresponding element of $\bigwedge^{n-1} (\mathbb{R}^n)$,
\begin{equation*}
    \mathcal{V}_1 e_1 \wedge \dots \wedge e_{n-1}
    + \mathcal{V}_2 e_1 \wedge \dots \wedge e_{n-2} \wedge e_n
    + \dots + \mathcal{V}_n e_2 \wedge e_3 \wedge \dots \wedge e_n.
\end{equation*}
Though we will generally use the same notation for both the vector and 
its associated form, we will keep the two cases separate by 
consistently using lower case letters for 1-forms and   
capital calligraphic letters for $(n-1)$-forms. 
Throughout our analysis, we will often use the relation
\begin{equation} \label{the-wedge-product}
    v \wedge \mathcal{V}
    = \Big(\sum_{i=1}^n (-1)^{i+1} v_i \mathcal{V}_{n+1-i}\Big) e_1 \wedge \dots \wedge e_n.
\end{equation}
\end{remark}

With $\mathcal{Y}^+ (x; \lambda)$ defined as in 
(\ref{mathcal-Y}), it's straightforward to 
check that 
\begin{equation} \label{wedge-ode}
    \mathcal{Y}^{+\,\prime} (x; \lambda)
    = \Big((\tr A(x; \lambda))I + \tilde{A} (x; \lambda)\Big) \mathcal{Y}^{+},
\end{equation}
where 
\begin{equation} \label{tildeA}
    \tilde{A} (x; \lambda)
    = \begin{pmatrix}
        -a_{nn} & a_{(n-1)n} & - a_{(n-2)n} & \cdots & (-1)^n a_{1n} \\
        a_{n(n-1)} & - a_{(n-1)(n-1)} & a_{(n-2)(n-1)} & \cdots & (-1)^{n-1} a_{1(n-1)} \\
        -a_{n(n-2)} & a_{(n-1)(n-2)} & -a_{(n-2)(n-2)} & \cdots & (-1)^{n-2} a_{2(n-2)} \\
        \vdots & \vdots & \vdots & \vdots & \vdots \\
        (-1)^n a_{n1} & (-1)^{n-1} a_{(n-1)1} & (-1)^{n-2} a_{(n-2)1} & \cdots & - a_{11}
    \end{pmatrix}.
\end{equation}
More succinctly, the elements $\{\tilde{a}_{ij}\}_{i,j=1}^n$ of $\tilde{A}$ are related to the elements 
$\{a_{ij}\}_{i,j=1}^n$ of $A$ by the relation
\begin{equation} \label{AtildeA-relation}
    \tilde{a}_{ij}
    = (-1)^{i+j+1} a_{(n+1-j)(n+1-i)}.
\end{equation}
(See Section \ref{ev-problems-section} below for a straightforward verification.) We set 
\begin{equation} \label{tilde-mathcal-Y-plus-defined}
    \tilde{\mathcal{Y}}^+ (x; \lambda)
    := e^{- \int_0^x \tr A (\xi; \lambda) d\xi} \mathcal{Y}^+ (x; \lambda)
\end{equation}
so that 
\begin{equation} \label{mathcal-Y-system}
    \tilde{\mathcal{Y}}^{+\,\prime}
    = \tilde{A} (x; \lambda) \tilde{\mathcal{Y}}^+,
\end{equation}
and we will see in Proposition \ref{ode-proposition2} below that 
\begin{equation*}
    \lim_{x \to + \infty} e^{\mu_+ (\lambda) x} \tilde{\mathcal{Y}}^+ (x; \lambda)
    = \tilde{\mathcal{V}}^+ (\lambda),
\end{equation*}
where $\mu_+ (\lambda)$ is as in Assumption {\bf (C)} and $\tilde{\mathcal{V}}^+ (\lambda)$
is a right eigenvector (uniquely defined up to 
a scaling constant) of 
\begin{equation*}
\tilde{A}_+ (\lambda) 
:= \lim_{x \to + \infty} \tilde{A} (x; \lambda),
\end{equation*}
associated to the eigenvalue $- \mu_+ (\lambda)$, which is the left-most 
eigenvalue of $\tilde{A}_+ (\lambda)$. 

In order to work with $\omega_2$ as specified in (\ref{omega2specified-intro}), we 
let $M$ denote the matrix introduced for $\omega_2$ and set 
\begin{equation} \label{mathcal-y-m-specified}
    \mathcal{Y}_M^+ (x; \lambda)
    := (M y_1^+ (x; \lambda)) \wedge \dots \wedge (M y_{n-1}^+ (x; \lambda)).
\end{equation}
Then $\mathcal{Y}_M^+ (x; \lambda)$ satisfies the same relations as 
$\mathcal{Y}^+ (x; \lambda)$, except with $A (x; \lambda)$ replaced
everywhere with 
\begin{equation*}
    \mathcal{A} (x; \lambda)
    := M A(x; \lambda) M^{-1}.
\end{equation*}
In particular, if we set 
\begin{equation*}
    \tilde{\mathcal{Y}}_M^+ (x; \lambda)
    := e^{- \int_0^x \tr \mathcal{A} (\xi; \lambda) d\xi}
    \mathcal{Y}_M^+ (x; \lambda),
\end{equation*}
then 
\begin{equation*}
    \tilde{\mathcal{Y}}_M^{+\,\prime}
    = \tilde{\mathcal{A}} (x; \lambda) \tilde{\mathcal{Y}}_M^+,
\end{equation*}
where $\tilde{\mathcal{A}} (x; \lambda)$ is related to $\mathcal{A} (x; \lambda)$
in the same way that $\tilde{A} (x; \lambda)$ is related to $A (x; \lambda)$
(via (\ref{AtildeA-relation})). According to Proposition \ref{ode-proposition2}
below, we have  
\begin{equation*}
    \lim_{x \to + \infty} e^{\mu_+ (\lambda) x} \tilde{\mathcal{Y}}_M^+ (x; \lambda)
    = \tilde{\mathcal{V}}_M^+ (\lambda),
\end{equation*}
where $\tilde{\mathcal{V}}_M^+ (\lambda)$ a right eigenvector (uniquely defined up to 
a scaling constant) of 
\begin{equation} \label{tilde-mathcal-A-asymptotic}
    \tilde{\mathcal{A}}_+ (\lambda)
    := \lim_{x \to + \infty} \tilde{\mathcal{A}} (x; \lambda),
\end{equation}
associated to the eigenvalue $- \mu_+ (\lambda)$. 

In our general notation, we now take $\mathbf{G} (x; \lambda) = \eta^- (x; \lambda)$
and 
\begin{equation*}
    \mathbf{H} (x; \lambda)
    = \begin{pmatrix}
    y_1^+ (x; \lambda) & y_2^+ (x; \lambda) & \dots & y_{n-1}^+ (x; \lambda)
    \end{pmatrix},
\end{equation*}
and for some value $c > 0$ to be chosen sufficiently large during
the analysis, we set 
\begin{equation*}
    \mathbf{F}^{c} (x; \lambda)
    := \begin{pmatrix}
        \mathbf{G} (x; \lambda) & 0_{n \times (n-1)} \\
        0_{n \times 1} & \mathbf{H} (c; \lambda)
    \end{pmatrix}.
\end{equation*}
With $\omega_1$ and $\omega_2$ as in (\ref{omega1specified-intro})
and (\ref{omega2specified-intro}), we correspondingly set 
\begin{equation} \label{tilde-omega1-defined-c}
    \tilde{\omega}_1^{c} (x; \lambda)
    := \omega_1 (f_1^c (x; \lambda), \dots, f_n^c (x; \lambda))
    = \eta^- (x; \lambda) \wedge \mathcal{Y}^+ (c; \lambda)
\end{equation}
and 
\begin{equation} \label{tilde-omega2-defined-c}
    \tilde{\omega}_2^{c} (x; \lambda)
     := \omega_2 (f_1^c (x; \lambda), \dots, f_n^c (x; \lambda))
    = \eta^- (x; \lambda) \wedge \mathcal{Y}_M^+ (c; \lambda),
\end{equation}
where the vector-functions $\{f_j^c (x; \lambda)\}_{j=1}^n$ comprise the columns of 
$\mathbf{F}^c (x; \lambda)$, and $\mathcal{Y}^+$ and $\mathcal{Y}_M^+$
are respectively defined in (\ref{mathcal-Y}) and (\ref{mathcal-y-m-specified}).
Associated with $\tilde{\omega}_1^{c} (x; \lambda)$ and $\tilde{\omega}_2^{c} (x; \lambda)$,
we define the corresponding normalized functions 
\begin{equation} \label{psi1-2-tilde-c-defined}
    \tilde{\psi}_i^c (x; \lambda)
    = \frac{\tilde{\omega}_i^c (x; \lambda)}{|\eta^- (x; \lambda)||\mathcal{Y}^+ (c; \lambda)|},
    \quad i = 1, 2.
\end{equation}
Here, for brevity, the notation of (\ref{tilde-omega1-defined-c}) and (\ref{tilde-omega2-defined-c})
takes slight liberties with interpretations of the left and right sides. This convention
is discussed in detail in Section \ref{maslov-section}.

We will be interested separately in limits as $c$ tends to 
$+ \infty$ and as $x$ tends to either $- \infty$ or $+\infty$, prompting 
the following notational conventions. First, for the limit as 
$c$ tends to $+ \infty$, we observe that we can write 
\begin{equation} \label{psi1-tilde-c-defined}
    \tilde{\psi}_1^c (x; \lambda)
    = \frac{\eta^- (x; \lambda) \wedge (e^{\mu_+ (\lambda) c}\tilde{\mathcal{Y}}^+ (c; \lambda))}
    {|\eta^- (x; \lambda)||e^{\mu_+ (\lambda) c} \tilde{\mathcal{Y}}^+ (c; \lambda)|},
\end{equation}
from which it's clear that 
\begin{equation} \label{psi1-tilde-plus-defined}
    \tilde{\psi}_1^+ (x; \lambda)
    := \lim_{c \to + \infty} \tilde{\psi}_1^c (x; \lambda)
    = \frac{\tilde{\omega}_1^+ (x; \lambda)}{|\eta^- (x; \lambda)||\tilde{\mathcal{V}}^+ (\lambda)|},
\end{equation}
where we've set 
\begin{equation} \label{omega1-tilde-defined}
    \tilde{\omega}_1^+ (x; \lambda)
    := \eta^- (x; \lambda) \wedge \tilde{\mathcal{V}}^+ (\lambda).
\end{equation}
Likewise, in precisely the same way, we see that 
\begin{equation} \label{psi2-tilde-plus-defined}
    \tilde{\psi}_2^+ (x; \lambda)
    := \lim_{c \to + \infty} \tilde{\psi}_2^c (x; \lambda)
    = \frac{\tilde{\omega}_2^+ (x; \lambda)}{|\eta^- (x; \lambda)||\tilde{\mathcal{V}}^+ (\lambda)|},
\end{equation}
where we've set 
\begin{equation} \label{omega2-tilde-defined}
    \tilde{\omega}_2^+ (x; \lambda)
    := \eta^- (x; \lambda) \wedge \tilde{\mathcal{V}}_M^+ (\lambda).
\end{equation}
Next, for $\tilde{\psi}_1^+ (x; \lambda)$, we can write 
\begin{equation*}
 \tilde{\psi}_1^+ (x; \lambda)
    = \frac{e^{- \mu_- (\lambda) x}\eta^- (x; \lambda) \wedge \tilde{\mathcal{V}}^+ (\lambda)}
    {|e^{- \mu_- (\lambda) x} \eta^- (x; \lambda)||\tilde{\mathcal{V}}^+ (\lambda)|},   
\end{equation*}
from which it's clear that 
\begin{equation} \label{psi1-tilde-plus-minus-defined}
    \tilde{\psi}_1^{+, -} (\lambda)
    := \lim_{x \to - \infty} \tilde{\psi}_1^+ (x; \lambda)
    = \frac{v^- (\lambda) \wedge \tilde{\mathcal{V}}^+ (\lambda)}{|v^- (\lambda)| |\tilde{\mathcal{V}}^+ (\lambda)|},
\end{equation}
and similarly 
\begin{equation} \label{psi2-tilde-plus-minus-defined}
    \tilde{\psi}_2^{+, -} (\lambda)
    := \lim_{x \to - \infty} \tilde{\psi}_2^+ (x; \lambda)
    = \frac{v^- (\lambda) \wedge \tilde{\mathcal{V}}_M^+ (\lambda)}{|v^- (\lambda)| |\tilde{\mathcal{V}}^+ (\lambda)|}.
\end{equation}

Under our most general assumptions, the limits of $\tilde{\psi}_1^+ (x; \lambda)$
and $\tilde{\psi}_2^+ (x; \lambda)$
as $x$ tends to $+ \infty$ don't necessarily exist, but in cases for 
which they do we will designate them respectively as $\tilde{\psi}_1^{+, +} (\lambda)$
and $\tilde{\psi}_2^{+, +} (\lambda)$.

With this notation in place, we are able to state the final set of assumptions
that will be needed for our main theorem. These assumptions, which primarily 
address invariance, are somewhat technical, and so we immediately follow 
the statement by a lengthy remark addressing both how they should 
be interpreted and how they can be verified. 

\medskip
\noindent
{\bf (E)} Suppose the following conditions hold for 
$[\lambda_1, \lambda_2] \subset I$, $\lambda_1 < \lambda_2$: (1)
for $i = 1, 2$, the limits 
\begin{equation*}
    \tilde{\psi}_1^{+, +} (\lambda_i) 
    := \lim_{x \to + \infty} \tilde{\psi}_1^{+} (x; \lambda_i)
    \quad \textrm{and} \quad
     \tilde{\psi}_2^{+, +} (\lambda_i) 
    := \lim_{x \to + \infty} \tilde{\psi}_2^{+} (x; \lambda_i)
\end{equation*}
are well defined; (2) for $i = 1, 2$, the values $\tilde{\psi}_1^+ (x; \lambda_i)$
and $\tilde{\psi}_2^+ (x; \lambda_i)$ do not simultaneously vanish 
at any $x \in \mathbb{R}$, and the limit functions $\tilde{\psi}_1^{+, +} (\lambda_i)$
and $\tilde{\psi}_2^{+, +} (\lambda_i)$ do not simultaneously vanish; 
(3) the values $\tilde{\psi}_1^{+,-} (\lambda)$ 
and $\tilde{\psi}_2^{+,-} (\lambda)$
do not simultaneously vanish at any $\lambda \in [\lambda_1, \lambda_2]$;
and (4) there exists a constant $c_0 > 0$ sufficiently large so that 
for all $c \ge c_0$ and all $\lambda \in [\lambda_1, \lambda_2]$, 
the values $\tilde{\psi}_1^c (c; \lambda)$
and $\tilde{\psi}_2^c (c; \lambda)$ do not simultaneously vanish. 

\begin{remark} \label{assumption-E-remark}
We will show in Section \ref{proof-section}
that under our Assumption {\bf (C)}, {\bf (E)}(1)
holds for any $\lambda_i$ that is not an eigenvalue 
of (\ref{nonhammy}), while if $\lambda_i$ is 
an eigenvalue of (\ref{nonhammy}), the validity 
of {\bf (E)}(1) will be determined by the behavior 
of its associated eigenfunction in the limit as 
$x$ tends toward $+ \infty$. (For the applications
we have in mind, $\lambda_1$ will be taken sufficiently
negative so that it is not an eigenvalue, and 
$\lambda_2$ will be taken to be $0$, which will 
be an eigenvalue.) For {\bf (E)}(2), we can often 
show that $\lambda_1$ can be taken sufficiently 
negative so that $\psi_1^+ (x; \lambda_1) \ne 0$ for 
all $x \in \mathbb{R}$, giving half of the condition.
On the other hand, we will find that 
$\psi_1^+ (x; \lambda_2)$ and $\psi_2^+ (x; \lambda_2)$
can often be evaluated exactly for all $x \in \mathbb{R}$,
allowing us to check the second half. For {\bf (E)}(3),
we will see that in many important applications, 
including the ones we consider in 
Section \ref{applications-section} of the current analysis, 
we can show that $\psi_1^{+,-} (\lambda) \ne 0$
for all $\lambda \in [\lambda_1, \lambda_2]$, which 
is more than we need. 
The most challenging assumption to verify is {\bf (E)}(4), because
it's the nature of the method to find indirect information about
the spectrum of (\ref{nonhammy}) without directly computing 
the values of 
$\psi_1^{+} (x; \lambda)$ and $\psi_2^{+} (x; \lambda)$ as 
$\lambda$ increases from $\lambda_1$ to $\lambda_2$. 
However, in practice we make the following important observation. 
If (\ref{nonhammy}) has no eigenvalues on an interval $[\lambda_1, \lambda_2]$,
then we can choose $c$ sufficiently large so that $\psi_1^{+} (c; \lambda) \ne 0$
for all $\lambda \in [\lambda_1, \lambda_2]$. In this way, we have a 
useful dichotomy: if Assumption {\bf (E)}(4) fails to hold, then we can 
conclude that (\ref{nonhammy}) certainly has at least one eigenvalue 
on the interval $[\lambda_1, \lambda_2]$. I.e., if an eigenvalue
is detected on $[\lambda_1, \lambda_2]$ under Assumption {\bf (E)}(4), 
and the other parts of Assumptions {\bf (E)} are shown to hold, then  
either there is an eigenvalue on $[\lambda_1, \lambda_2]$ because 
Assumptions {\bf (E)} hold, or there is an eigenvalue on $[\lambda_1, \lambda_2]$ 
because Assumption {\bf (E)}(4) fails to hold. In either case 
we can rigorously conclude the existence of an eigenvalue on the 
interval $[\lambda_1, \lambda_2]$. 
\end{remark}

We will see in Section \ref{proof-section} that under 
Assumptions {\bf (A)} through {\bf (E)}, there exists a constant $C > 0$ sufficiently
large so that for all $c \ge C$ the sum of hyperplane indices 
\begin{equation*} \label{sum-of-indices}
\begin{aligned}
\ind& (\mathpzc{g} (-c; \cdot), \mathpzc{h} (c; \cdot); [\lambda_1, \lambda_2])
+ \ind (\mathpzc{g} (\cdot; \lambda_2), \mathpzc{h} (c; \lambda_2); [-c, c]) \\
& \quad - \ind (\mathpzc{g} (c; \cdot), \mathpzc{h} (c; \cdot); [\lambda_1, \lambda_2])
- \ind (\mathpzc{g} (\cdot; \lambda_1), \mathpzc{h} (c; \lambda_1); [-c, c])
\end{aligned}
\end{equation*}
remains constant. We will denote this constant $\mathfrak{m}$ and 
refer to it as the {\it boundary invariant}. 

In what follows, we will define the notation 
\begin{equation} \label{asymptotic-index}
    \ind (\mathpzc{g} (\cdot; \lambda), \mathpzc{h}^+ (\lambda); [-\infty, +\infty])
\end{equation}
to mean the winding number in projective space $\mathbb{R}P^1$ of the map 
\begin{equation*}
    x \mapsto [\tilde{\psi}_1^+ (x; \lambda):\tilde{\psi}_2^+ (x; \lambda)]
\end{equation*}
as $x$ increases from $-\infty$ to $+\infty$, including a possible 
departure associated with the asymptotic limit on the left and a possible 
arrival associated with the asymptotic limit on the right. Likewise, 
we denote by 
\begin{equation} \label{asymptotic-bottom-shelf}
    \ind (\mathpzc{g}^- (\cdot), \mathpzc{h}^+ (\cdot); [\lambda_1, \lambda_2])
\end{equation}
the winding number in projective space $\mathbb{R}P^1$ of the map 
\begin{equation*}
    \lambda \mapsto [\tilde{\psi}_1^{+,-} (\lambda):\tilde{\psi}_2^{+,-} (\lambda)]
\end{equation*}
as $\lambda$ increases from $\lambda_1$ to $\lambda_2$. 

Regarding (\ref{asymptotic-index}), this value is computed by 
tracking the rotation of a point $p^+ (x; \lambda)$ around 
$S^1$ as $x$ increases from $- \infty$ to $+ \infty$. Under our 
assumptions, the limits 
\begin{equation*}
    p^{+,-} (\lambda_i)
    := \lim_{x \to - \infty} p^+ (x; \lambda_i)
    \quad \textrm{and} \quad
    p^{+,+} (\lambda_i)
    := \lim_{x \to + \infty} p^+ (x; \lambda_i)
\end{equation*}
both exist. If $p^{+,\pm} (\lambda_i) \ne (-1,0)$,
then there is no asymptotic crossing point at the 
associated side, and the hyperplane index can be 
computed as usual on that side. In the event that 
$(-1,0)$ is achieved as one of these asymptotic 
limits, the situation is slightly more complicated.
As a specific case, suppose 
\begin{equation*}
     p^{+,+} (\lambda_1) = (-1,0).
\end{equation*}
It may be the case that as $x$ increases toward $+\infty$ 
the point $p^+(x; \lambda_1)$ crosses $(-1,0)$ an infinite 
number of times, so that no true crossing count is valid.
Nonetheless, since the limit 
\begin{equation*}
    \lim_{x \to + \infty} p^+ (x; \lambda_1)
\end{equation*}
is well defined, we can {\it define} 
\begin{equation*}
\ind (\mathpzc{g} (\cdot; \lambda_1), \mathpzc{h}^+ (\lambda_1); [0, +\infty])    
\end{equation*}
in the following way. (The restriction to one infinite endpoint is 
simply to allow us to focus on a single side; the case of $-\infty$ 
is treated similarly.) Given any $\epsilon > 0$, there exists some value 
$L$ sufficiently large so that
\begin{equation} \label{condition-star}
    |p^+ (x; \lambda_1) - (-1,0)| < \epsilon
\end{equation}
for all $x \ge L$. There are three possibilities for the location
of $p^+ (L; \lambda_1)$: (a) $p^+ (L; \lambda_1)$ is a small distance from 
$(-1,0)$ in the clockwise direction; or (b) $p^+ (L; \lambda_1) = (-1,0)$; 
or (c) $p^+ (L; \lambda_1)$ is a small distance from $(-1,0)$ in the 
counterclockwise direction. We define 
\begin{equation*}
    \ind (\mathpzc{g} (\cdot; \lambda_1), \mathpzc{h}^+ (\lambda_1); [0, + \infty])    
    := \ind (\mathpzc{g} (\cdot; \lambda_1), \mathpzc{h}^+ (\lambda_1); [0, L])
    + \begin{cases}
    1 & \textrm{case (a)} \\
    0 & \textrm{cases (b) and (c)}. 
    \end{cases}
\end{equation*}
Here, we emphasize that this definition does not depend on the particular 
choice of $L$, only on (\ref{condition-star}). 

We are now in a position to state the main theorem of the analysis. 

\begin{theorem} \label{main-theorem}
For (\ref{nonhammy}), let Assumptions {\bf (A)} through {\bf (D)} hold, 
and for some fixed interval $[\lambda_1, \lambda_2] \subset I$, $\lambda_1 < \lambda_2$,
suppose Assumption {\bf (E)} holds as well. If $\mathcal{N}_{\#} ([\lambda_1, \lambda_2])$
denotes the number of eigenvalues that (\ref{nonhammy}) has on the interval 
$[\lambda_1, \lambda_2]$, counted without multiplicity, then 
\begin{equation}
\begin{aligned}
\mathcal{N}_{\#} ([\lambda_1, \lambda_2])
&\ge 
\Big|\ind (\mathpzc{g} (\cdot; \lambda_2), \mathpzc{h}^+ (\lambda_2); [- \infty, + \infty])
- \ind (\mathpzc{g} (\cdot; \lambda_1), \mathpzc{h}^+ (\lambda_1); [- \infty, + \infty]) \\
& \quad \quad + \ind (\mathpzc{g}^- (\cdot), \mathpzc{h}^+ (\cdot); [\lambda_1, \lambda_2]) - \mathfrak{m} \Big|.
\end{aligned}
\end{equation}
\end{theorem}

In the remainder of this introduction, we provide some background
and context for our analysis and also set out a plan for the 
paper. For the former, our analysis is motivated by oscillation 
results for linear Hamiltonian systems, which have their origins
in the classical work of Sturm and Morse, respectively 
\cite{Sturm1836} and \cite{Morse1934}. As discussed at length in 
\cite{Howard2022b}, numerous authors have contributed to 
the development and application of such results, and the theory 
for linear Hamiltonian systems has become well established
(see, for example, \cite{CH2007, HLS2018, Howard2022b} for development of
the general theory, \cite{BCJLMS2018, BJ1995, CJ2018, J1988a, J1988b, JM2012}
for applications, and \cite{BM2022, BM2013, CDB2009, CDB2011, Chardard2009}
for associated numerical calculations. 

Such results have all been limited either to linear Hamiltonian 
systems or systems with underlying Hamiltonian structure, but 
the recent result \cite{BCCJM2022} provides a tool applicable
in fully non-Hamiltonian settings such as those considered 
here. In \cite{BCCJM2022}, the authors employed their generalized
Maslov index to obtain an oscillation result for non-Hamiltonian 
systems on a bounded domain, and the current analysis seems to 
be the first effort to obtain such oscillation results for 
a class of non-Hamiltonian systems on $\mathbb{R}$. 

The paper is organized as follows. In Section \ref{maslov-section},
we review elements of the hyperplane index that will be used 
in our development, and in Section \ref{ev-problems-section} 
we summarize some general results on solutions of (\ref{nonhammy})
that will be necessary for the proof of Theorem \ref{main-theorem}.
In Section \ref{proof-section}, we prove Theorem \ref{main-theorem},
and in Section \ref{evans-section} we discuss the role of the 
Evans function in the current setting. Finally, in Section 
\ref{applications-section} we provide two illustrative 
applications, first to the generalized KdV equation, and 
second to the KdV-Burgers equation.

\section{Properties of the Hyperplane Index} 
\label{maslov-section}

In this section, we emphasize properties of the hyperplane 
index that will have a role in our analysis, leaving  
a full development of the theory to \cite{BCCJM2022}. 
In particular, a proper discussion of this object requires
some items from algebraic topology that are (1) already 
covered clearly and concisely in \cite{BCCJM2022}; 
and (2) not critical to the development of our results. 
Aside from an occasional clarifying comment for interested 
readers, these items are omitted from the current 
discussion. 

As in the introduction, we let $\mathpzc{g}: [a, b] \to Gr_n (\mathbb{R}^{2n})$ 
denote a continuous path of Grassmannian subspaces, and we let 
$\mathpzc{q} \in Gr_n (\mathbb{R}^{2n})$ denote a fixed {\it target}
subspace. We let $\omega_1$ denote a skew-symmetric $n$-linear map
such that $\ker \omega_1 = \mathpzc{q}$, and we let $\omega_2$ denote
a second skew-symmetric $n$-linear map so that the triple
$(\mathpzc{g} (\cdot), \omega_1, \omega_2)$ satisfies 
the invariance property described in Definition \ref{invariance-definition}
on the interval $[a, b]$ (i.e., $\mathpzc{g} (t) \in \mathcal{M}$
for all $t \in [a, b]$, where $\mathcal{M}$ is as in (\ref{hyperplane-ma})). 
Recalling that our notational convention is to fix a choice of 
frames $\mathbf{G} (t)$ for $\mathpzc{g} (t)$ with columns 
$\{g_i (t)\}_{i=1}^n$, we set 
\begin{equation} \label{tilde-omega-again}
    \tilde{\omega}_i (t) 
    := \omega_i (g_1 (t), g_2 (t), \dots, g_n (t)),
    \quad i = 1, 2.
\end{equation}
I.e., $\omega_i$ will consistently denote a skew-symmetric $n$-linear map,
and $\tilde{\omega}_i$ will consistently denote the evaluation of 
$\omega_i$ along a particular path mapping $[a, b]$ to $Gr_n (\mathbb{R}^{2n})$.

The hyperplane index $\ind (\mathpzc{g} (\cdot), \mathpzc{q}; [a, b])$
is then computed as described in (\ref{extended-maslov}), with appropriate
conventions for counting arrivals and departures to and from the point 
in projective space $[0:1]$ (described below). In practice, we proceed
by tracking a point $p(t) \in S^1$, which can be precisely specified as 
\begin{equation} \label{p-specified}
p(t) := 
\begin{cases}
\Big( \frac{\tilde{\omega}_2 (t)}{\sqrt{\tilde{\omega}_1 (t)^2 + \tilde{\omega}_2 (t)^2}}, 
  \frac{\tilde{\omega}_1 (t)}{\sqrt{\tilde{\omega}_1 (t)^2 + \tilde{\omega}_2 (t)^2}} \Big) & \tilde{\omega}_2 (t) \le 0 \\
- \Big( \frac{\tilde{\omega}_2 (t)}{\sqrt{\tilde{\omega}_1 (t)^2 + \tilde{\omega}_2 (t)^2}}, 
   \frac{\tilde{\omega}_1 (t)}{\sqrt{\tilde{\omega}_1 (t)^2 + \tilde{\omega}_2 (t)^2}} \Big)  & \tilde{\omega}_2 (t) > 0,
\end{cases}
\end{equation}
or equivalently with $\tilde{\omega}_1 (t)$ and $\tilde{\omega}_2 (t)$
replaced by the scaled variables 
\begin{equation} \label{scaled-variables}
    \tilde{\psi}_1 (t)
    = \frac{\tilde{\omega}_1 (t)}{|g_1 (t) \wedge g_2 (t) \wedge \dots \wedge g_n (t)|}
    \quad 
     \tilde{\psi}_2 (t)
    = \frac{\tilde{\omega}_2 (t)}{|g_1 (t) \wedge g_2 (t) \wedge \dots \wedge g_n (t)|}.
\end{equation}
In the usual way, we think of mapping $\mathbb{R}P^1$ to the left half
of the unit circle and then closing to $S^1$ by equating the points 
$(0,1)$ and $(0,-1)$. It's clear that $t_*$ is a {\it crossing point}
of the flow (i.e., a point so that $\mathpzc{g} (t_*) \cap \mathpzc{q} \ne \{0\}$)
if and only if $p (t_*) = (-1, 0)$, so the hyperplane index is computed 
as a count of the number of times the point $p(t)$ crosses $(-1, 0)$. 
We take crossings in the clockwise direction to be negative and crossings
in the counterclockwise direction to be positive. Regarding
behavior at the endpoints, if $p(t)$
rotates away from $(-1,0)$ in the clockwise direction as $t$ increases
from $0$, then the hyperplane index decrements by 1, while if $p(t)$ 
rotates away from $(-1,0)$ in the counterclockwise direction as $t$ increases
from $0$, then the hyperplane index does not change. Likewise, 
if $p(t)$ rotates into $(-1,0)$ in the 
counterclockwise direction as $t$ increases
to $1$, then the hyperplane index increments by 1, while if 
$p(t)$ rotates into $(-1,0)$ in the clockwise direction as $t$ increases
to $1$, then the hyperplane index does not change. Finally, 
it's possible that $p(t)$ will arrive 
at $(-1,0)$ for $t = t_*$ and remain at $(-1,0)$ as $t$ traverses
an interval. In these cases, the hyperplane index only increments/decrements upon arrival or 
departure, and the increments/decrements are determined 
as for the endpoints (departures determined as with $t=0$,
arrivals determined as with $t = 1$). 

\begin{remark} \label{bccjm-p}
In \cite{BCCJM2022}, the authors view $S^1$ as a circle in $\mathbb{C}$, 
and make the specification
\begin{equation*}
    \tilde{p}(t) 
    := \Big(\frac{\tilde{\omega}_1 (t) - i \tilde{\omega}_2 (t)}{|\tilde{\omega}_1 (t) - i \tilde{\omega}_2 (t)|}\Big)^2.
\end{equation*}
This choice leads to precisely the same dynamics as those described above, and 
in particular to the same values of the hyperplane index. 
\end{remark}

We emphasize, as in the introduction, that in contrast with the
Maslov index in the setting of Lagrangian flow, 
the hyperplane index does not keep track of the dimensions 
of the intersections. 

To set some notation, we let $\omega_1$ and $\omega_2$ be as 
above, and denote 
by $\mathcal{P}_{\omega_1, \omega_2} ([a, b])$ 
the collection of all continuous paths $\mathpzc{g}: [a, b] \to \mathcal{M}$, 
with $\mathcal{M}$ as in (\ref{hyperplane-ma}). (I.e., 
$\mathcal{P}_{\omega_1, \omega_2} ([a, b])$ comprises 
the collection of all continuous paths $\mathpzc{g}: [a, b] \to Gr_n (\mathbb{R}^{2n})$ 
that are invariant with respect to the skew-symmetric $n$-linear maps
$\omega_1$ and $\omega_2$.) The hyperplane index of \cite{BCCJM2022} has the following
properties (see Proposition 3.8 in \cite{BCCJM2022}). 

\medskip
\noindent
{\bf (P1)} (Path Additivity) If $\mathpzc{g} \in \mathcal{P}_{\omega_1, \omega_2} ([a, b])$
and $\mathpzc{q} = \ker \omega_1$, then for any $\tilde{a}, \tilde{b}, \tilde{c} \in [a, b]$, 
with $\tilde{a} < \tilde{b} < \tilde{c}$, we have
\begin{equation*}
\ind (\mathpzc{g} (\cdot), \mathpzc{q}; [\tilde{a}, \tilde{c}]) 
= \ind (\mathpzc{g} (\cdot), \mathpzc{q};[\tilde{a}, \tilde{b}]) 
+ \ind (\mathpzc{g} (\cdot), \mathpzc{q}; [\tilde{b}, \tilde{c}]).
\end{equation*}

\medskip
\noindent
{\bf (P2)} (Homotopy Invariance) If 
$\mathpzc{g}, \mathpzc{h} \in \mathcal{P}_{\omega_1, \omega_2} ([a, b])$ 
are homotopic in $\mathcal{M}$ with $\mathpzc{g} (a) = \mathpzc{h} (a)$ and  
$\mathpzc{g} (b) = \mathpzc{h} (b)$ (i.e., if $\mathpzc{g}, \mathpzc{h}$
are homotopic with fixed endpoints) then 
\begin{equation*}
\ind (\mathpzc{g} (\cdot), \mathpzc{q}; [a, b]) = \ind (\mathpzc{h} (\cdot), \mathpzc{q}; [a, b]).
\end{equation*} 

\subsection{Grassmannian Pairs}
\label{grassmanian-pairs-section}

In this section, we clarify both the approach and 
notation from the introduction by providing a general 
development for computing the hyperplane index for evolving
pairs of Grassmannian spaces 
$\mathpzc{g}: [a, b] \to Gr_m (\mathbb{R}^{n})$ and 
$\mathpzc{h}: [a, b] \to Gr_{n-m} (\mathbb{R}^{n})$, where
$m \in \{1, 2, \dots, n-1\}$. In order to facilitate 
such calculations, we can think of letting 
$\mathbf{F} (t)$ denote the matrix function specified 
in (\ref{F-defined}), and taking $\omega_1$ and $\omega_2$
respectively 
as in (\ref{omega1specified-intro}) and (\ref{omega2specified-intro}).
With these choices in place, we obtain the relation  
\begin{equation*}
\begin{aligned}
    \omega_1 (f_1, f_2, \dots, f_n)
    &= \begin{pmatrix} g_1 \\ 0 \end{pmatrix}
    \wedge \begin{pmatrix} 0 \\ h_1 \end{pmatrix}
    \wedge \dots \wedge \begin{pmatrix} 0 \\ h_{n-1} \end{pmatrix}
    \wedge \tilde{\delta}_1 \wedge \dots \wedge \tilde{\delta}_n \\
    &= g_1 \wedge h_1 \wedge \dots \wedge h_{n-1} \wedge e_{n+1} \wedge \dots \wedge e_{2n},
\end{aligned}    
\end{equation*}
and this prompts us to set 
\begin{equation} \label{omega1-tilde-defined-wedge-formulation}
    \tilde{\omega}_1 (t)
    := g_1 (t) \wedge h_1 (t) \wedge \dots \wedge h_{n-1} (t).
\end{equation}
Likewise, we'll set 
\begin{equation*}
\begin{aligned}
     \omega_2 (f_1, \dots, f_n) 
     &= \begin{pmatrix} g_1 \\ 0 \end{pmatrix}
    \wedge \begin{pmatrix} 0 \\ h_1 \end{pmatrix}
    \wedge \dots \wedge \begin{pmatrix} 0 \\ h_{n-1} \end{pmatrix}
    \wedge \tilde{\sigma}_1 \wedge \dots \wedge \tilde{\sigma}_n \\ 
    &= g_1 \wedge (M h_1) \wedge \dots \wedge (Mh_{n-1}) e_{n+1} \dots \wedge e_{2n},
\end{aligned}
\end{equation*}
and also 
\begin{equation*}
    \tilde{\omega}_2 (t) 
    := g_1 (t) \wedge (M h_1 (t)) \wedge \dots \wedge (Mh_{n-1} (t)).  
\end{equation*}

Having specified $\omega_1$ and $\omega_2$, it's now also convenient to 
introduce normalized variables 
\begin{equation} \label{psi-i-defined}
\psi_i (f_1, \dots, f_n) := \frac{\omega_i (f_1, \dots, f_n)}{|f_1 \wedge \dots \wedge f_n|},
\quad i = 1, 2, 
\end{equation}
and likewise 
\begin{equation} \label{tilde-psi-i-defined}
\begin{aligned}
    \tilde{\psi}_i (t) &:= \frac{\tilde{\omega}_i (t)}
{\Big| \begin{pmatrix} g_1 (t) \\ 0 \end{pmatrix} \wedge \begin{pmatrix} 0 \\ h_1 (t) \end{pmatrix}
\wedge \dots \wedge \begin{pmatrix} 0 \\ h_{n-1} (t) \end{pmatrix} \Big|} \\
&= \frac{\tilde{\omega}_i (t)}
{|g_1 (t)| |h_1 (t) \wedge \dots \wedge h_{n-1} (t)|},
\quad i = 1, 2.
\end{aligned}
\end{equation}
It's clear from the relationship between $\{\tilde{\omega}_i\}_{i=1}^2$ 
and $\{\tilde{\psi}_i\}_{i=1}^2$ that if we replace $\{\tilde{\omega}_i\}_{i=1}^2$ 
with $\{\tilde{\psi}_i\}_{i=1}^2$ in our expression (\ref{p-specified}) for the tracking 
point $p(t)$, the value of $p(t)$ isn't changed.

\subsection{Direction of Rotation}
\label{direction-section}

We can employ the approach of Section 4 in \cite{BCCJM2022} to locally analyze
the direction associated with a given crossing point. For this, 
our starting point is the observation that for $p(t)$ near $(-1,0)$,
the location of $p(t)$ can be tracked via the angle 
\begin{equation} \label{p-angle}
    \theta(t) = \pi + \tan^{-1} \frac{\tilde{\omega}_1 (t)}{\tilde{\omega}_2 (t)},
\end{equation}
with $\pi$ arising from our convention of placing crossings at $(-1,0)$.
By the monotonicity of $\tan^{-1} x$, the direction of $\theta(t)$ near
a value $t=t_*$ for which $\theta (t_*) = \pi$ is determined by 
the derivative of the ratio $r(t) = \frac{\tilde{\omega}_1 (t)}{\tilde{\omega}_2 (t)}$,
for which $r(t_*) = 0$. Precisely, if 
$r'(t_*) = \frac{\tilde{\omega}'_1 (t_*)}{\tilde{\omega}_2 (t_*)} < 0$
then the rotation of $p(t)$ is clockwise at $t_*$, corresponding with a decrement 
of the hyperplane index, while if $r'(t_*) > 0$ then 
the rotation is counterclockwise, corresponding with an increment of 
the hyperplane index.

\subsection{The Boundary Invariant $\mathfrak{m}$}
\label{boundary-invariant-section}

One of the most challenging aspects of working with the hyperplane 
index is evaluating the boundary invariant $\mathfrak{m}$. One 
strategy, introduced in \cite{BCCJM2022}, is to set 
\begin{equation}
    \rho (t) = \frac{1}{2} (\tilde{\psi}_1 (t)^2 + \tilde{\psi}_2 (t)^2), 
\end{equation}
and show directly that $\rho (t) \ne 0$ for all $t \in [a, b]$. This 
has been shown to work in certain cases in both \cite{BCCJM2022} 
and \cite{Howard2022a}, but in both of those analyses critical use 
was made of boundedness of the domain of the independent variable.

More generally, a consequence of Lemmas 4.9 and 4.10 in \cite{BCCJM2022} 
is that under certain fairly general conditions, $\mathfrak{m}$ must
be an even integer. Our goal in this section is to slightly relax 
the assumptions from these lemmas. We begin with the following 
{\it exchange principle}, addressing what happens if the skew-symmetric
$n$-linear form $\omega_2$ is exchanged for an alternative choice
$\omega_3$.

\begin{lemma}[Exchange Principle] \label{exchange-lemma}
Suppose that for some interval $[a, b]$, $a < b$, we have 
$\mathpzc{g} \in C([a, b], Gr_n (\mathbb{R}^{2n}))$, 
and that $\{\omega_i\}_{i=1}^3$ are three skew-symmetric 
$n$-linear forms on $\mathbb{R}^{2n}$, with 
$\ker \omega_1 = \mathpzc{q}$. If the triples 
$(\mathpzc{g}, \omega_1, \omega_2)$ and 
$(\mathpzc{g}, \omega_1, \omega_3)$ are both invariant 
on $[a, b]$, and neither of the endpoints $t = a$ and $t = b$
is a crossing point, then the hyperplane indices computed
for $(\mathpzc{g}, \omega_1, \omega_2)$ and 
$(\mathpzc{g}, \omega_1, \omega_3)$ on $[a, b]$ can differ 
only by an even integer (if at all). 
\end{lemma}

\begin{proof}
Starting with the triple $(\mathpzc{g}, \omega_1, \omega_2)$, a value 
$t_* \in [a, b]$ is a crossing point if and only if $\tilde{\omega}_1 (t_*) = 0$,
and by our assumption of invariance, we must correspondingly 
have $\tilde{\omega}_2 (t_*) \ne 0$. As $t$ increases through $t_*$, 
$\tilde{\omega}_1 (t)$ might change signs, but by continuity 
$\tilde{\omega}_2 (t)$ will not. If $\tilde{\omega}_1 (t)$ changes
signs, the contribution to the 
hyperplane index is either $+1$ or $-1$. On the other hand, if 
$\tilde{\omega}_1 (t)$ fails to change signs, then 
there is no contribution to the hyperplane index. 

Turning to the triple $(\mathpzc{g}, \omega_1, \omega_3)$, precisely
the same statements above are true, and in particular we see that 
a crossing point $t_*$ gives no contribution to the hyperplane 
index if and only
if $\tilde{\omega}_1 (t)$ fails to change signs as $t$ increases 
through $t_*$. This means that there will be a non-zero 
contribution to the hyperplane index at $t_*$ for the
triple $(\mathpzc{g}, \omega_1, \omega_2)$ if and only if 
there is a non-zero contribution to the hyperplane 
index at $t_*$ for the triple $(\mathpzc{g}, \omega_1, \omega_3)$. 

According to these considerations, the hyperplane index for 
the triple $(\mathpzc{g}, \omega_1, \omega_2)$ on $[a, b]$ will 
have precisely the same number of non-zero crossings as the hyperplane 
index for the triple $(\mathpzc{g}, \omega_1, \omega_3)$ on $[a, b]$. 
If $P_2$ and $N_2$ respectively denote the number of positive and 
negative crossings for $(\mathpzc{g}, \omega_1, \omega_2)$, and 
$P_3$ and $N_3$ respectively denote the number of positive and 
negative crossings for $(\mathpzc{g}, \omega_1, \omega_3)$, then 
we must have $P_2 + N_2 = P_3 + N_3$. If follows that the difference
between the hyperplane indices computed
for $(\mathpzc{g}, \omega_1, \omega_2)$ and 
$(\mathpzc{g}, \omega_1, \omega_3)$ on $[a, b]$ is 
\begin{equation*}
    P_3 - N_3 - (P_2 - N_2)
    = (P_2 + N_2 - N_3) - N_3 - P_2 + N_2
    = 2N_2 - 2N_3,
\end{equation*}
an even number. 
\end{proof}

\begin{lemma} \label{invariance-lemma}
Suppose that for some intervals $[a, b]$, $a < b$ and $[c, d]$, $c < d$,
$\mathpzc{g} \in C([a, b] \times [c, d], Gr_n (\mathbb{R}^{2n}))$, and 
for some fixed $\mathpzc{q} \in Gr_n (\mathbb{R}^{2n})$ let $\omega_1$ 
denote a skew-symmetric $n$-linear form on $\mathbb{R}^{2n}$ so that 
$\ker \omega_1 = \mathpzc{q}$. Let $\omega_2$ denote a second 
skew-symmetric $n$-linear form on $\mathbb{R}^{2n}$, and suppose 
that for some point $(s_*, t_*)$ in the interior of $[a, b] \times [c, d]$
we have 
\begin{equation*}
    \omega_i (g_1 (s_*, t_*), g_2 (s_*, t_*), \dots, g_n (s_*, t_*))
    = 0, \quad i = 1, 2,
\end{equation*}
but that there exists a sufficiently small ball $B \subset \mathbb{R}^2$ 
centered at $(s_*, t_*)$
so that for all $(s, t) \in B \backslash \{(s_*, t_*)\}$ the values 
\begin{equation*}
    \omega_i (g_1 (s, t), g_2 (s, t), \dots, g_n (s, t)), \quad i = 1, 2,
\end{equation*}
are not both $0$. In short, the triple $(\mathpzc{g}, \omega_1, \omega_2)$
loses invariance at an isolated point $(s_*, t_*)$. Then there exists 
some $\epsilon > 0$ sufficiently small so that for any ball 
$\mathcal{B} \subset \mathbb{R}^2$ centered at $(s_*, t_*)$
with radius less than $\epsilon$
\begin{equation*}
    \ind (\mathpzc{g}, \mathpzc{q}; \partial \mathcal{B})
    \in 2 \mathbb{Z},
\end{equation*}
where as with all boundary indices we take $\partial \mathcal{B}$ to 
be traversed in the counterclockwise direction (though 
the direction doesn't strictly matter for the result).
\end{lemma}

\begin{proof} First, let $\omega_3$ denote any third skew-symmetric 
$n$-linear form on $\mathbb{R}^{2n}$ so that 
\begin{equation*}
    \omega_3 (g_1 (s_*, t_*), g_2 (s_*, t_*), \dots, g_n (s_*, t_*))
    \ne 0.
\end{equation*}
This is always possible by choosing vectors $\{g_i\}_{i = n+1}^{2n}$
so that the collection $\{g_i (s_*, t_*)\}_{i=1}^n \cup \{g_i\}_{i = n+1}^{2n}$
comprises a basis for $\mathbb{R}^{2n}$ and taking the kernel of $\omega_3$ to be
the space spanned by the collection $\{g_i\}_{i = n+1}^{2n}$ (ensuring 
that none of the vectors  $\{g_i (s_*, t_*)\}_{i=1}^n$ is contained in 
the kernel of $\omega_3$). By continuity
of $\mathpzc{g}$, we can take $\epsilon > 0$ sufficiently small so that 
for any ball $\mathcal{B} \subset \mathbb{R}^2$ centered at $(s_*, t_*)$
with radius smaller than $\epsilon$
\begin{equation*}
    \omega_3 (g_1 (s, t), g_2 (s, t), \dots, g_n (s, t))
    \ne 0, \quad \forall\, (s, t) \in \overline{\mathcal{B}}, 
\end{equation*}
where the overbar denotes closure. 
I.e., the triple $(\mathpzc{g}, \omega_1, \omega_3)$ is invariant 
on $\overline{\mathcal{B}}$, and so by homotopy invariance
\begin{equation*}
    \ind_{\omega_3} (\mathpzc{g}, \mathpzc{q}, \partial \mathcal{B}) = 0.
\end{equation*}
If $\omega_1 (g_1 (s, t), \dots, g_n (s, t))$ is identically 0 for 
$(s, t) \in \partial \mathcal{B}$, then 
$ \ind_{\omega_i} (\mathpzc{g}, \mathpzc{q}, \partial \mathcal{B}) = 0$
for $i = 2, 3$ and the claim holds trivially. Otherwise, we can select
any $(s_0, t_0) \in \partial \mathcal{B}$ so that 
$\omega_1 (g_1 (s_0, t_0), \dots, g_n (s_0, t_0)) \ne 0$
and compute $\ind_{\omega_i} (\mathpzc{g}, \mathpzc{q}, \partial \mathcal{B})$,
$i = 2, 3$, along $\partial \mathcal{B}$, starting and ending at 
$(s_0, t_0)$. It follows immediately from the exchange principle that 
the difference 
\begin{equation*}
\ind_{\omega_2} (\mathpzc{g}, \mathpzc{q}, \partial \mathcal{B})
    - \ind_{\omega_3} (\mathpzc{g}, \mathpzc{q}, \partial \mathcal{B})
\end{equation*}
is an even number. Since the subtracted index is 0, this gives the claim. 
\end{proof}

We can now use Lemma \ref{invariance-lemma} to show that under circumstances 
that hold quite generally the boundary invariant is an even number. 

\begin{proposition} \label{boundary-index-proposition} 
Suppose that for some intervals $[a, b]$, $a < b$ and $[c, d]$, $c < d$,
$\mathpzc{g} \in C([a, b] \times [c, d], Gr_n (\mathbb{R}^{2n}))$, and 
for some fixed $\mathpzc{q} \in Gr_n (\mathbb{R}^{2n})$ let $\omega_1$ 
denote a skew-symmetric $n$-linear form on $\mathbb{R}^{2n}$ so that 
$\ker \omega_1 = \mathpzc{q}$. Let $\omega_2$ denote a second 
skew-symmetric $n$-linear form on $\mathbb{R}^{2n}$, and suppose the 
triple $(\mathpzc{g}, \omega_1, \omega_2)$ is invariant on the boundary 
of the rectangle $\mathcal{R} := [a, b] \times [c, d]$ and also invariant 
at all except possibly a finite number of points in the interior of
$\mathcal{R}$. Then 
\begin{equation*}
 \ind (\mathpzc{g}, \mathpzc{q}; \partial \mathcal{R})
    \in 2 \mathbb{Z}.    
\end{equation*}
\end{proposition}

\begin{proof}
Let $N$ denote the number of points of invariance in the interior 
of $\mathcal{R}$, and denote these points $\{(s_i, t_i)\}_{i = 1}^N$. 
Using homotopy invariance, we can compute $\ind (\mathpzc{g}, \mathpzc{q}; \partial \mathcal{R})$
by summing the individual hyperplane indices 
$\ind (\mathpzc{g}, \mathpzc{q}; \partial \mathcal{B}_i)$, where $\mathcal{B}_i$ denotes
a ball centered at $(s_i, t_i)$ with radius sufficiently small so that $\mathcal{B}_i \subset \mathcal{R}$
and the triple $(\mathpzc{g}, \omega_1, \omega_2)$ is invariant on $\partial \mathcal{B}_i$.
According to Lemma \ref{invariance-lemma}, each such index must be an even number, and so 
the sum must be an even number as well. 
\end{proof}

\begin{remark} \label{invariance-remark}
In order to understand why we expect the number of points at which 
invariance is lost to be finite, we observe that generally
the sets 
\begin{equation*}
    \mathcal{C}_i := \{(s, t): \tilde{\omega}_i (s, t) = 0 \}, 
    \quad i = 1, 2,
\end{equation*}
comprise one-dimensional curves in $[a, b] \times [c, d]$, and 
points at which invariance is lost are precisely the points
at which these curves intersect. In principle, such intersections
certainly need not be isolated, but in practice we generally find
that they are. For now, the theory is missing a sufficiently 
general result along these lines, and the boundary invariant must
be computed in applications on a case-by-case basis (see Section
\ref{applications-section}). While a precise value of $\mathfrak{m}$
is certainly optimal, we emphasize that for instability arguments
it's often sufficient to identify its parity, since this allows us 
to determine whether the number of unstable eigenvalues is even 
or odd. 
\end{remark}

\section{ODE Preliminaries} 
\label{ev-problems-section}

In this section, we collect several straightforward results 
associated with solutions to (\ref{nonhammy}) and (\ref{mathcal-Y-system}). 
As a starting point, the following proposition is adapted from Proposition 1.2 of 
\cite{PW1992}. 

\begin{proposition} \label{ode-proposition1}
Let Assumptions {\bf (A)} through {\bf (D)} hold. Then the following statements
are true. 

\medskip
(i) There exists a unique
solution $\eta^- (x; \lambda)$ to (\ref{nonhammy}) for which the limit 
\begin{equation*}
    \lim_{x \to - \infty} e^{- \mu_- (\lambda) x} \eta^- (x; \lambda) = v^- (\lambda) 
\end{equation*}
holds, where $v^- (\lambda)$ is the right eigenvector of $A_- (\lambda)$ described 
in Assumption {\bf (C)}. Moreover, the convergence is uniform on compact 
subsets of $\Omega$.

\medskip
(ii) There exists a (non-unique) solution $\zeta^+ (x; \lambda)$ to (\ref{nonhammy})
for which the limit 
\begin{equation*}
    \lim_{x \to + \infty} e^{- \mu_+ (\lambda) x} \zeta^+ (x; \lambda) 
    = v^+ (\lambda) 
\end{equation*}
holds, where $v^+ (\lambda)$ is the right eigenvector of $A_+ (\lambda)$ described 
in Assumption {\bf (C)}. Moreover, the convergence is uniform on compact 
subsets of $\Omega$.
\end{proposition}

For labeling purposes, we will let $\{y_j^- (x; \lambda)\}_{j=1}^n$ denote 
a linearly independent collection of solutions to (\ref{nonhammy}), 
indexed so that $\eta^- (x; \lambda) = y_n^- (x; \lambda)$, and 
we will let $\{y_j^+ (x; \lambda)\}_{j=1}^n$ denote 
a linearly independent collection of solutions to (\ref{nonhammy}),
indexed so that $\zeta^+ (x; \lambda) = y_n^+ (x; \lambda)$. 
In some places, it will be useful to express coordinates of the elements 
$\{y_j^{\pm} (x; \lambda)\}_{j=1}^n$ by writing 
\begin{equation*}
    y_j^{\pm} (x; \lambda) 
    = \begin{pmatrix}
    y_{1j}^{\pm} (x; \lambda) & y_{2j}^{\pm} (x; \lambda) & \cdots & y_{nj}^{\pm} (x; \lambda)      
    \end{pmatrix}^T,
    \quad j = 1, 2, \dots, n,
\end{equation*}
and we also introduce the matrices
\begin{equation*}
    Y^{\pm} (x; \lambda) 
    := \begin{pmatrix}
    y_1^{\pm} (x; \lambda) & y_2^{\pm} (x; \lambda) & \cdots & y_{n-1}^{\pm} (x; \lambda)  
    \end{pmatrix}.
\end{equation*}

Recalling the specification of $\mathcal{Y}^+ (x; \lambda)$ in (\ref{mathcal-Y}),
it's straightforward to show that $\mathcal{Y}^+ (x; \lambda)$ can be 
expressed as 
\begin{equation} \label{mathcal-A-coordinates}
\begin{aligned}
    \mathcal{Y}^+ (x; \lambda) 
    &= d_n^+ (x; \lambda) e_1 \wedge \dots \wedge e_{n-1} 
    + d_{n-1}^+ (x; \lambda) e_1 \wedge \dots \wedge e_{n-2} \wedge e_n \\
    &+ \dots + d_1^+ (x; \lambda) e_2 \wedge \dots \wedge e_n,
\end{aligned}    
\end{equation}
where $d_i^+ (x; \lambda)$ denotes the determinant of the $(n-1) \times (n-1)$ matrix
obtained by eliminating the $i^{th}$ row of $Y^+$. In this way, we associate 
$\mathcal{Y}^+ (x; \lambda)$ with the vector 
\begin{equation*}
\mathcal{Y}^+ (x; \lambda)
= \begin{pmatrix}
d_n^+ (x; \lambda) & d_{n-1}^+ (x; \lambda) & \dots & d_1^+ (x; \lambda)
\end{pmatrix}^T,
\end{equation*}
with the convention of Remark \ref{wedge-not-wedge}.

\begin{proposition}\label{tilde-A-proposition}
With $\mathcal{Y}^+$ specified as in (\ref{mathcal-Y}), 
(\ref{wedge-ode}) holds.
\end{proposition}

\begin{proof}
Upon differentiation of (\ref{mathcal-Y}), we obtain the relation 
\begin{equation}\label{A-summands}
    \mathcal{Y}^{+\,\prime} (x; \lambda)
    = \sum_{j=1}^{n-1} y_1^+ \wedge \cdots \wedge A(x; \lambda) y_j^+ \wedge \cdots \wedge y_{n-1}^+,
\end{equation}
for which each summand can be understood similarly as in (\ref{mathcal-A-coordinates}). Focusing
on the first summand 
\begin{equation*}
    (A(x; \lambda) y_1^+) \wedge y_2^+ \wedge \cdots \wedge y_{n-1}^+,
\end{equation*}
we can write 
\begin{equation*}
    A(x; \lambda) y_1^+ 
    = \begin{pmatrix}
        a_{1j} y_{j1}^+ & a_{2j} y_{j1}^+ & \dots & a_{nj} y_{j1}^+
    \end{pmatrix}^T,
\end{equation*}
where for notational brevity summation is assumed over repeated indices. In addition, we 
introduce the matrix
\begin{equation*}
    Y_{A1}^+ (x; \lambda) 
    := \begin{pmatrix}
    A(x; \lambda) y_1^+ (x; \lambda) & y_2^+ (x; \lambda) & \cdots & y_{n-1}^+ (x; \lambda)  
    \end{pmatrix},
\end{equation*}
and for each $k \in \{1, 2, \dots, n\}$ we let $d_{A1,k}^+$ denote the determinant of the 
$(n-1) \times (n-1)$ matrix obtained by eliminating the $k^{th}$ row of $Y_{A1}^+$. 
Then, as in (\ref{mathcal-A-coordinates}),
\begin{equation} \label{mathcal-A-first}
\begin{aligned}
   (A(x; \lambda) y_1^+) &\wedge \cdots \wedge  y_j^+ \wedge \cdots \wedge y_n^+
    = d_{A1,n}^+ (x; \lambda) e_1 \wedge \dots \wedge e_{n-1} \\
    &+ d_{A1,n-1}^+ (x; \lambda) e_1 \wedge \dots \wedge e_{n-2} \wedge e_n 
    + \dots + d_{A1,1}^+ (x; \lambda) e_2 \wedge \dots \wedge e_n.
\end{aligned}    
\end{equation}
In this way, we have associated $\mathcal{Y}^+$ with the vector 
$(d_n^+ \,\, d_{n-1}^+ \,\, \dots, \,\, d_1^+)$ and 
$(A y_1^+) \wedge y_2^+ \wedge \cdots \wedge y_{n-1}^+$
with the vector $(d_{A1,n}^+ \,\, d_{A1,n-1}^+ \,\, \dots, \,\, d_{A1,1}^+)$,
and likewise we can associate the $k^{th}$ summand from 
(\ref{A-summands}) with a vector $(d_{Ak,n}^+ \,\, d_{Ak,n-1}^+ \,\, \dots, \,\, d_{Ak,1}^+)$.
Using these associations, we can write the derivative of the $j^{\textrm{th}}$ component
of $\mathcal{Y}^+$ as 
\begin{equation*}
    \mathcal{Y}_j^{+ \, \prime} = \sum_{k=1}^{n-1} d^+_{Ak,n-(j-1)}.
\end{equation*}
Focusing for specificity on the first component $j=1$, we see that 
$\mathcal{Y}_1^{+ \, \prime}$ is a sum of $n$ determinants. If we focus 
still further on terms in this sum associated with a specific entry
of the matrix $A$, then we can readily identify the appearance of 
that component in our final relation. Using $a_{11}$ as an example
case, we can schematically view the terms in $\mathcal{Y}_1^{+ \, \prime}$
associated with $a_{11}$ as arising from the sum 
\begin{equation*}
\begin{aligned}
    \det &\begin{pmatrix}
        a_{11} y_{11}^+ & y_{12}^+ & \cdots & y_{1(n-1)}^+ \\
        * & y_{22}^+ & \cdots & y_{2(n-1)}^+ \\
        \vdots & \vdots & \cdots & \vdots \\
        * & y_{(n-1)2}^+ & \cdots & y_{(n-1)(n-1)}^+ \\
    \end{pmatrix}
    +  \det \begin{pmatrix}
        y_{11}^+ & a_{11} y_{12}^+ & \cdots & y_{1(n-1)}^+ \\
        y_{21}^+ & * & \cdots & y_{2(n-1)}^+ \\
        \vdots & \vdots & \cdots & \vdots \\
        y_{(n-1)1}^+ & * & \cdots & y_{(n-1)(n-1)}^+ \\
    \end{pmatrix}
    + \dots \\
    &+ \det \begin{pmatrix}
        y_{11}^+ & y_{12}^+ & \cdots & a_{11} y_{1(n-1)}^+ \\
        y_{21}^+ & y_{22}^+ & \cdots & * \\
        \vdots & \vdots & \cdots & \vdots \\
        y_{(n-1)1}^+ & y_{(n-1)2}^+ & \cdots & * \\
    \end{pmatrix},
    \end{aligned}
\end{equation*}
where the asterisks indicate terms irrelevant to the calculation
(because they don't contain $a_{11}$). 
If we now think of expanding each determinant along the column 
with $a_{11}$, the combinations of terms including $a_{11}$ 
are precisely the same as $a_{11}$ multiplied by a determinant 
expansion along the first row of the matrix 
\begin{equation*}
    \begin{pmatrix}
        y_{11}^+ & y_{12}^+ & \cdots & y_{1(n-1)}^+ \\
        y_{21}^+ & y_{22}^+ & \cdots & y_{2(n-1)}^+ \\
        \vdots & \vdots & \cdots & \vdots \\
        y_{(n-1)1}^+ & y_{(n-1)2}^+ & \cdots & y_{(n-1)(n-1)}^+ \\
    \end{pmatrix}.
\end{equation*}
In summary, the sole multiplier of $a_{11}$ in the expression 
for $\mathcal{Y}_1^{+ \, \prime}$ is the quantity labeled
above as $d_n^+$, which is precisely the first component of 
$\mathcal{Y}^+$ (i.e., the component $\mathcal{Y}_1^+$). Proceeding similarly 
for each element of $A$ and each component of $\mathcal{Y}^+$,
we obtain the claim. 
\end{proof}

\begin{proposition} \label{eigenvalues-proposition}
For any matrix $A \in \mathbb{C}^{n \times n}$, let $\tilde{A}$ 
be the associated matrix specified in (\ref{tildeA}). 
Then $\mu \in \sigma (A)$ if and only if $- \mu \in \sigma(\tilde{A})$.
\end{proposition}

\begin{proof}
If $\mu \in \sigma (A)$ then there exists a left eigenvector 
of $A$, which we denote $w \in \mathbb{C}^n$, so that 
$wA = \mu wA$. If we then specify a new (column) vector
$\tilde{\mathcal{V}} \in \mathbb{C}^n$ with components
\begin{equation} \label{corresponding-eigenvector1}
\tilde{\mathcal{V}}_j = (-1)^{n-j} w_{n+1-j},
\quad j = 1, 2, \dots, n,    
\end{equation}
then we find by direct calculation that 
$\tilde{A} \tilde{\mathcal{V}} = - \mu \tilde{\mathcal{V}}$. 
Reversing the argument gives the converse direction. 
\end{proof}

\begin{remark} \label{eigenvector-remark}
    We see from the proof of Proposition \ref{eigenvalues-proposition}
    that if $w^{\pm} (\lambda)$ is a left eigenvector of $A_{\pm} (\lambda)$
    associated with the simple eigenvalue $\mu_{\pm} (\lambda)$, with 
    $w^{\pm} (\lambda)$ chosen to be analytic in $\lambda$
    as in Assumption {\bf (C)}, then the corresponding right eigenvector 
    of $\tilde{A}_{\pm} (\lambda)$ specified via (\ref{corresponding-eigenvector1})
    will also be analytic. If (\ref{corresponding-eigenvector1})
    holds, then we correspondingly have 
    \begin{equation} \label{corresponding-eigenvector2}
    w_k^+ = (-1)^{k-1} \tilde{\mathcal{V}}_{n+1-k}^+.   
    \end{equation}
    In addition, if $v^{\pm} (\lambda)$ denotes 
    the analytic right eigenvector of $A_{\pm} (\lambda)$ normalized so 
    that $w^{\pm} (\lambda) v^{\pm}(\lambda) = 1$, then  
    we can use (\ref{the-wedge-product}) along with 
    (\ref{corresponding-eigenvector2}) to compute  
    \begin{equation*}
    \begin{aligned}
    v^{\pm} (\lambda) \wedge \tilde{\mathcal{V}}^{\pm} (\lambda)
    &= \sum_{i=1}^n (-1)^{i+1} v_i^{\pm} (\lambda) \tilde{\mathcal{V}}_{n+1-i}^{\pm} (\lambda)
    = \sum_{i=1}^n (-1)^{i+1} v_i^{\pm} (\lambda) (-1)^{i-1} w_i^{\pm} (\lambda) \\
    &= w^{\pm} (\lambda) v^{\pm} (\lambda) = 1.
    \end{aligned}
    \end{equation*}   
\end{remark}

\begin{proposition} \label{conjugation-proposition}
For any matrix $A \in \mathbb{C}^{n \times n}$, let $\tilde{A}$ 
be the associated matrix specified in (\ref{tildeA}), and for 
$u, \mathcal{U} \in \mathbb{R}^n$ interpret $u$ and $Au$ as 
elements of $\bigwedge^1 (\mathbb{R})$ and $\mathcal{U}$ 
and $\tilde{A} \mathcal{U}$ as elements of $\bigwedge^{n-1} (\mathbb{R})$, 
as in Remark \ref{wedge-not-wedge}. Then 
\begin{equation*}
    (Au) \wedge \mathcal{U} + u \wedge (\tilde{A} \mathcal{U}) = 0.
\end{equation*}
\end{proposition}

\begin{proof}
First, since $Au$ is interpreted as a 1-form and $\mathcal{U}$ is 
interpreted as an $(n-1)$-form, the wedge product $(Au) \wedge \mathcal{U}$ is 
an $n$-form, which we see from (\ref{the-wedge-product}) is  
\begin{equation*}
    (Au) \wedge \mathcal{U}
    = \sum_{i, j = 1}^n (-1)^{i+1} a_{ij} u_j \mathcal{U}_{n+1-i}.
\end{equation*}
We can compare this with 
\begin{equation*}
    \begin{aligned}
    u \wedge (\tilde{A} \mathcal{U})
    &= u_1 \sum_{j=1}^n \tilde{a}_{nj} \mathcal{U}_j - u_2 \sum_{j=1}^n \tilde{a}_{(n-1)j} \mathcal{U}_j
    + \dots + (-1)^{n+1} u_n \sum_{j=1}^n \tilde{a}_{1j} \mathcal{U}_j \\
    &= \sum_{i, j = 1}^n (-1)^{i+1} \tilde{a}_{(n+1-i)j} u_i \mathcal{U}_j 
    = \sum_{i, j = 1}^n (-1)^{i+1} (-1)^{n-i+j} a_{(n+1-j)i} u_i \mathcal{U}_j \\
    &= \sum_{i, j = 1}^n (-1)^{n+1+j} a_{(n+1-j)i} u_i \mathcal{U}_j
    = \sum_{i, k = 1}^n (-1)^{k} a_{k i} u_i \mathcal{U}_{n+1-k},
    \end{aligned}
\end{equation*}
where in obtaining the final equality we changed indices from $j$ to 
$k = n+1-j$. Comparing our expression for $ (Au) \wedge \mathcal{U}$ with 
our expression for $u \wedge (\tilde{A} \mathcal{U})$, we see that the claim 
is proved. 
\end{proof}

With $\eta^- (x; \lambda)$ as specified in (\ref{eta-minus}) and 
$\tilde{\mathcal{Y}}^+ (x; \lambda)$ as specified in 
(\ref{tilde-mathcal-Y-plus-defined}), we 
detect intersections between $\Span \{\eta^- (x; \lambda)\}$
(a one-dimensional subspace of $\mathbb{R}^n$) and $\Span \{y_j^+ (x; \lambda)\}_{j=1}^{n-1}$
(an $(n-1)$-dimensional subspace of $\mathbb{R}^n$), with 
the wedge product 
\begin{equation*}
    \eta^- (x; \lambda) \wedge \tilde{\mathcal{Y}}^+ (x; \lambda).
\end{equation*}

The following proposition serves as a direct connection between the 
current analysis and that of \cite{PW1992}. 

\begin{proposition} \label{connection-to-PW}
Suppose $\tilde{\mathcal{Y}} (x; \lambda)$ solves the ODE system 
$\tilde{\mathcal{Y}}' = \tilde{A} (x; \lambda) \tilde{\mathcal{Y}}$,
and let the row vector $z = (z_1 \,\, z_2 \,\, \dots \,\, z_n)$ 
be specified with components
\begin{equation} \label{proposition5relation}
    z_j (x; \lambda)
    := (-1)^{j-1} \tilde{\mathcal{Y}}_{n+1-j} (x; \lambda),
    \quad j = 1, 2, \dots, n.
\end{equation}
Then $z$ satisfies the ODE system $z' = -z A(x; \lambda)$.  
Moreover, if $\tilde{\mathcal{Y}}$ is interpreted as 
an $(n-1)$-form as in Remark \ref{wedge-not-wedge}, and
we introduce a column vector $y$ interpreted as a 1-form 
as in Remark \ref{wedge-not-wedge}, 
then 
\begin{equation*}
    y \wedge \tilde{\mathcal{Y}} (x; \lambda)
    = z(x;\lambda) y.
\end{equation*}
\end{proposition}

\begin{proof}
For the first part of the statement, we can compute directly, writing 
\begin{equation*}
    \begin{aligned}
    z_j' &= (-1)^{j-1} \tilde{\mathcal{Y}}^{\prime}_n
    = (-1)^{j-1} \sum_{k=1}^n \tilde{a}_{(n+1-j)k} \tilde{\mathcal{Y}}_k \\
    &= (-1)^{j-1} \sum_{k=1}^n (-1)^{n+k-j} a_{(n+1-k)j} (-1)^{n-k} z_{n+1-k}
    = - \sum_{k=1}^n a_{(n+1-k)j} z_{n+1-k},   
    \end{aligned}
\end{equation*}
for $j = 1, 2, \dots, n$, which is precisely $z' = -z A(x; \lambda)$
in component form. 

For the second claim, we compute 
\begin{equation*}
    y \wedge \tilde{\mathcal{Y}}
    = \sum_{i=1}^n (-1)^{i+1} y_i \tilde{\mathcal{Y}}_{n+1-i}
    = \sum_{i=1}^n (-1)^{i+1} y_i (-1)^{i-1} z_i 
    = \sum_{i=1}^n \eta_i^- z_i = zy. 
\end{equation*}
\end{proof}

\begin{remark} \label{PW1992-remark} We see from 
Proposition \ref{connection-to-PW} that as in \cite{PW1992},
we could carry out our analysis entirely with appropriate 
inner products rather than wedge products. Our convention 
of working with wedge products is motivated by the prospect
of extending our analysis to more general settings in which 
the inner-product formulation isn't viable. 
\end{remark}

Proposition \ref{connection-to-PW} allows us to adopt Proposition 1.2 
from \cite{PW1992} (addressing the variable denoted $z$ here) to a statement 
about $\tilde{\mathcal{Y}}^+ (x; \lambda)$. Precisely, we have the following.

\begin{proposition} \label{ode-proposition2}
Let Assumptions {\bf (A)} through {\bf (D)} hold, and let 
$\tilde{\mathcal{V}}^+ (\lambda)$ denote the right eigenvector of 
$\tilde{A}_+ (\lambda)$ associated with the eigenvalue 
$- \mu_+ (\lambda)$, constructed via (\ref{corresponding-eigenvector2})
from $w^+ (\lambda)$ as in Assumption {\bf (C)}. Then  
there exists a unique
solution $\tilde{\mathcal{Y}}^+ (x; \lambda)$ to
(\ref{mathcal-Y-system}) for which the limit 
\begin{equation*}
    \lim_{x \to + \infty} e^{\mu_+ (\lambda) x} \tilde{\mathcal{Y}}^+ (x; \lambda) 
    = \tilde{\mathcal{V}}^+ (\lambda) 
\end{equation*}
holds, and moreover, the convergence is uniform on compact subsets of $\Omega$.
\end{proposition}

\section{Proof of Theorem \ref{main-theorem}}
\label{proof-section}

With $\eta^- (x; \lambda)$ and $\tilde{\mathcal{Y}}^+ (x; \lambda)$ as in 
Propositions \ref{ode-proposition1} and \ref{ode-proposition2}, 
we let $\mathpzc{g} (x; \lambda) \in Gr_1 (\mathbb{R}^{n})$ 
denote the path of Grassmannian subspaces with 
frame $\mathbf{G} (x; \lambda) = \eta^- (x; \lambda)$, 
and we let $\mathpzc{h} (x; \lambda) \in Gr_{n-1} (\mathbb{R}^{n})$
denote the path of Grassmannian subspaces with frame
\begin{equation} \label{H-frame}
    \mathbf{H} (x; \lambda) 
    = \begin{pmatrix}
        y_1^+ (x; \lambda) & y_2^+ (x; \lambda) & \dots & y_{n-1}^+ (x; \lambda)
    \end{pmatrix}.
\end{equation}
We prove Theorem \ref{main-theorem} by fixing $\lambda_1, \lambda_2 \in I$, 
$\lambda_1 < \lambda_2$, along with values $c_1 \ll 0$ and $c_2 \gg 0$,
and computing the hyperplane index for the pair $\mathpzc{g} (x; \lambda)$
and $\mathpzc{h} (c_2; \lambda)$ along the following sequence of lines
often referred to as the {\it Maslov box}: (1) fix $x = c_1$ and let $\lambda$ 
increase from $\lambda_1$ to $\lambda_2$ (the {\it bottom shelf}); 
(2) fix $\lambda = \lambda_2$ and let $x$ increase from $c_1$ to $c_2$ 
(the {\it right shelf}); (3) fix $x = c_2$ and let $\lambda$
decrease from $\lambda_2$ to $\lambda_1$ (the {\it top shelf}); and (4) fix
$\lambda = \lambda_1$ and let $x$ decrease from $c_2$ to $c_1$ (the 
{\it left shelf}). See Figure \ref{box-figure}. 
 
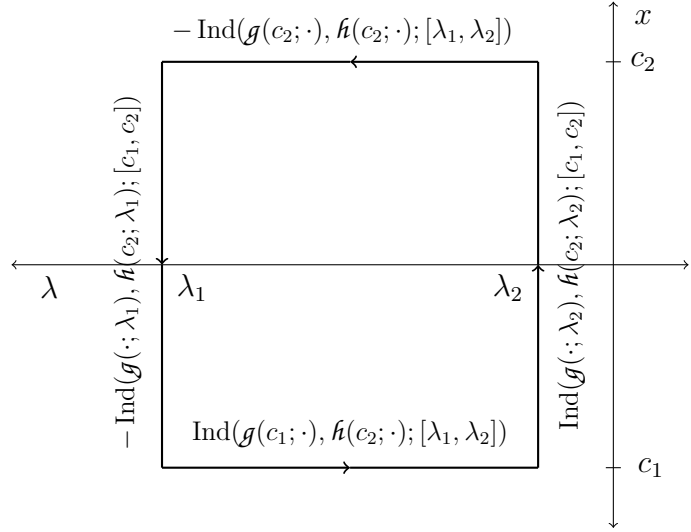
\begin{figure}[ht]
\begin{center}
\begin{tikzpicture}
\draw[<->] (-8,1.5) -- (1,1.5);	
\draw[<->] (0,-2) -- (0,5);	
\node at (.4,4.8) {$x$};
\node at (-7.5,1.2) {$\lambda$};
\node at (-5.6,1.2) {$\lambda_1$};
\node at (-1.4,1.2) {$\lambda_2$};
\node at (.4,4.2) {$c_2$};
\draw[-] (-.1,4.2) -- (.1,4.2);
\node at (.5,-1.2) {$c_1$};
\draw[-] (-.1,-1.2) -- (.1,-1.2);
%
\draw[thick, ->] (-6,-1.2) -- (-3.5,-1.2);
\draw[thick] (-3.5,-1.2) -- (-1,-1.2);	
\draw[thick, ->] (-1,-1.2) -- (-1,1.5);
\draw[thick] (-1,1.5) -- (-1, 4.2);
\draw[thick,->] (-1,4.2) -- (-3.5, 4.2);
\draw[thick] (-3.5,4.2) -- (-6, 4.2);
\draw[thick,->] (-6,4.2) -- (-6, 1.5);
\draw[thick] (-6,1.5) -- (-6, -1.2);
%
\node[scale = .85] at (-3.5, -.75) {$\ind (\mathpzc{g} (c_1; \cdot), \mathpzc{h} (c_2; \cdot); [\lambda_1, \lambda_2])$};
\node[scale = .85, rotate=90] at (-.55, 1.5) {$\ind (\mathpzc{g} (\cdot; \lambda_2), \mathpzc{h} (c_2; \lambda_2); [c_1, c_2])$};
\node[scale = .85] at (-3.6, 4.6) {$- \ind (\mathpzc{g} (c_2; \cdot), \mathpzc{h} (c_2; \cdot); [\lambda_1, \lambda_2])$};
\node[scale = .85, rotate=90] at (-6.45, 1.4) {$- \ind (\mathpzc{g} (\cdot; \lambda_1), \mathpzc{h} (c_2; \lambda_1); [c_1, c_2])$};
\end{tikzpicture}
\end{center}
\caption{The Maslov Box.} \label{box-figure}
\end{figure}
 
{\it Overview of the Maslov box.}
Along the top shelf of the Maslov box, we have $x = c_2$, so the 
hyperplane index 
\begin{equation*}
    \ind (\mathpzc{g} (c_2; \cdot), \mathpzc{h} (c_2; \cdot); [\lambda_1, \lambda_2])
\end{equation*}
detects eigenvalues, albeit counted without multiplicity and with 
no guarantee of monotonicity. If $\mathcal{N}_{\#} ([\lambda_1, \lambda_2])$
denotes the number of eigenvalues that (\ref{nonhammy}) has on the interval 
$[\lambda_1, \lambda_2]$, counted without multiplicity, then in the 
event of monotonicity the hyperplane index on the top shelf 
would equal either $\mathcal{N}_{\#} ([\lambda_1, \lambda_2])$
or $-\mathcal{N}_{\#} ([\lambda_1, \lambda_2])$, depending on the 
direction of the crossings. In the absence of monotonicity, 
such an equality isn't achieved, and instead we have 
the inequality,
\begin{equation} \label{count-inequality}
    \mathcal{N}_{\#} ([\lambda_1, \lambda_2])
    \ge |\ind (\mathpzc{g} (c_2; \cdot), \mathpzc{h} (c_2; \cdot); [\lambda_1, \lambda_2])|.
\end{equation}

The hyperplane index along the bottom shelf detects intersections
between $\mathpzc{g} (c_1; \lambda)$ and $\mathpzc{h} (c_2; \lambda)$, and can be 
denoted 
\begin{equation*}
    \ind (\mathpzc{g} (c_1; \cdot), \mathpzc{h} (c_2; \cdot); [\lambda_1, \lambda_2]).
\end{equation*}
Likewise the hyperplane indices along the left and right shelves respectively detect 
intersections between $\mathpzc{g} (x; \lambda_i)$ and $\mathpzc{h} (x; \lambda_i)$,
$i = 1, 2$, as $x$ decreases from $c_2$ to $c_1$ (left shelf) and increases from 
$c_1$ to $c_2$ (right shelf). We denote these respectively
\begin{equation*}
- \ind (\mathpzc{g} (\cdot; \lambda_1), \mathpzc{h} (c_2; \lambda_1); [c_1, c_2])    
\end{equation*}
and 
\begin{equation*}
\ind (\mathpzc{g} (\cdot; \lambda_2), \mathpzc{h} (c_2; \lambda_2); [c_1, c_2]).
\end{equation*}

From Assumption {\bf (E)}, we can conclude  that for $c_1$ sufficiently negative
and $c_2$ sufficiently positive, we have invariance along each of the four shelves.
It follows that we can compute the hyperplane index along the boundary 
of the Maslov box, and we denote this value $\mathfrak{m} (c_1, c_2)$, writing
\begin{equation} \label{mathfrak-m-defined}
\begin{aligned}
    \mathfrak{m} (c_1, c_2) 
    &:=  \ind (\mathpzc{g} (c_1; \cdot), \mathpzc{h} (c_2; \cdot); [\lambda_1, \lambda_2])
    + \ind (\mathpzc{g} (\cdot; \lambda_2), \mathpzc{h} (c_2; \lambda_2); [c_1, c_2]) \\
    &\quad - \ind (\mathpzc{g} (c_2; \cdot), \mathpzc{h} (c_2; \cdot); [\lambda_1, \lambda_2])
    - \ind (\mathpzc{g} (\cdot; \lambda_1), \mathpzc{h} (c_2; \lambda_1); [c_1, c_2]).  
\end{aligned}
\end{equation}

In order to evaluate the four hyperplane indices in (\ref{mathfrak-m-defined}), we follow the approach 
outlined in the introduction, beginning with the specification of 
a third Grassmannian subspace $\mathpzc{f}^{c_2} (x; \lambda) \in Gr_n (\mathbb{R}^{2n})$
with frame 
\begin{equation*} 
    \mathbf{F}^{c_2} (x; \lambda)
    = \begin{pmatrix}
        \mathbf{G} (x; \lambda) & 0_{n \times (n-1)} \\
        0_{n \times 1} & \mathbf{H} (c_2; \lambda)
    \end{pmatrix}.
\end{equation*}
As discussed in Section \ref{maslov-section}, we set 
\begin{equation} \label{omega1specified-proof}
    \omega_1^{c_2} (f_1, f_2, \dots, f_n)
    := f_1 (x; \lambda) \wedge \dots \wedge f_n (x; \lambda)
    \wedge \tilde{\delta}_1 \wedge \dots \wedge \tilde{\delta}_n,
\end{equation}
where the vectors $\{\tilde{\delta}_i\}_{i=1}^n$ comprise the columns
of $\tilde{\mathbf{\Delta}} = \genfrac{(}{)}{0pt}{2}{-I_n}{I_n}$. 
Likewise, we fix some invertible matrix $M \in \mathbb{R}^{n \times n}$ and set 
\begin{equation} \label{omega2specified-proof}
    \omega_2^{c_2} (f_1, f_2, \dots, f_n)
    := f_1 (x; \lambda) \wedge \dots \wedge f_n (x; \lambda)
    \wedge \tilde{\sigma}_1 \wedge \dots \wedge \tilde{\sigma}_n,
\end{equation}
where the vectors $\{\tilde{\sigma}_i\}_{i=1}^n$ comprise the columns
of $\tilde{\mathbf{\Sigma}} = \genfrac{(}{)}{0pt}{2}{-M}{I_n}$. 

For the subsequent calculations, we will evaluate $\omega_1$ and $\omega_2$
on the columns of $\mathbf{F}^{c_2} (x; \lambda)$, giving 
\begin{equation} \label{tilde-omega1-again}
    \tilde{\omega}_1^{c_2} (x; \lambda)
    := \eta^- (x; \lambda) \wedge \mathcal{Y}^+ (c_2; \lambda)
\end{equation}
and 
\begin{equation} \label{tilde-omega2-again}
    \tilde{\omega}_2^{c_2} (x; \lambda)
    := \eta^- (x; \lambda) \wedge \mathcal{Y}_M^+ (c_2; \lambda),
\end{equation}
where $\mathcal{Y}^+$ and $\mathcal{Y}_M^+$
are respectively defined in (\ref{mathcal-Y}) and (\ref{mathcal-y-m-specified}).
(See Section \ref{grassmanian-pairs-section} for additional details
about these wedge products.) 
Here, $\tilde{\omega}_1^{c_2} (x; \lambda)$ and $\tilde{\omega}_2^{c_2} (x; \lambda)$
are the same as (\ref{tilde-omega1-defined-c}) and (\ref{tilde-omega2-defined-c}), 
except with $c$ replaced by $c_2$. Likewise, we take 
$\tilde{\psi}_1^{c_2} (x; \lambda)$, $\tilde{\psi}_1^+ (x; \lambda)$, and
$\tilde{\psi}_1^{+,-} (\lambda)$ to be as respectively defined 
in (\ref{psi1-tilde-c-defined}), (\ref{psi1-tilde-plus-defined}), 
and (\ref{psi1-tilde-plus-minus-defined}), with analogous definitions
for $\tilde{\psi}_2^{c_2} (x; \lambda)$, $\tilde{\psi}_2^+ (x; \lambda)$,
and $\tilde{\psi}_2^{+,-} (\lambda)$ in (\ref{psi1-2-tilde-c-defined}), (\ref{psi2-tilde-plus-defined}), 
and (\ref{psi2-tilde-plus-minus-defined}). Finally, we let $\tilde{\psi}_1^{c_2, -} (\lambda)$
and $\tilde{\psi}_2^{c_2, -} (\lambda)$ be defined as 
\begin{equation*}
\tilde{\psi}_i^{c_2, -} (\lambda)
:= \lim_{x \to -\infty} \tilde{\psi}_i^{c_2} (x; \lambda),
\quad i = 1, 2.
\end{equation*}

In addition to the specifications above, we will denote by $p^{c_2} (x; \lambda)$
the tracking point $p$ from (\ref{p-specified}) evaluated with $\tilde{\omega}_1$
and $\tilde{\omega}_2$ replaced with $\tilde{\psi}_1^{c_2} (x; \lambda)$ and 
$\tilde{\psi}_2^{c_2} (x; \lambda)$, and we define  
$p^{+} (x; \lambda)$, $p^{+,-} (\lambda)$, and $p^{c_2, -} (\lambda)$ 
analogously. 

We are now in a position to state the following useful lemma. 

\begin{lemma} \label{technical-lemma1} Let the assumptions of Theorem 
\ref{main-theorem} hold. Given any $\epsilon > 0$, there exists a 
constant $L$ sufficiently large so that the following hold 
for all $c_2 \ge L$: 

\medskip
(1) For $i = 1, 2$,
\begin{equation*}
    |p^{c_2} (x; \lambda_i) - p^+ (x; \lambda_i)| < \epsilon
\end{equation*}
for all $x \in \mathbb{R}$. 

\medskip
(2) 
\begin{equation*}
        |p^{c_2, -} (\lambda) - p^{+,-} (\lambda)| < \epsilon
\end{equation*}
for all $\lambda \in [\lambda_1, \lambda_2]$.
\end{lemma}

\begin{proof}
Beginning with (1), we observe from our definitions of $\tilde{\psi}_1^{c_2} (x; \lambda)$ and 
$\tilde{\psi}_1^+ (x; \lambda)$ the relation 
\begin{equation*}
\tilde{\psi}_1^{c_2} (x; \lambda)
- \tilde{\psi}_1^+ (x; \lambda)
= \Big( \frac{\eta^- (x; \lambda)}{|\eta^- (x; \lambda)|} \Big) \wedge
\Big(\frac{e^{\mu_+ (\lambda) c_2} \tilde{\mathcal{Y}}^+ (c_2; \lambda)}{|e^{\mu_+ (\lambda) c_2} \tilde{\mathcal{Y}}^+ (c_2; \lambda)|}
- \frac{\tilde{\mathcal{V}}^+ (\lambda)}{|\tilde{\mathcal{V}}^+ (\lambda)|}\Big).
\end{equation*}
Since $\eta^- (x; \lambda) / |\eta^- (x; \lambda)|$ is bounded and 
\begin{equation*}
\lim_{c_2 \to + \infty} 
\Big(\frac{e^{\mu_+ (\lambda) c_2} \tilde{\mathcal{Y}}^+ (c_2; \lambda)}{|e^{\mu_+ (\lambda) c_2} \tilde{\mathcal{Y}}^+ (c_2; \lambda)|}
- \frac{\tilde{\mathcal{V}}^+ (\lambda)}{|\tilde{\mathcal{V}}^+ (\lambda)|}\Big) = 0,
\end{equation*}
we can make the difference $|\tilde{\psi}_1^{c_2} (x; \lambda)
- \tilde{\psi}_1^+ (x; \lambda)|$ as small as we like by taking $c_2$ 
sufficiently large. A similar statement holds for $\tilde{\psi}_2^{c_2} (x; \lambda)$
and $\tilde{\psi}_2^+ (x; \lambda)$, and the claim about 
the difference $|p^{c_2} (x; \lambda_i) - p^+ (x; \lambda_i)|$ follows from 
the continuous dependence of the tracking point on its two inputs, bearing in 
mind that the points $(0,1)$ and $(0,-1)$ are equated for $p$. Here, 
we emphasize that the values $\tilde{\psi}_1^{c_2} (x; \lambda)^2 + \tilde{\psi}_2^{c_2} (x; \lambda)^2$
and $\tilde{\psi}_1^+ (x; \lambda)^2 + \tilde{\psi}_2^+ (x; \lambda)^2$ are bounded away from 
zero by virtue of our invariance assumption, so that the pairs
$(\tilde{\psi}_1^{c_2} (x; \lambda), \tilde{\psi}_2^{c_2} (x; \lambda))$ and 
$(\tilde{\psi}_1^{+} (x; \lambda), \tilde{\psi}_2^{+} (x; \lambda))$ are confined to compact 
subsets of $\mathbb{R}^2$ that do not contain the origin. On such sets, the tracking 
point $p$ is uniformly continuous in its arguments. 

For assertion (2), we write 
\begin{equation*}
\tilde{\psi}_1^{c_2, -} (\lambda) - \tilde{\psi}_1^{+, -} (\lambda)
= \Big(\frac{v^- (\lambda)}{|v^- (\lambda)|} \Big) 
\wedge \Big(\frac{e^{\mu_+ (\lambda) c_2} \tilde{\mathcal{Y}}^+ (c_2; \lambda)}{|e^{\mu_+ (\lambda) c_2} \tilde{\mathcal{Y}}^+ (c_2; \lambda)|}
- \frac{\tilde{\mathcal{V}}^+ (\lambda)}{|\tilde{\mathcal{V}}^+ (\lambda)|}\Big),
\end{equation*}
and the claim follows as for (1).
\end{proof}

Using Lemma \ref{technical-lemma1}, we can now establish 
the following lemma, which uses the notation of (\ref{asymptotic-index})
and (\ref{asymptotic-bottom-shelf}). 

\begin{lemma} \label{equivalence-lemma}
Under the assumptions of Theorem \ref{main-theorem}, there exist a positive constant 
$L$ sufficiently large so that for any constants $c_1$ and $c_2$ so that $c_1 \le - L$ 
and $c_2 \ge L$ there holds
\begin{equation*}
    \begin{aligned}
     &\ind (\mathpzc{g} (c_1; \cdot), \mathpzc{h} (c_2; \cdot); [\lambda_1, \lambda_2])
    + \ind (\mathpzc{g} (\cdot; \lambda_2), \mathpzc{h} (c_2; \lambda_2); [c_1, c_2]) 
    - \ind (\mathpzc{g} (\cdot; \lambda_1), \mathpzc{h} (c_2; \lambda_1); [c_1, c_2]) \\
    &\quad \quad =   \ind (\mathpzc{g}^- (\cdot), \mathpzc{h}^+ (\cdot); [\lambda_1, \lambda_2])
    + \ind (\mathpzc{g} (\cdot; \lambda_2), \mathpzc{h}^+ (\lambda_2); [-\infty, +\infty]) \\
    & \quad \quad \quad \quad  - \ind (\mathpzc{g} (\cdot; \lambda_1), \mathpzc{h}^+ (\lambda_1); [-\infty, +\infty]).
    \end{aligned}
\end{equation*}
\end{lemma}

\begin{proof}
We begin by observing that each hyperplane index on the left-hand side
of the sought equality is computed by tracking the point $p^{c_2} (x; \lambda)$
around $S^1$ as the points $(x, \lambda)$ move along the Maslov box (bottom, right, and left 
shelves respectively). Likewise, the first hyperplane index on the right-hand
side is computed by tracking $p^{+,-} (\lambda)$ around $S^1$ as 
$\lambda$ increases from  $\lambda_1$ to $\lambda_2$, and the latter two hyperplane 
indices on the right-hand side are computed by tracking (for $i = 1, 2$) $p^+ (x; \lambda_i)$ 
around $S^1$ as $x$ increases from $- \infty$ to $+\infty$. For the proof 
of Lemma \ref{equivalence-lemma}, our strategy will be to take advantage of 
Lemma \ref{technical-lemma1} to show that these points can be kept close enough 
so that the indices computed must be equivalent. 

We effectively have four cases to consider, based on the 
asymptotic limits 
\begin{equation*}
    p^{+,+} (\lambda_i) 
    = \lim_{x \to + \infty} p^+ (x; \lambda_i),
    \quad i = 1, 2,
\end{equation*}
which necessarily exist under the assumptions of Theorem 
\ref{main-theorem}. Namely, we can have (1) $p^{+,+} (\lambda_i) \ne (-1,0)$,
$i = 1, 2$; (2) $p^{+,+} (\lambda_1) \ne (-1,0)$, $p^{+,+} (\lambda_2) = (-1, 0)$;
(3) $p^{+,+} (\lambda_1) = (-1,0)$, $p^{+,+} (\lambda_2) \ne (-1, 0)$;
and (4) $p^{+,+} (\lambda_i) = (-1,0)$, $i = 1, 2$.

Beginning with Case (1), we first observe that 
we can take some $L$ sufficiently large so that for all 
$x \ge L$, we have $p^+ (x; \lambda_i) \ne (-1,0)$, $i = 1, 2$.  
We can conclude that for each of $i = 1, 2$, and for any 
$c_2 \ge L$, 
\begin{equation} \label{right-asymptotic}
    \ind (\mathpzc{g} (\cdot; \lambda_i), \mathpzc{h}^+ (\lambda_i); [c_2, +\infty])
    = 0;
\end{equation}
In addition, for Case (1) neither $\lambda_1$ nor $\lambda_2$ can be an eigenvalue,
so we must have $p^{c_2} (c_2; \lambda_i) \ne (-1,0)$, $i = 1, 2$. 

We now think of starting the evolution of $p^{c_2} (x; \lambda)$ and 
$p^+ (x; \lambda)$ at the point $(c_2, \lambda_1)$ (top left corner 
of the Maslov box). According to Lemma \ref{technical-lemma1}, given 
any $\epsilon > 0$ we can choose $L$ sufficiently large (possibly 
larger than before) so that 
for all $c_2 \ge L$
\begin{equation} \label{inequality1}
    |p^{c_2} (x; \lambda_1) - p^+ (x; \lambda_1)| < \epsilon
\end{equation}
for all $x \in \mathbb{R}$. Likewise, according to Lemma \ref{technical-lemma1},
we can choose $L$ sufficiently large 
so that for all $c_2 \ge L$
\begin{equation} \label{inequality2}
    |p^{c_2, -} (\lambda) - p^{+,-} (\lambda)| < \epsilon
\end{equation}
for all $\lambda \in [\lambda_1, \lambda_2]$. Last, we can complete 
a U-shaped contour by choosing $L$ sufficiently large so that 
for all $c_2 \ge L$
\begin{equation} \label{inequality3}
    |p^{c_2} (x; \lambda_2) - p^+ (x; \lambda_2)| < \epsilon
\end{equation}
for all $x \in \mathbb{R}$. 

At this point, we can choose $\epsilon$
sufficiently small so that the winding numbers
associated with the points $p^{c_2} (x; \lambda)$
and $p^+ (x; \lambda)$ must be the same on 
the following U-shaped contour: (1) fix $\lambda = \lambda_1$,
and let $x$ decrease from $c_2$ to $-\infty$ (asymptotic limit
sense); (2) for the asymptotic limit points $p^{c_2, -} (\lambda)$
and $p^{+, -} (\lambda)$, let $\lambda$ increase from $\lambda_1$ to $\lambda_2$; 
and (3) fix $\lambda = \lambda_2$ and let $x$ increase
from $- \infty$ to $c_2$. In total, we obtain the index relation
\begin{equation} \label{intermediate-relation}
    \begin{aligned}
     &\ind (\mathpzc{g}^- (\cdot), \mathpzc{h} (c_2; \cdot); [\lambda_1, \lambda_2])
    + \ind (\mathpzc{g} (\cdot; \lambda_2), \mathpzc{h} (c_2; \lambda_2); [-\infty, c_2]) \\
    & \quad \quad \quad \quad - \ind (\mathpzc{g} (\cdot; \lambda_1), \mathpzc{h} (c_2; \lambda_1); [-\infty, c_2]) \\
    &\quad \quad =   \ind (\mathpzc{g}^- (\cdot), \mathpzc{h}^+ (\cdot); [\lambda_1, \lambda_2])
    + \ind (\mathpzc{g} (\cdot; \lambda_2), \mathpzc{h}^+ (\lambda_2); [-\infty, c_2]) \\
    & \quad \quad \quad \quad  - \ind (\mathpzc{g} (\cdot; \lambda_1), \mathpzc{h}^+ (\lambda_1); [-\infty, c_2]).
    \end{aligned}
\end{equation}
For the indices on the right-hand side of this relation, we can use 
(\ref{right-asymptotic}) to obtain precisely the expressions stated 
on the right-hand side in Lemma \ref{equivalence-lemma}. For the 
indices on the left-hand side of this relation, we observe that 
by virtue of our invariance assumption {\bf (E)}(3) for the bottom shelf we can 
take $L$ sufficiently large so that for all $c_1 \le - L$ and 
$c_2 \ge L$, the scaled variables $\tilde{\psi}_1^{c_2} (x; \lambda)$ 
and $\tilde{\psi}_2^{c_2} (x; \lambda)$ do not simultaneously vanish at 
any point of the asymptotic rectangle 
$[-\infty, c_1] \times [\lambda_1, \lambda_2]$. Accordingly, we 
can use homotopy invariance in the Maslov-Arnold space to see that 
\begin{equation*}
    \begin{aligned}
     \ind &(\mathpzc{g}^- (\cdot), \mathpzc{h} (c_2; \cdot); [\lambda_1, \lambda_2])
     + \ind (\mathpzc{g} (\cdot; \lambda_2), \mathpzc{h} (c_2; \lambda_2); [-\infty, c_1)) \\
     &\quad \quad \quad \quad - \ind (\mathpzc{g} (\cdot; \lambda_1), \mathpzc{h} (c_2; \lambda_1); [-\infty, c_1])
     = \ind (\mathpzc{g} (c_1; \cdot), \mathpzc{h} (c_2; \cdot); [\lambda_1, \lambda_2]).
    \end{aligned}
\end{equation*}
If we think of solving this last relation for 
$\ind (\mathpzc{g}^- (\cdot), \mathpzc{h} (c_2; \cdot); [\lambda_1, \lambda_2])$ and substituting 
the result into the left-hand side of (\ref{intermediate-relation}), we see 
that the left-hand side of (\ref{intermediate-relation}) becomes
\begin{equation*}
    \begin{aligned}
    \ind (\mathpzc{g} (c_1; \cdot), \mathpzc{h} (c_2; \cdot); [\lambda_1, \lambda_2])
    + \ind (\mathpzc{g} (\cdot; \lambda_2), \mathpzc{h} (c_2; \lambda_2); [c_1, c_2])
    - \ind (\mathpzc{g} (\cdot; \lambda_1), \mathpzc{h} (c_2; \lambda_1); [c_1, c_2]),
    \end{aligned}
\end{equation*}
which is precisely the left-hand side claimed in Lemma \ref{equivalence-lemma}.

Turning to Case (2), the critical difference is that we now have
$p^{c_2} (c_2; \lambda_2) = (-1,0)$, and 
additionally
\begin{equation} \label{p-plus-plus}
    p^{+,+} (\lambda_2) 
    = \lim_{x \to + \infty} p^+ (x; \lambda_2) 
    = (-1,0).
\end{equation}
Similarly as for Case (1), we would like to think of tracking relevant 
points around $S^1$ as points $(x, \lambda)$ move along the lower 
U-shaped contour, but for Case (2) we have three subcases for the 
location of $p^+(c_2; \lambda_2)$:  
(i) $p^+ (c_2; \lambda_2)$ is rotated slightly away from 
$(-1,0)$ in the clockwise direction; or (ii) $p^+ (c_2; \lambda_2) = (-1,0)$;
or $p^+ (c_2; \lambda_2)$ is rotated slightly away from 
$(-1,0)$ in the counterclockwise direction. It follows immediately
from (\ref{p-plus-plus}) and our specification of hyperplane 
indices computed on unbounded domains that 
\begin{equation} \label{discrepancies}
    \ind (\mathpzc{g} (\cdot; \lambda_2), \mathpzc{h}^+ (\lambda_2); [-\infty, +\infty])
    =  \ind (\mathpzc{g} (\cdot; \lambda_2), \mathpzc{h}^+ (\lambda_2); [-\infty, c_2])
    + \begin{cases}
    +1 & \textrm{case (i)} \\
    0 & \textrm{cases (ii) or (iii)}. 
    \end{cases}
\end{equation}
For all cases, given any $\epsilon > 0$,
we can choose $L$ large enough so that 
(\ref{inequality1}), (\ref{inequality2}), and (\ref{inequality3}) 
all hold for all $c_2 \ge L$. It follows that the quantities on the left 
and right sides of (\ref{intermediate-relation}) can differ
only by either $0$ or $1$. We claim that the discrepancies 
are precisely as follows: the quantity obtained by subtracting 
the right-hand side of (\ref{intermediate-relation}) from the left 
must be $+1$ for Case (i) and $0$ for Cases (ii) and (iii). To 
see this, we need only consider the possible ways in which 
the points $p^{c_2} (c_2; \lambda_2)$ and $p^+ (c_2; \lambda_2)$
can arrive respectively at $p^{c_2} (c_2; \lambda_2) = (-1, 0)$
and the Case (i) location of $p^+ (c_2; \lambda_2)$. Ignoring 
transient crossings (i.e., crossings along with return crossings),
either $p^{c_2} (c_2; \lambda_2)$ has arrived at $(-1,0)$ one
more time moving in the counterclockwise direction than 
$p^+ (c_2; \lambda_2)$, or $p^+ (c_2; \lambda_2)$ has crossed 
$(-1,0)$ once more in the clockwise direction than $p^+ (c_2; \lambda_2)$.
In either case, the discrepancy is $+1$. The reasoning is similar
for Cases (ii) and (iii), and we see that the discrepancies are 
precisely the values on the right-hand side of (\ref{discrepancies}). 
In this way, we see that the possible discrepancy between the 
left and right sides of (\ref{intermediate-relation}) are precisely
addressed in Case (2) by the inclusion of $+ \infty$ in the 
index 
\begin{equation*}
    \ind (\mathpzc{g} (\cdot; \lambda_2), \mathpzc{h}^+ (\lambda_2); [-\infty, +\infty]),
\end{equation*}
and the equality of Lemma \ref{equivalence-lemma} is seen to hold. 

Cases (3) and (4) can be handled similarly.
\end{proof}

One thing clear from Lemma \ref{equivalence-lemma} is that the left-hand side
of the stated relation takes the same value for all $c_1 \le -L$ and $c_2 \ge L$ (because
the right-hand side is independent of $c_1$ and $c_2$). 
In order to show from (\ref{mathfrak-m-defined}) that $\mathfrak{m} (c_1, c_2)$ 
is actually independent of $c_1$ and $c_2$, we last need to establish that 
the top-shelf index
\begin{equation*}
\ind (\mathpzc{g} (c_2; \cdot), \mathpzc{h} (c_2; \cdot); [\lambda_1, \lambda_2])    
\end{equation*}
is independent of $c_2$ (for $c_2$ taken sufficiently large). To see this, 
we first recall that crossing points along the top shelf correspond precisely
with eigenvalues of (\ref{nonhammy}) and so certainly do not depend on 
$c_2$. This means we only need to verify that the directions we associate 
with these crossings must be independent of $c_2$. 

If the eigenvalues in $[\lambda_1, \lambda_2]$ aren't discrete, we set 
$\mathcal{N}_{\#} ([\lambda_1, \lambda_2]) = \infty$, in which case 
the claim of Theorem \ref{main-theorem} holds by convention, so we can 
restrict our analysis to the case of discrete spectrum. For this, 
we suppose $\lambda_0 \in (\lambda_1, \lambda_2)$ is an isolated
eigenvalue and we consider a Maslov box sufficiently small so that $\lambda_0$
is the only eigenvalue it contains. (See Figure \ref{micro-box-figure}.)
In particular, we compute the hyperplane index detecting 
intersections between $\mathpzc{g} (x; \lambda)$ and $\mathpzc{h} (x; \lambda)$ 
as this small box is traversed. According to 
our Assumption {\bf (E)}(4), we have invariance along this Maslov box
and its interior, so we can use homotopy invariance in the associated
Maslov-Arnold space to conclude that 
\begin{equation*}
\begin{aligned}
\ind &(\mathpzc{g} (c_2; \cdot), \mathpzc{h} (c_2; \cdot); [\lambda_0 - h, \lambda_0 +h]) \\
& \quad - \ind (\mathpzc{g} (c_2 + \Delta c_2; \cdot), \mathpzc{h} (c_2 + \Delta c_2; \cdot); [\lambda_0 - h, \lambda_0 +h])
= 0.
\end{aligned}
\end{equation*}
Since only a single crossing occurs for each of these hyperplane indices, it must be the 
case that these crossings are in the same direction. This discussion has
been for interior values $\lambda_0 \in (\lambda_1, \lambda_2)$, but the 
endpoints $\lambda_0 = \lambda_1$ and $\lambda_0 = \lambda_2$ can be 
treated similarly if $\lambda_1$ or $\lambda_2$ is an eigenvalue, with 
boxes either extending to the right of $\lambda_1$ or to the left of 
$\lambda_2$. 

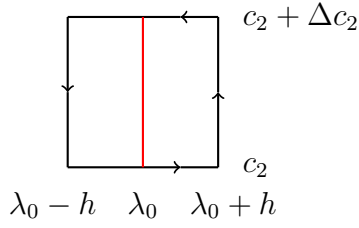
\begin{figure}[ht]
\begin{center}
\begin{tikzpicture}

\draw[thick, ->] (-1,-1) -- (.5,-1);
\draw[thick] (.5,-1) -- (1,-1);	
\draw[thick, ->] (1,-1) -- (1,0);
\draw[thick] (1,0) -- (1, 1);
\draw[thick,->] (1,1) -- (.5, 1);
\draw[thick] (.5,1) -- (-1, 1);
\draw[thick,->] (-1,1) -- (-1, 0);
\draw[thick] (-1,0) -- (-1, -1);
\draw[thick, red] (0,-1) -- (0,1);
%
\node at (0,-1.5) {$\lambda_0$};
\node at (-1.2,-1.5) {$\lambda_0 - h$};
\node at (1.2,-1.5) {$\lambda_0 + h$};
\node at (1.5,-1) {$c_2$};
\node at (2.1,1) {$c_2 + \Delta c_2$};
\end{tikzpicture}
\end{center}
\caption{The Maslov Box at $\lambda_0$.} \label{micro-box-figure}
\end{figure}

These considerations allow us to associate a value $\mathfrak{m}$ 
with any interval $[\lambda_1, \lambda_2] \subset I$ for which 
the assumptions of Theorem \ref{main-theorem} apply. In particular, 
by rearranging (\ref{mathfrak-m-defined}), we can write 
\begin{equation*} 
\begin{aligned}
\ind &(\mathpzc{g} (c_2; \cdot), \mathpzc{h} (c_2; \cdot); [\lambda_1, \lambda_2])
=  \ind (\mathpzc{g} (\cdot; \lambda_2), \mathpzc{h} (c_2; \lambda_2); [c_1, c_2]) \\
& - \ind (\mathpzc{g} (\cdot; \lambda_1), \mathpzc{h} (c_2; \lambda_1); [c_1, c_2]) 
+ \ind (\mathpzc{g} (c_1; \cdot), \mathpzc{h} (c_2; \cdot); [\lambda_1, \lambda_2])
- \mathfrak{m}.
\end{aligned}
\end{equation*}
Upon substitution of the relation from Lemma \ref{equivalence-lemma}, we obtain 
\begin{equation} \label{penultimate-relation}
\begin{aligned}
\ind &(\mathpzc{g} (c_2; \cdot), \mathpzc{h} (c_2; \cdot); [\lambda_1, \lambda_2])
=  \ind (\mathpzc{g} (\cdot; \lambda_2), \mathpzc{h}^+ (\lambda_2); [-\infty, + \infty]) \\
& - \ind (\mathpzc{g} (\cdot; \lambda_1), \mathpzc{h}^+ (\lambda_1); [-\infty, +\infty]) 
+ \ind (\mathpzc{g}^- (\cdot), \mathpzc{h}^+ (\cdot); [\lambda_1, \lambda_2])
- \mathfrak{m}.
\end{aligned}
\end{equation}
Last, using (\ref{count-inequality}), we obtain the claimed inequality 
\begin{equation*} 
\begin{aligned}
\mathcal{N}_{\#} &([\lambda_1, \lambda_2])
\ge  \Big| \ind (\mathpzc{g} (\cdot; \lambda_2), \mathpzc{h}^+ (\lambda_2); [-\infty, + \infty]) \\
& - \ind (\mathpzc{g} (\cdot; \lambda_1), \mathpzc{h}^+ (\lambda_1); [-\infty, +\infty]) 
+ \ind (\mathpzc{g}^- (\cdot), \mathpzc{h}^+ (\cdot); [\lambda_1, \lambda_2])
- \mathfrak{m} \Big|.
\end{aligned}
\end{equation*}
This completes the proof of Theorem \ref{main-theorem}.

\subsection{The Bottom and Left Shelves}
\label{bottom-and-left}

In this section, we provide additional information about 
calculations associated with the bottom and left shelves, emphasizing 
cases in which one or both of these values can be shown to be 0. 

\subsubsection{The Bottom Shelf}
\label{bottom-shelf-section}

In Theorem \ref{main-theorem}, the hyperplane 
index associated with the bottom 
shelf detects intersections between the spaces $\mathpzc{g}^- (\lambda)$ 
and $\mathpzc{h}^+ (\lambda)$ as $\lambda$ increases from $\lambda_1$
to $\lambda_2$. These intersections can be detected as zeros of the 
wedge product 
\begin{equation*}
    v^- (\lambda) \wedge \tilde{\mathcal{V}}^+ (\lambda),
\end{equation*}
where we recall from Proposition \ref{ode-proposition1} 
that $v^- (\lambda)$ denotes a right eigenvector 
of $A_- (\lambda)$ associated with the eigenvalue $\mu_- (\lambda)$, 
and we recall from Proposition \ref{ode-proposition2} 
that $\tilde{\mathcal{V}}^+ (\lambda)$ denotes a right eigenvector 
of $\tilde{A}_+ (\lambda)$ associated with the eigenvalue $-\mu_+ (\lambda)$. 
In general, this allows us to explicitly compute the hyperplane index
along the bottom shelf by working with the scaled variables 
$\tilde{\psi}_1^{+, -} (\lambda)$ and $\tilde{\psi}_2^{+, -} (\lambda)$,
defined respectively in (\ref{psi1-tilde-plus-minus-defined}) 
and (\ref{psi2-tilde-plus-minus-defined}). 
In some cases, including the applications we consider in Section 
\ref{applications-section}, we can show that for each $\lambda \in [\lambda_1, \lambda_2]$,
we have $v^- (\lambda) \wedge \tilde{\mathcal{V}}^+ (\lambda) \ne 0$, from which we can 
conclude that there are no crossing points along the bottom shelf. In such 
cases, we naturally have invariance along the bottom shelf, and additionally
\begin{equation*}
    \ind (\mathpzc{g}^- (\cdot), \mathpzc{h}^+ (\cdot); [\lambda_1, \lambda_2])
    = 0.
\end{equation*}

\subsubsection{The Left Shelf}
\label{left-shelf-section}

In most applications, we expect to take $\lambda_1$ sufficiently 
negative so that it is not an eigenvalue of (\ref{nonhammy}), and 
in such cases it's straightforward to verify that the limit 
\begin{equation*}
    \tilde{\psi}^{+, +} (\lambda_1)
    = \lim_{x \to +\infty} \tilde{\psi}^+ (x; \lambda_1)
\end{equation*}
is well defined. Precisely, we can prove the following lemma. 

\begin{lemma} \label{left-shelf-lemma1}
Let Assumptions {\bf (A)} through {\bf (D)} hold, 
and fix any $\lambda_0 \in [\lambda_1, \lambda_2]$
that is not an eigenvalue of (\ref{nonhammy}). Then 
there exists a constant $c_+ (\lambda_0) \ne 0$ so that 
\begin{equation*}
    \lim_{x \to + \infty} e^{- \mu_+ (\lambda_0) x} \eta^- (x; \lambda_0)
    = c_+ (\lambda_0) v^+ (\lambda_0),
\end{equation*}
where $v^+ (\lambda_0)$ is the eigenvector of $A_+ (\lambda_0)$ associated 
with the eigenvalue $\mu_+ (\lambda_0)$. It follows that 
\begin{equation*}
\begin{aligned}
    \tilde{\psi}_1^{+, +} (\lambda_0)
    &:= \lim_{x \to + \infty} \tilde{\psi}_1^{+} (x; \lambda_0)
    = \frac{v^+ (\lambda_0) \wedge \tilde{\mathcal{V}}^+ (\lambda_0)}{|v^+ (\lambda_0)| |\tilde{\mathcal{V}}^+ (\lambda_0)|} \\
     \tilde{\psi}_2^{+, +} (\lambda_0)
    &:= \lim_{x \to + \infty} \tilde{\psi}_2^{+} (x; \lambda_0)
    = \frac{v^+ (\lambda_0) \wedge \tilde{\mathcal{V}}_M^+ (\lambda_0)}{|v^+ (\lambda_0)| |\tilde{\mathcal{V}}^+ (\lambda_0)|}.
\end{aligned}
\end{equation*}
\end{lemma}

\begin{proof}
As in the discussion of (\ref{v-plus-specified}), we let $\{y_j^+ (x; \lambda_0)\}_{j=1}^n$
denote a basis of linearly independent solutions of (\ref{nonhammy}) indexed so that 
(\ref{v-plus-specified}) holds. Then there exist constants $\{c_j (\lambda_0)\}_{j=1}^n$
so that 
\begin{equation} \label{left-right-expansion}
    \eta^- (x; \lambda_0)
    = \sum_{j=1}^n c_j (\lambda_0) y_j^+ (x; \lambda_0),
    \quad i = 1, 2,
\end{equation}
where we must have $c_n (\lambda_0) \ne 0$ or $\lambda_0$ would be 
an eigenvalue of (\ref{nonhammy}) (in the sense described in the 
introduction). According to Assumption {\bf (C)}, along with our labeling convention, 
the solutions $\{y_j^+ (x; \lambda)\}_{j=1}^{n-1}$ are all 
$\mathbf{o} (e^{\mu_+ (\lambda_0) x})$ as $x \to + \infty$, so 
\begin{equation*}
\begin{aligned}
\lim_{x \to + \infty} e^{- \mu_+ (\lambda_0) x} \eta^- (x; \lambda_0)
& = \sum_{j=1}^n c_j (\lambda_0) \lim_{x \to + \infty} e^{- \mu_+ (\lambda_0) x}  y_j^+ (x; \lambda_0) \\
&= c_n (\lambda_0) \lim_{x \to + \infty} e^{- \mu_+ (\lambda_0) x}  y_n^+ (x; \lambda_0)
= c_n (\lambda_0) v^+ (\lambda).
\end{aligned}
\end{equation*}
The first claim of our lemma follows from denoting $c_n (\lambda_0)$ by $c_+ (\lambda_0)$. 

For the second claim, we can write 
\begin{equation*}
    \tilde{\psi}_1^+ (x; \lambda_0)
    = \frac{(e^{-\mu_+ (\lambda_0) x} \eta^- (x; \lambda_0)) \wedge \tilde{\mathcal{V}}^+ (\lambda)}
    {|e^{-\mu_+ (\lambda_0) x} \eta^- (x; \lambda_0)| |\tilde{\mathcal{V}}^+ (\lambda)|}, 
\end{equation*}
from which the claim regarding $\tilde{\psi}_1^{+,+} (\lambda_0)$ is clear upon taking $x \to + \infty$. 
The claim about $\tilde{\psi}_2^{+,+} (\lambda_0)$ follows similarly. 
\end{proof}

In some cases, including the applications we consider 
in Section \ref{applications-section}, we can take $\lambda_1$ sufficiently 
negative so that there are no crossings along the left shelf. To understand
conditions under which this occurs, we recall that in the left-shelf 
computation 
\begin{equation*}
    \ind (\mathpzc{g} (\cdot; \lambda_1), \mathpzc{h}^+ (\lambda_1); [-\infty, +\infty]), 
\end{equation*}
we detect intersections between the evolving subspaces $\mathpzc{g} (x; \lambda_1)$
and the fixed target space $\mathpzc{h}^+ (\lambda_1)$. If we can show that 
\begin{equation*}
    \eta^- (x; \lambda_1) \wedge \tilde{\mathcal{V}}^+ (\lambda_1) \ne 0
\end{equation*}
for all $x \in \mathbb{R}$, and additionally that no intersection is obtained 
in the limit as $x \to \pm \infty$, then we can conclude that there are no crossings
on the left shelf. For this discussion, we will lean heavily on the development 
of \cite{PW1992}, especially the statement and proof of Proposition 1.17 from 
that reference. 

\begin{proposition} \label{left-shelf-proposition}
Let Assumptions {\bf (A)} through {\bf (D)} hold,  
with the interval $I$ unbounded on the left.  
For the matrices $A_{\pm} (\lambda)$ specified in Assumption {\bf (B)},
suppose $A_- (\lambda) = A_+ (\lambda)$, and that 
there exists a diagonalizing matrix $V_- (\lambda)$, so that
for each $\lambda \in I$ the matrix
\begin{equation*}
    D_- (\lambda) := V_- (\lambda)^{-1} A_- (\lambda) V_- (\lambda) 
\end{equation*}
is diagonal with the eigenvalues of $A_- (\lambda)$ on its diagonal, 
and in particular with $\mu_- (\lambda)$ in the first column of the first row. 
In addition, we set 
\begin{equation*}
    F(x; \lambda) := V_- (\lambda)^{-1} (A (x; \lambda) - A_- (\lambda)) V_- (\lambda),
\end{equation*}
and assume the following three items: 

\medskip
(i) There exists a positive constant $\lambda_{\infty}$
and a corresponding constant $C_{\infty}$ so that 
for all $\lambda \le - \lambda_{\infty}$ we have
the inequality 
\begin{equation*}
    \int_{-\infty}^{+ \infty} |F (x; \lambda)| dx \le C_{\infty}.
\end{equation*}

\medskip
(ii) The limit 
\begin{equation*}
    \lim_{x_0 \to - \infty} \int_{-\infty}^{x_0} |F (x; \lambda)| dx = 0
\end{equation*}
converges uniformly for all $\lambda \le - \lambda_{\infty}$.

\medskip
(iii) If $e_1$ denotes the usual first Euclidian basis element, then
\begin{equation*}
    \lim_{\lambda \to - \infty} \int_{- \infty}^{+\infty} |F (x; \lambda) e_1| dx 
    = 0.
\end{equation*}
Under these assumptions, we can conclude 
\begin{equation*}
 V_- (\lambda)^{-1} (e^{- \mu_- (\lambda) x} \eta^- (x; \lambda)) 
 = e_1 + \mathbf{o} (1), 
 \quad \lambda \to -\infty,
\end{equation*}
and likewise if $z^+ (x; \lambda)$ is the solution to 
$z^{+\,\prime} = - z^+ A (x; \lambda)$ associated to 
$\tilde{\mathcal{Y}}^+ (x; \lambda)$ via
(\ref{proposition5relation}), then
\begin{equation*}
e^{\mu_+ (\lambda) x} z^+ (x; \lambda) V_+ (\lambda) e_1
= 1 + \mathbf{o} (1), \quad \lambda \to - \infty,
\end{equation*}
where in both cases the order term is uniform for $x \in \mathbb{R}$.
In addition, the vector $z^+ (0; \lambda) V_+ (\lambda)$ remains
bounded as $\lambda$ tends toward $-\infty$. 
\end{proposition}

\begin{proof}
We begin by looking for solutions to (\ref{nonhammy}) of the form 
\begin{equation*}
    y(x; \lambda) = e^{\mu_- (\lambda) x} V_- (\lambda) (e_1 + v (x; \lambda)) 
\end{equation*}
for which we find by direct calculation that 
\begin{equation*}
    v' (x; \lambda) 
    = V_- (\lambda)^{-1} (A_- (\lambda) - \mu_- (\lambda)I) V_- (\lambda) (e_1 + v(x; \lambda))
    + F(x; \lambda) (e_1 + v(x; \lambda)).
\end{equation*}
For notational convenience, we set 
\begin{equation*}
    B_- (\lambda)
    := V_- (\lambda)^{-1} (A_- (\lambda) - \mu_- (\lambda) I) V_- (\lambda),
\end{equation*}
and we observe that according to our convention with $V_- (\lambda)$ 
we have $B_- (\lambda) e_1 = 0$. This allows us to express $v'(x; \lambda)$
in the more compact form 
\begin{equation} \label{v-ode}
    v' 
    = B_- (\lambda) v + F (x; \lambda) (v + e_1).
\end{equation}
Our immediate goal is to use this last relation to obtain 
a convenient expression for $v (x; \lambda)$, but first 
we determine the asymptotic behavior of $v (x; \lambda)$ 
as $x$ tends toward $-\infty$. Our particular 
interest is the function $v$ associated with $\eta^- (x; \lambda)$,
which we take to be defined by the relation
\begin{equation} \label{eta-minus-defined}
    \eta^- (x; \lambda) 
    = e^{\mu_- (\lambda) x} V_- (\lambda) (v (x; \lambda) + e_1).
\end{equation}
Solving for $v(x; \lambda)$, we find 
\begin{equation*}
    v(x; \lambda)
    = - e_1 + e^{- \mu_- (\lambda) x} V_- (\lambda)^{-1} \eta^- (x; \lambda).
\end{equation*}
According to Proposition \ref{ode-proposition1}, we have the convergence
\begin{equation*}
    \lim_{x \to - \infty} e^{- \mu_- (\lambda) x} \eta^- (x; \lambda)
    = v^- (\lambda),
\end{equation*}
where $v^- (\lambda)$ is the first column of $V_- (\lambda)$
(the eigenvector of $A_- (\lambda)$ corresponding with eigenvalue $\mu_- (\lambda)$). 
It follows that $V_- (\lambda)^{-1} v^- (\lambda) = e_1$, and consequently
$v(x; \lambda) \to 0$ as $x \to -\infty$. 

Returning now to (\ref{v-ode}), we note that we can express the equation as 
\begin{equation*}
    (e^{- B_- (\lambda) x}v)'
    = e^{- B_- (\lambda) x} F (x; \lambda) (v + e_1). 
\end{equation*}
Since $\mu_- (\lambda)$ is the largest eigenvalue of $A_- (\lambda)$,
the eigenvalues of $B_- (\lambda)$ must all be non-positive. Observing that $v$
is bounded on any interval $[x_0, x]$, $x_0 < x$, we see that we can integrate on 
an arbitrary such interval to obtain the relation 
\begin{equation} \label{return}
    e^{- B_- (\lambda) x} v(x; \lambda) - e^{- B_- (\lambda) x_0} v (x_0; \lambda)
    = \int_{x_0}^x e^{- B_- (\lambda) \xi} F(\xi; \lambda) (v(\xi; \lambda) + e_1) d \xi.
\end{equation}
The exponential $e^{- B_- (\lambda) x_0}$ remains bounded as $x_0 \to -\infty$, and 
$v (x_0) \to 0$ as $x_0 \to -\infty$. Additionally since $F(\xi; \lambda)$
is assumed to be integrable on $\mathbb{R}$ (Assumption (i)), we can take a limit with 
$x_0 \to - \infty$ on both sides of this last expression to see that 
\begin{equation*}
    e^{- B_- (\lambda) x} v(x; \lambda)
    = \int_{-\infty}^x e^{- B_- (\lambda) \xi} F(\xi; \lambda) (v(\xi; \lambda) + e_1) d \xi,
\end{equation*}
or equivalently 
\begin{equation*}
    v(x; \lambda)
    = \int_{-\infty}^x e^{B_- (\lambda) (x-\xi)} F(\xi; \lambda) (v(\xi; \lambda) + e_1) d \xi.
\end{equation*} 
By construction, $B_- (\lambda)$ is a diagonal matrix, and it is easily seen that 
the matrix norm of $\operatorname{exp} (B_- (\lambda) (x - \xi))$ is 1.  This 
allows us to write 
\begin{equation*}
    \begin{aligned}
    |v (x; \lambda)| 
    &\le \int_{-\infty}^x |e^{B_- (\lambda) (x-\xi)}| |F(\xi; \lambda) (v(\xi; \lambda) + e_1)| d \xi \\
    &= \int_{-\infty}^x |F(\xi; \lambda) (v(\xi; \lambda) + e_1)| d \xi \\
    &\le \int_{-\infty}^x |F(\xi; \lambda) v(\xi; \lambda)| d \xi
    + \int_{-\infty}^x |F(\xi; \lambda) e_1| d \xi \\
    &\le \sup_{\xi \le x} |v(\xi; \lambda)| \int_{-\infty}^x |F(\xi; \lambda)| d \xi
    + \int_{-\infty}^x |F(\xi; \lambda) e_1| d \xi.
    \end{aligned}
\end{equation*}
Using Assumption (ii), we can take $x_0$ sufficiently negative so that for all 
$x \le x_0$, 
\begin{equation*}
 |v (x; \lambda)| 
 \le \frac{1}{2} \sup_{\xi \le x_0} |v(\xi; \lambda)|
 + \int_{-\infty}^{x_0} |F(\xi; \lambda) e_1| d \xi.
\end{equation*}
We can now take a supremum on both sides over $x \le x_0$ 
to see that 
\begin{equation*}
    \frac{1}{2} \sup_{\xi \le x_0} |v(\xi; \lambda)|
    \le \int_{-\infty}^{x_0} |F(\xi; \lambda) e_1| d \xi. 
\end{equation*}
Using Assumption (iii), we see that the right-hand side 
of this last expression tends to 0 as $\lambda \to -\infty$,
so for some fixed $x_0 \ll 0$, we can assert that 
\begin{equation*}
    \lim_{\lambda \to -\infty} \sup_{\xi \le x_0} |v(\xi; \lambda)|
    = 0.
\end{equation*}
We now fix $x_0$ as such a value and return to (\ref{return})
to see that 
\begin{equation*}
    |v (x; \lambda)|
    \le \Big( |v(x_0; \lambda)| + \int_{x_0}^x |F (\xi; \lambda) e_1| d \xi \Big)
    + \int_{x_0}^x |F(\xi; \lambda)||v(\xi; \lambda)| d\xi.
\end{equation*}
If we apply Gr\"onwall's inequality to this integral relation, 
we obtain the inequality 
\begin{equation} \label{theplace}
\begin{aligned}
 |v (x; \lambda)|
 &\le \Big( |v(x_0; \lambda)| + \int_{x_0}^x |F (\xi; \lambda) e_1| d \xi \Big) e^{\int_{x_0}^x |F (\xi; \lambda)| d \xi} \\
 &\le \Big( |v(x_0; \lambda)| + \int_{x_0}^{\infty} |F (\xi; \lambda) e_1| d \xi \Big) e^{\int_{x_0}^{\infty} |F (\xi; \lambda)| d \xi}. 
 \end{aligned}
\end{equation}
Here, the quantity in large parentheses tends to 0 as $\lambda \to -\infty$,
so we can choose $\lambda_{\infty} \gg 0$
sufficiently large so that $v(x; \lambda)$ is small for all $x \in \mathbb{R}$, 
$\lambda \le -\lambda_{\infty}$. 

Last, upon multiplication of (\ref{eta-minus-defined}) on the left 
by $V_- (\lambda)^{-1} e^{- \mu_- (\lambda) x}$, we obtain the relation 
\begin{equation*}
    V_- (\lambda)^{-1} (e^{- \mu_- (\lambda) x} \eta^- (x; \lambda)) 
    = e_1 +  v (x; \lambda).
\end{equation*}
We have seen that $v (x; \lambda) = \mathbf{o} (1)$, $\lambda \to - \infty$
uniformly for $x \in \mathbb{R}$, giving the first claim.

The second claim is proven similarly, combining the second part of 
the proof of Proposition 1.17 from \cite{PW1992} with the first part 
of the current proof. Here, we primarily just indicate how the 
final observation on boundedness of $z^+ (0; \lambda) V_+ (\lambda)$
is established. The proof in this case begins by setting 
\begin{equation*}
    w (x; \lambda) := - e_1^T + e^{\mu_+ (\lambda) x} z^+ (x; \lambda) V_+ (\lambda).
\end{equation*}
Then, proceeding as in the first part of this proof, we can establish 
that there exists a constant $C$ sufficiently large so that for all 
$\lambda$ sufficiently negative we have the bound 
\begin{equation*}
    |w (x; \lambda)| \le C, \quad \forall \,\, x \in \mathbb{R}. 
\end{equation*}
In particular, $w (0; \lambda)$ satisfies this estimate, and we 
have the relation  
\begin{equation*}
z^+ (0; \lambda) V_+ (\lambda) 
= e_1^T + w (0; \lambda),
\end{equation*}
verifying that the left-hand side is bounded as $\lambda \to - \infty$.
\end{proof}

\begin{remark} \label{vminus-vplus-remark}
The only place in the proof of Proposition \ref{left-shelf-proposition}
in which we absolutely require the assumption that $A_- = A_+$ is 
in obtaining the second inequality in (\ref{theplace}). Nonetheless, 
it is critical at that point, and for the general case of $A_- \ne A_+$,
verifying the absence of crossings on the left shelf is more 
delicate. (See the appendix of this paper for one example.)
\end{remark}

We can now use Proposition \ref{left-shelf-proposition} to 
establish the following result on crossings along the left
shelf. 

\begin{proposition} \label{left-shelf-crossings}
Suppose the assumptions of Proposition \ref{left-shelf-proposition}
hold. Then there exists a value $\lambda_{\infty} > 0$ sufficiently 
large so that for all $\lambda_1 \le - \lambda_{\infty}$ there holds
\begin{equation*}
\mathpzc{g} (x; \lambda_1) \cap \mathpzc{h}^+ (\lambda_1) = \{0\} 
\end{equation*}
for all $x \in \mathbb{R}$, and in particular, 
\begin{equation*}
    \ind (\mathpzc{g} (\cdot; \lambda_1), \mathpzc{h}^+ (\lambda_1), [- \infty, + \infty]) = 0.
\end{equation*}
\end{proposition}

\begin{proof}
First, crossings along the left shelf correspond precisely with zeros of 
\begin{equation*}
(e^{- \mu_- (\lambda) x} \eta^- (x; \lambda)) \wedge \tilde{\mathcal{V}}^+ (\lambda),  
\end{equation*}
which by virtue of Proposition \ref{connection-to-PW} can be 
expressed as 
\begin{equation*}
    w^+ (\lambda) \cdot (e^{- \mu_- (\lambda) x} \eta^- (x; \lambda)),
\end{equation*}
where $w^+ (\lambda)$ is related to $\tilde{\mathcal{V}}^+ (\lambda)$
as in \ref{corresponding-eigenvector2}. 

According to Proposition \ref{left-shelf-crossings}, we
have the asymptotic relation
\begin{equation*}
V_- (\lambda)^{-1} (e^{- \mu_- (\lambda) x} \eta^- (x; \lambda)) 
 = e_1 + \mathbf{o} (1), 
 \quad \lambda \to -\infty,
\end{equation*}
uniformly for $x \in \mathbb{R}$, and according to 
the normalization in Assumption {\bf (C)}, along with our 
assumption that $A_+ (\lambda)$ is diagonalizable, we 
can write 
\begin{equation*}
    w^+ (\lambda) V_+ (\lambda) = e_1^T.   
\end{equation*}
Under the assumptions of Proposition \ref{left-shelf-proposition}, 
we have $V_- (\lambda) = V_+ (\lambda)$, allowing 
us to compute 
\begin{equation*}
\begin{aligned}
    w^+ (\lambda) \cdot (e^{- \mu_- (\lambda) x} \eta^- (x; \lambda))
    &= e_1^T V_- (\lambda)^{-1} (e^{- \mu_- (\lambda) x} \eta^- (x; \lambda)) \\
    &= e_1^T (e_1 + \mathbf{o} (1)) 
    = 1 + \mathbf{o} (1).
\end{aligned}
\end{equation*}
We see immediately that for $\lambda$ sufficiently negative we must have 
\begin{equation*}
(e^{- \mu_- (\lambda) x} \eta^- (x; \lambda)) \wedge \tilde{\mathcal{V}}^+ (\lambda)
> 0,  
\end{equation*}
for all $x \in \mathbb{R}$, indicating that there are no crossings 
along the left shelf. 
\end{proof}

\section{The Evans Function}
\label{evans-section}

One of our goals is to place information gained from the Evans 
function into the broader geometrical framework of the current 
analysis. The Evans function in this setting has already 
been elegantly developed in \cite{PW1992}, so we proceed primarily 
by translating the results obtained there into the current setting. 

Generally, the Evans function serves as a characteristic function
for eigenvalue problems such as (\ref{nonhammy}) (standard 
references include \cite{AGJ1990, Evans1972a,
Evans1972b, Evans1972c, Evans1975, GZ1998}, along with 
\cite{PW1992}). For (\ref{nonhammy})
under Assumptions {\bf (A)} through {\bf (D)}, 
it's natural to specify the Evans function as the wedge product
\begin{equation} \label{evans-function-defined}
    D(\lambda) := \eta^- (x; \lambda) \wedge \tilde{\mathcal{Y}}^+ (x;\lambda),
\end{equation}
where Proposition \ref{conjugation-proposition} allows us to verify 
that the right-hand side is independent of $x$ (i.e., its 
derivative with respect to $x$ is 0).  

\begin{remark} \label{equivalence-pw1992}
According to Proposition \ref{connection-to-PW}, our specification 
(\ref{evans-function-defined}) is equivalent to the specification 
from \cite{PW1992}, which 
with $z$ as in Proposition \ref{connection-to-PW} can be expressed 
here as 
\begin{equation*}
    D_{\textrm{PW}} (\lambda) = z(0; \lambda) \cdot \eta^- (0; \lambda).
\end{equation*}
This correspondence between $D (\lambda)$ and $ D_{\textrm{PW}} (\lambda)$
allows us to adapt results
from \cite{PW1992} directly to the current setting, though we include some 
details of the proofs for completeness. 
\end{remark}

\begin{proposition} \label{evans-proposition}
Let Assumptions {\bf (A)} through {\bf (D)} hold, and let the 
Evans function $D (\lambda)$ be specified as in (\ref{evans-function-defined}). 
Then for any $\lambda \in I$, 
\begin{equation*}
    \begin{aligned}
    D'(\lambda) &= \Big(\frac{v^{-\,\prime} (\lambda) \wedge \tilde{\mathcal{V}}^- (\lambda)}
    {v^- (\lambda) \wedge \tilde{\mathcal{V}}^- (\lambda)} 
    + \frac{v^+ (\lambda) \wedge \tilde{\mathcal{V}}^{+\,\prime} (\lambda)}
    {v^+ (\lambda) \wedge \tilde{\mathcal{V}}^+ (\lambda)} \Big) D(\lambda)  \\
    &\quad + \int_{-\infty}^0 \Big((A_{\lambda} (x; \lambda) - \mu_-'(\lambda) I) \eta^- (x; \lambda)\Big) 
    \wedge \tilde{\mathcal{Y}}^+ (x; \lambda) dx  \\
    &\quad \quad + \int_0^{+\infty} \Big((A_{\lambda} (x; \lambda) - \mu_+'(\lambda) I) \eta^- (x; \lambda)\Big) 
    \wedge \tilde{\mathcal{Y}}^+ (x; \lambda) dx.
    \end{aligned}
\end{equation*}
Here $v^- (\lambda)$ is an eigenvector of $A_- (\lambda)$ corresponding to the eigenvalue 
$\mu_- (\lambda)$; $\tilde{\mathcal{V}}^- (\lambda)$ is an eigenvector of $\tilde{A}_- (\lambda)$
corresponding to the eigenvalue $- \mu_- (\lambda)$; $v^+ (\lambda)$ is an eigenvector 
of $A_+ (\lambda)$ corresponding to the eigenvalue $\mu_+ (\lambda)$; and $\tilde{\mathcal{V}}^+ (\lambda)$
is an eigenvector of $\tilde{A}_+ (\lambda)$ corresponding to the eigenvalue $- \mu_+ (\lambda)$. 
\end{proposition}

\begin{proof}
Our Proposition \ref{evans-proposition} is effectively a restatement of Theorem 1.11 from 
\cite{PW1992} in the current setting, and we only briefly sketch the proof, also adapted 
to our setting. First, for $x < 0$, it's useful to write 
\begin{equation*}
    D (\lambda) 
    = u^- (x; \lambda) \wedge \tilde{\mathcal{U}}^- (x; \lambda),
\end{equation*}
where we're introducing the notation 
\begin{equation*}
    \begin{aligned}
    u^- (x; \lambda) &:= e^{- \mu_- (\lambda) x} \eta^- (x; \lambda) \\
    \tilde{\mathcal{U}}^- (x; \lambda) &= e^{+ \mu_- (\lambda) x} \tilde{\mathcal{Y}}^+ (x; \lambda).
    \end{aligned}
\end{equation*}
Recalling that $\eta^{-\,\prime} = A \eta^-$ and 
$\tilde{\mathcal{Y}}^{+\,\prime} = \tilde{A} \tilde{\mathcal{Y}}^+$,
we can write 
\begin{equation*}
    \begin{aligned}
    u^{- \,\prime} &= -\mu_- u^- + Au \\
    \tilde{\mathcal{U}}^{-\,\prime} &= \mu_- \tilde{\mathcal{U}}^- + \tilde{A}\tilde{\mathcal{U}}^-.
    \end{aligned}
\end{equation*}
Taking a $\lambda$-derivative of these expressions, we find 
\begin{equation*}
    \begin{aligned}
    u_{\lambda}^{- \,\prime} 
    &= -\mu_-' u^- - \mu_- u_{\lambda}^- + A_{\lambda} u^- + Au_{\lambda}^- \\ 
    \tilde{\mathcal{U}}_{\lambda}^{- \,\prime} 
    &= \mu_-' \tilde{\mathcal{U}}^- + \mu_- \tilde{\mathcal{U}}_{\lambda}^- 
    + \tilde{A}_{\lambda} \tilde{\mathcal{U}}^- + \tilde{A} \tilde{\mathcal{U}}_{\lambda}^-. 
    \end{aligned}
\end{equation*}
We now differentiate $u_{\lambda}^- \wedge \tilde{\mathcal{U}}^-$ in $x$, 
\begin{equation*}
\begin{aligned}
    \frac{d}{dx} (u_{\lambda}^- \wedge \tilde{\mathcal{U}}^-)
     &= \Big((A_{\lambda} - \mu_-' I)u^-\Big) \wedge \tilde{\mathcal{U}}^-
     + (A u_{\lambda}^-) \wedge \tilde{\mathcal{U}}^- + u_{\lambda}^- \wedge (\tilde{A} \tilde{\mathcal{U}}^-) \\
     &= \Big((A_{\lambda} - \mu_-' I)u^-\Big) \wedge \tilde{\mathcal{U}}^-,
\end{aligned}
\end{equation*}
where the second equality follows from Proposition \ref{conjugation-proposition}. 
For $R > 0$, we can now integrate on $(-R,0)$ to obtain the 
relation
\begin{equation} \label{uminus-relation}
\begin{aligned}
    u_{\lambda}^-& (0; \lambda) \wedge \tilde{\mathcal{U}}^- (0; \lambda)
    - u_{\lambda}^- (-R; \lambda) \wedge \tilde{\mathcal{U}}^- (-R; \lambda) \\
    &= \int_{-R}^0 \Big((A_{\lambda} (x; \lambda) - \mu_-' (\lambda) I)u^- (x; \lambda)\Big) 
    \wedge \tilde{\mathcal{U}}^- (x; \lambda) dx. 
\end{aligned}    
\end{equation}

Our next goal will be to take a limit of this expression as $R$ tends to 
$+ \infty$, and for this we first need to look carefully at the 
wedge product $u_{\lambda}^- (-R; \lambda) \wedge \tilde{\mathcal{U}}^- (-R; \lambda)$. 
Beginning with $u_{\lambda}^- (-R; \lambda)$, we recall from Proposition
\ref{ode-proposition1} (and the definition of $u^- (-R; \lambda)$) that
\begin{equation*}
    \lim_{R \to + \infty} u^- (-R; \lambda)
    = v^- (\lambda),
\end{equation*}
where $v^- (\lambda)$ is analytic on $\Omega$ and the convergence is 
uniform on compact subsets of $\Omega$. It follows that the limit 
of the derivatives with respect to $\lambda$ converges to 
derivatives of the limits,
\begin{equation*}
    \lim_{R \to + \infty} u_{\lambda}^- (-R; \lambda)
    = v^{-\,\prime} (\lambda).
\end{equation*}
Next, for $\tilde{\mathcal{U}}^- (-R; \lambda)$ we have the complication
that we don't have a convenient expression for $\tilde{\mathcal{Y}}^+ (-R; \lambda)$
for large values of $R$. Nonetheless, recalling that 
$\tilde{\mathcal{Y}}^+ (x; \lambda)$ can be viewed as a solution 
to the ODE $\tilde{\mathcal{Y}}^{+\,\prime} = \tilde{A} (x; \lambda) \tilde{\mathcal{Y}}^+$, we 
can characterize $\tilde{\mathcal{Y}}^+ (-R; \lambda)$ via a basis 
of solutions to this ODE constructed for $x \ll 0$. Precisely, 
under the assumptions of Proposition \ref{ode-proposition2} 
we can construct a basis $\{\tilde{\mathcal{Y}}_j^- (x; \lambda)\}_{j=1}^n$
for the solutions of $\tilde{\mathcal{Y}}' = \tilde{A} (x; \lambda) \tilde{\mathcal{Y}}$,
indexed so that 
\begin{equation*}
    \lim_{R \to +\infty} e^{- \mu_- (\lambda) R} \tilde{\mathcal{Y}}_1^- (-R; \lambda)
    = \tilde{\mathcal{V}}^- (\lambda),
\end{equation*}
where $\tilde{\mathcal{V}}^- (\lambda)$ is an eigenvector of $\tilde{A}_- (\lambda)$
associated with the eigenvalue $- \mu_- (\lambda)$ (which is the most 
negative eigenvalue of $\tilde{A}_- (\lambda)$), and additionally 
\begin{equation*}
    \lim_{R \to +\infty} e^{- \mu_- (\lambda) R} \tilde{\mathcal{Y}}_j^- (-R; \lambda)
    = 0, \quad j = 2, 3, \dots, n.
\end{equation*}
If we now write 
\begin{equation*}
    \tilde{\mathcal{Y}}^+ (-R; \lambda)
    = \sum_{j=1}^n \tilde{c}_j (\lambda) \tilde{\mathcal{Y}}_j^- (-R; \lambda),
\end{equation*}
for some expansion coefficients $\{\tilde{c}_j (\lambda)\}_{j=1}^n$  then 
\begin{equation*}
    \tilde{\mathcal{U}}^- (-R; \lambda)
    = \sum_{j=1}^n \tilde{c}_j (\lambda) e^{- \mu_- (\lambda) R} \tilde{\mathcal{Y}}_j^- (-R; \lambda),
\end{equation*}
and 
\begin{equation*}
    \lim_{R \to + \infty} \tilde{\mathcal{U}}^- (-R; \lambda)
    = \tilde{c}_1 (\lambda) \tilde{\mathcal{V}}^- (\lambda).
\end{equation*}

In order to better understand the nature of the expansion coefficient 
$\tilde{c}_1 (\lambda)$, we recall that $D (\lambda)$ is independent 
of $x$, allowing us to write 
\begin{equation*}
    D (\lambda) = \lim_{R \to + \infty} u^- (-R; \lambda) \wedge \tilde{\mathcal{U}}^- (-R; \lambda)
    = \tilde{c}_1 (\lambda) v^- (\lambda) \wedge \tilde{\mathcal{V}}^{-} (\lambda).
\end{equation*}
We see that 
\begin{equation*}
\tilde{c}_1 (\lambda) = \frac{D (\lambda)}{v^- (\lambda) \wedge \tilde{\mathcal{V}}^{-} (\lambda)},    
\end{equation*}
and consequently 
\begin{equation*}
    \lim_{R \to + \infty} u_{\lambda}^- (-R; \lambda) \wedge \tilde{\mathcal{U}}^- (-R; \lambda)
    = \frac{v^{-\,\prime} (\lambda) \wedge \tilde{\mathcal{V}}^{-} (\lambda)}
    {v^- (\lambda) \wedge \tilde{\mathcal{V}}^{-} (\lambda)} D(\lambda).
\end{equation*}
Combining this last relation with (\ref{uminus-relation}), in which we take the 
limit as $R \to - \infty$, we obtain the relation
\begin{equation}\label{uminus-relation-again}
\begin{aligned}
u_{\lambda}^-& (0; \lambda) \wedge \tilde{\mathcal{U}}^- (0; \lambda)
= \frac{v^{-\,\prime} (\lambda) \wedge \tilde{\mathcal{V}}^{-} (\lambda)}
    {v^- (\lambda) \wedge \tilde{\mathcal{V}}^{-} (\lambda)} D(\lambda) \\
    &+ \int_{-\infty}^0 \Big((A_{\lambda} (x; \lambda) - \mu_-' (\lambda) I)u^- (x; \lambda)\Big) 
    \wedge \tilde{\mathcal{U}}^- (x; \lambda) dx. 
\end{aligned}
\end{equation}

By a similar calculation, if we start with 
\begin{equation*}
    D (\lambda) 
    = u^+ (x; \lambda) \wedge \tilde{\mathcal{U}}^+ (x; \lambda),
\end{equation*}
where
\begin{equation*}
    \begin{aligned}
    u^+ (x; \lambda) &:= e^{- \mu_+ (\lambda) x} \eta^- (x; \lambda) \\
    \tilde{\mathcal{U}}^+ (x; \lambda) &= e^{+ \mu_+ (\lambda) x} \tilde{\mathcal{Y}}^+ (x; \lambda).
    \end{aligned}
\end{equation*}
we find the relation
\begin{equation*}
    \frac{d}{dx} (u^+ (x; \lambda) \wedge \tilde{\mathcal{U}}_{\lambda}^+ (x; \lambda))
    = u^+ (x; \lambda) \wedge ((\tilde{A}_{\lambda} (x; \lambda) + \mu_+' (\lambda) I) \tilde{\mathcal{U}}^+ (x; \lambda)). 
\end{equation*}
Upon integrating on $(0, R)$ and using Proposition \ref{conjugation-proposition},
we obtain the relation 
\begin{equation} \label{uplus-relation}
\begin{aligned}
    u_{\lambda}^+ &(R; \lambda) \wedge \tilde{\mathcal{U}}^+ (R; \lambda)
    - u_{\lambda}^+ (0; \lambda) \wedge \tilde{\mathcal{U}}^+ (0; \lambda) \\
    &= \int_0^R \Big((- A_{\lambda} (x; \lambda) + \mu_+' (\lambda) I)u^+ (x; \lambda)\Big) 
    \wedge \tilde{\mathcal{U}}^+ (x; \lambda) dx. 
\end{aligned}
\end{equation}
Proceeding similarly as before, we obtain the relation 
\begin{equation*}
    \lim_{R \to + \infty} 
    u_{\lambda}^+ (R; \lambda) \wedge \tilde{\mathcal{U}}^+ (R; \lambda)
    = \frac{v^{+} (\lambda) \wedge \tilde{\mathcal{V}}^{+\,\prime} (\lambda)}
    {v^+ (\lambda) \wedge \tilde{\mathcal{V}}^{+} (\lambda)} D(\lambda), 
\end{equation*}
and from (\ref{uplus-relation}) (by taking a limit as $R$ tends to $+ \infty$)
\begin{equation} \label{uplus-relation-again}
\begin{aligned}
    u_{\lambda}^+ (0; \lambda) &\wedge \tilde{\mathcal{U}}^+ (0; \lambda)
    = \frac{v^{+} (\lambda) \wedge \tilde{\mathcal{V}}^{+\,\prime} (\lambda)}
    {v^+ (\lambda) \wedge \tilde{\mathcal{V}}^{+} (\lambda)} D(\lambda) \\
    &\quad + \int_0^{+\infty} \Big((A_{\lambda} (x; \lambda) - \mu_+' (\lambda) I)u^+ (x; \lambda)\Big) 
    \wedge \tilde{\mathcal{U}}^+ (x; \lambda) dx. 
\end{aligned}
\end{equation}

At this point, let's note the relation 
\begin{equation*}
    \begin{aligned}
    u^- (0; \lambda) 
    &= u^+ (0; \lambda) = \eta^- (0;\lambda) \\
    \tilde{\mathcal{U}}^- (0; \lambda) 
    &= \tilde{\mathcal{U}}^+ (0; \lambda) = \tilde{\mathcal{Y}}^+ (0;\lambda).
    \end{aligned}
\end{equation*}
Starting with 
\begin{equation*}
    D (\lambda) = u^- (0; \lambda) \wedge \tilde{\mathcal{U}}^- (0; \lambda), 
\end{equation*}
these relations allow us to write 
\begin{equation*}
    \begin{aligned}
    D' (\lambda) &= u_{\lambda}^- (0; \lambda) \wedge \tilde{\mathcal{U}}^- (0; \lambda)
    + u^- (0; \lambda) \wedge \tilde{\mathcal{U}}_{\lambda}^- (0; \lambda) \\
    &= u_{\lambda}^- (0; \lambda) \wedge \tilde{\mathcal{U}}^- (0; \lambda)
    + u^+ (0; \lambda) \wedge \tilde{\mathcal{U}}_{\lambda}^+ (0; \lambda).  
    \end{aligned}
\end{equation*}
Combining this final relation with (\ref{uminus-relation-again}) 
and (\ref{uplus-relation-again}), we obtain the assertion 
of Proposition \ref{evans-proposition}.
\end{proof}

In the following proposition, we address a special case that will be 
important for our applications.

\begin{proposition} \label{evans-proposition-derivatives}
Let Assumptions {\bf (A)} through {\bf (D)} hold, with 
additionally $A_{\lambda \lambda} (x; 0) = 0$ (trivially true if
$A(x;\lambda)$ depends linearly on $\lambda$), and let the 
Evans function $D (\lambda)$ be specified as in (\ref{evans-function-defined}). 
If $D (0) = 0$, then 
\begin{equation*}
    D'(0) = \int_{-\infty}^{+\infty} \Big(A_{\lambda} (x; 0) \eta^- (x; 0)\Big) 
    \wedge \tilde{\mathcal{Y}}^+ (x; 0) dx. 
\end{equation*}
Moreover, if additionally $D'(0) = 0$, then 
\begin{equation*}
    D'' (0)
    =  \int_{-\infty}^{+\infty} \Big(A_{\lambda} (x; 0) \eta_{\lambda}^- (x; 0)\Big) 
    \wedge \tilde{\mathcal{Y}}^+ (x; 0) dx
    + \int_{-\infty}^{+\infty} \Big(A_{\lambda} (x; 0) \eta^- (x; 0)\Big) 
    \wedge \tilde{\mathcal{Y}}_{\lambda}^+ (x; 0) dx. 
\end{equation*}
\end{proposition}

\begin{proof}
First, from Proposition \ref{evans-proposition}, we have the relation
\begin{equation*}
    \begin{aligned}
    D'(0) &= \int_{-\infty}^0 \Big((A_{\lambda} (x; 0) - \mu_-'(0) I) \eta^- (x; 0)\Big) 
    \wedge \tilde{\mathcal{Y}}^+ (x; 0) dx  \\
    &\quad \quad + \int_0^{+\infty} \Big((A_{\lambda} (x; 0) - \mu_+'(0) I) \eta^- (x; 0)\Big) 
    \wedge \tilde{\mathcal{Y}}^+ (x; 0) dx.
    \end{aligned}
\end{equation*}
The wedge product $\eta^- (x; 0) \wedge \tilde{\mathcal{Y}}^+ (x; 0)$ is just $D (0)$,
and so is identically 0, allowing us to reduce this to 
\begin{equation*}
    D' (0) = \int_{-\infty}^{+\infty} \Big(A_{\lambda} (x; 0) \eta^- (x; 0)\Big) 
    \wedge \tilde{\mathcal{Y}}^+ (x; 0) dx,
\end{equation*}
as claimed. 

Next, we assume $D(0) = 0$ and $D'(0) = 0$, and we compute $D''(0)$.
We proceed by differentiating the 
expression for $D' (\lambda)$ in Proposition \ref{evans-proposition}. 
First, since $D(0) = 0$ and $D' (0) = 0$, we see that the $\lambda$-derivative
of 
\begin{equation*}
    \Big(\frac{v^{-\,\prime} (\lambda) \wedge \tilde{\mathcal{V}}^- (\lambda)}
    {v^- (\lambda) \wedge \tilde{\mathcal{V}}^- (\lambda)} 
    + \frac{v^+ (\lambda) \wedge \tilde{\mathcal{V}}^{+\,\prime} (\lambda)}
    {v^+ (\lambda) \wedge \tilde{\mathcal{V}}^+ (\lambda)} \Big) D(\lambda),
\end{equation*}
evaluated at $\lambda = 0$ must be 0. The remaining terms in $D' (\lambda)$ can be 
expressed as 
\begin{equation*}
    \begin{aligned}
    \int_{-\infty}^{0} &\Big(A_{\lambda} (x; \lambda) \eta^- (x; \lambda)\Big) 
    \wedge \tilde{\mathcal{Y}}^+ (x; \lambda) - \mu_-' (\lambda) D(\lambda) dx \\ 
    & \quad + \int_0^{+\infty} \Big(A_{\lambda} (x; \lambda) \eta^- (x; \lambda)\Big) 
    \wedge \tilde{\mathcal{Y}}^+ (x; 0) - \mu_+'(\lambda) D(\lambda) dx.
    \end{aligned}
\end{equation*}
Upon differentiation in $\lambda$ and evaluation at $\lambda = 0$ (and using 
the assumptions $A_{\lambda \lambda} (x; 0) = 0$, $D(0) = 0$ and $D' (0) = 0$), we are left with 
\begin{equation*}
    D'' (0)
    =  \int_{-\infty}^{+\infty} \Big(A_{\lambda} (x; 0) \eta_{\lambda}^- (x; 0)\Big) 
    \wedge \tilde{\mathcal{Y}}^+ (x; 0) dx
    + \int_{-\infty}^{+\infty} \Big(A_{\lambda} (x; 0) \eta^- (x; 0)\Big) 
    \wedge \tilde{\mathcal{Y}}_{\lambda}^+ (x; 0) dx,
\end{equation*}
as claimed. 
\end{proof}

In order to take full advantage of the relations for $D'(0)$ and $D''(0)$
from Proposition \ref{evans-proposition-derivatives}, it's useful to be clear
about the connection between $\tilde{\omega}_1 (x; \lambda)$ and $D (\lambda)$. 
For this, we note from Propositions \ref{ode-proposition1} and \ref{ode-proposition2} 
and the specifications
\begin{equation*}
    D(\lambda) 
    = \eta^- (x; \lambda) \wedge \tilde{\mathcal{Y}}^+ (x; \lambda)
\end{equation*}
and 
\begin{equation*}
    \tilde{\omega}_1^+ (x; \lambda)
    = \eta^- (x; \lambda) \wedge \tilde{\mathcal{V}}^+ (\lambda), 
\end{equation*}
that given any closed interval $K \subset I$, and any $\epsilon > 0$, 
we can take $x > 0$ sufficiently large so that 
\begin{equation} \label{evans-omega1-close}
    |D(\lambda) - e^{- \mu_+ (\lambda) x} \tilde{\omega}_1^+ (x; \lambda)| < \epsilon
\end{equation}
for all $\lambda \in K$. Since $D (\lambda)$ is analytic at $\lambda = 0$,
its sign for $\lambda$ sufficiently close to $0$ is determined by 
the first non-zero value $D(0)$, $D'(0)$, $D''(0)$ etc. via the relation
\begin{equation*}
    D(\lambda) = D(0) + D' (0) \lambda + \frac{1}{2} D'' (0) \lambda^2 
    + \dots.
\end{equation*}

In order to be clear about this process, let's take the specific 
case $D(0) = 0$ and $D'(0) > 0$, and let's suppose (as will be the 
case in our applications) that by taking
$x$ sufficiently large we can fix the sign of $\tilde{\psi}_2^+ (x; 0)$,
say $\tilde{\psi}_2^+ (x; 0) > 0$.  
Under these conditions, there exists a value $\delta > 0$ sufficiently small so that 
$D(\lambda) < 0$ for all $\lambda \in (-\delta, 0)$,
and in particular $D (-\delta/2) < 0$. It follows that we can take $x$ sufficiently 
large so that $\tilde{\psi}_1 (x; -\delta/2) < 0$ (because we must 
have $\tilde{\omega}_1^+ (x; \lambda) < 0$ by (\ref{evans-omega1-close})). 
Since $\tilde{\psi}_2 (x; \lambda)$
depends smoothly on $\lambda$, we can additionally choose $\delta$ small enough 
so that $\tilde{\psi}_2 (x; -\delta/2) > 0$, which places $p^+ (x; -\delta/2)$
in the second quadrant. We can conclude that 
\begin{equation*}
    \ind (\mathpzc{g} (x; \cdot), \mathpzc{h} (x; \cdot); [-\delta/2, 0])
    = +1.
\end{equation*}
Returning to (\ref{penultimate-relation}) from the proof of Theorem \ref{main-theorem},
we can write (with $\lambda_2 = 0$)
\begin{equation} \label{penultimate-relation-again}
\begin{aligned}
\ind &(\mathpzc{g} (c_2; \cdot), \mathpzc{h} (c_2; \cdot); [\lambda_1, -\delta/2])
+   \ind (\mathpzc{g} (x; \cdot), \mathpzc{h} (x; \cdot); [-\delta/2, 0]) \\
&=  \ind (\mathpzc{g} (\cdot; \lambda_2), \mathpzc{h}^+ (\lambda_2); [-\infty, + \infty]) 
 - \ind (\mathpzc{g} (\cdot; \lambda_1), \mathpzc{h}^+ (\lambda_1); [-\infty, +\infty]) \\
&+ \ind (\mathpzc{g}^- (\cdot), \mathpzc{h}^+ (\cdot); [\lambda_1, \lambda_2])
- \mathfrak{m},
\end{aligned}
\end{equation}
from which we now obtain a lower bound on $\mathcal{N}_{\#} ([\lambda_1, \lambda_2))$
rather than on $\mathcal{N}_{\#} ([\lambda_1, \lambda_2])$. In particular, 
in this case, we obtain the relation
\begin{equation} \label{detailed1}
\begin{aligned}
\mathcal{N}_{\#} ([\lambda_1, \lambda_2)) + 1
&\ge \Big|\ind (\mathpzc{g} (\cdot; \lambda_2), \mathpzc{h}^+ (\lambda_2); [-\infty, + \infty]) 
 - \ind (\mathpzc{g} (\cdot; \lambda_1), \mathpzc{h}^+ (\lambda_1); [-\infty, +\infty]) \\
&+ \ind (\mathpzc{g}^- (\cdot), \mathpzc{h}^+ (\cdot); [\lambda_1, \lambda_2])
- \mathfrak{m} \Big|.
\end{aligned}
\end{equation}

Additional information about the spectrum of (\ref{nonhammy}) can 
be found by also computing the limit of $D(\lambda)$ as 
$\lambda$ tends toward $- \infty$. For this, we will make
use of Proposition \ref{left-shelf-proposition}, so we make
the same assumptions as stated there. 

\begin{proposition} \label{evans-large-lambda}
Let the Assumptions of Proposition \ref{left-shelf-proposition} hold,
and let $D(\lambda)$ be specified as in (\ref{evans-function-defined}).
Then 
\begin{equation*}
    \lim_{\lambda \to - \infty} D(\lambda) = 1. 
\end{equation*}
\end{proposition}

\begin{proof}
First, we recall that $v^- (\lambda)$ denotes an eigenvector
of $A_- (\lambda)$ associated with the eigenvalue $\mu_- (\lambda)$,
and $\tilde{\mathcal{V}}^+ (\lambda)$ denotes an eigenvector of 
$\tilde{A}_+ (\lambda)$ associated with the eigenvalue $- \mu_+ (\lambda)$.
The normalization in place via Assumption {\bf (C)} and the 
relation between $\tilde{\mathcal{V}}^+ (\lambda)$ and 
$w^+ (\lambda)$ (described in Remark \ref{eigenvector-remark}) 
is 
\begin{equation} \label{normalization1}
    v^- (\lambda) \wedge \tilde{\mathcal{V}}^+ (\lambda)
     = 1.
\end{equation}

According to Proposition \ref{left-shelf-proposition}, we have the relations
\begin{equation*}
    V_- (\lambda)^{-1} \eta^- (0; \lambda) = e_1 + \mathbf{o} (1),
    \quad \lambda \to - \infty,
\end{equation*}
and 
\begin{equation*}
    z^+ (0; \lambda) V_+ (\lambda) e_1 = 1 + \mathbf{o} (1),
    \quad \lambda \to - \infty,
\end{equation*}
where under the Assumptions of Proposition \ref{left-shelf-proposition}
we have $V_- (\lambda) = V_+ (\lambda)$. 
Using Remark \ref{equivalence-pw1992}, we can now compute 
\begin{equation*}
\begin{aligned}
    D(\lambda)
    &= \eta^- (0; \lambda) \wedge \tilde{\mathcal{Y}}^+ (0; \lambda)
    = z^+ (0; \lambda) \cdot \eta^- (0; \lambda) 
    = z^+ (0; \lambda) V_+ (\lambda) V_- (\lambda)^{-1} \eta^- (0; \lambda) \\
    &= z^+ (0; \lambda) V_+ (\lambda) (e_1 + \mathbf{o} (1))
    =  z^+ (0; \lambda) V_+ (\lambda) e_1 
    + z^+ (0; \lambda) V_+ (\lambda) \mathbf{o} (1)
    = 1 + \mathbf{o} (1),
\end{aligned}
\end{equation*}
where we have used the boundedness of $z^+ (0; \lambda) V_+ (\lambda)$
(from the final assertion of Proposition \ref{left-shelf-proposition}). 
\end{proof}

\section{Applications}
\label{applications-section}

In this section, we illustrate the implementation of Theorem 
\ref{main-theorem} with two applications. 

\subsection{The Generalized KdV Equation}
\label{gkdv-section}

As our first application, we consider traveling waves $\bar{u} (x - st)$
occurring as solutions to the generalized Korteweg-De Vries 
equation
\begin{equation} \label{gKdV} 
    u_t + f(u)_x = - u_{xxx},
    \quad f(u) = \frac{u^{p+1}}{p+1}, 
    \quad p \ge 1.
\end{equation}
For these calculations, it will be convenient to shift 
to a moving coordinate system for which $\bar{u} (x)$ 
is a stationary solution to 
\begin{equation} \label{gKdV-moving} 
    u_t + (f(u) - su)_x = - u_{xxx}.
\end{equation}
It follows that $\bar{u} (x)$ is a solution to the 
ODE 
\begin{equation} \label{gkdv-wave-equation}
\bar{u}^p \bar{u}' - s \bar{u}' = - \bar{u}''',
\end{equation}
for which we can readily check that for $s > 0$
\begin{equation} \label{gkdv-wave}
    \bar{u} (x) = \alpha \sech^{2/p}(\gamma x),
    \quad \alpha = (\frac{1}{2} s (p+1) (p+2))^{1/p},
    \quad \gamma = \frac{p \sqrt{s}}{2}
\end{equation}
is an exact solution. (These solutions are taken 
from \cite{PW1992}.) Our goal is to use the 
preceding analysis to analyze the spectral 
stability/instability of the wave (\ref{gkdv-wave})
as a solution to (\ref{gKdV-moving}). 

If we linearize (\ref{gKdV-moving}) about the wave 
$\bar{u} (x)$, we obtain the associated eigenvalue
problem 
\begin{equation} \label{gkdv-evalue-problem}
    L \phi = - \phi''' - (a(x) \phi)' = \lambda \phi,
    \quad a(x) = \bar{u}(x)^p - s.
\end{equation}
As a starting point for the analysis, we recall from \cite{He1981, KP2013}
that the essential spectrum of $L$ can be determined from 
the asymptotic equation 
\begin{equation} \label{gkdv-asymptotic}
    - \phi''' - a_{\pm} \phi' = \lambda \phi,
\end{equation}
where
\begin{equation*}
    a_{\pm} = \lim_{x \to \pm \infty} a(x) 
    = - s = - 4 \frac{\gamma^2}{p^2}.
\end{equation*}
(For this application, $a_- = a_+$, so the subscript 
$\pm$ is intended to signify that both sides 
are analyzed at once.)
Precisely, the essential spectrum of $L$ comprises 
the values $\lambda$ for which $\phi (x; k) = e^{i k x}$
is a solution to (\ref{gkdv-asymptotic}) for some 
$k \in \mathbb{R}$, namely 
\begin{equation*}
    \sigma_{\ess} (L)
    = \{\lambda \in \mathbb{C}: \lambda = i k^3 - i a_{\pm} k, \,\, k \in \mathbb{R}\},
\end{equation*}
which is the imaginary axis. 

In the usual way, we express (\ref{gkdv-evalue-problem}) 
as a first-order system by introducing a vector function 
$y = (y_1 \,\, y_2 \,\, y_3)^T$
with $y_1 = \phi$, $y_2 = \phi'$, and $y_3 = \phi''$, yielding 
(\ref{nonhammy}) with 
\begin{equation} \label{gkdv-A}
    A (x; \lambda)
    = \begin{pmatrix}
        0 & 1 & 0 \\
        0 & 0 & 1 \\
        -\lambda - a'(x) & -a(x) & 0 
    \end{pmatrix},
\end{equation}
and correspondingly 
\begin{equation*}
    A_{\pm} (\lambda)
    := \lim_{x \to \pm \infty} A(x; \lambda)
    = \begin{pmatrix}
        0 & 1 & 0 \\
        0 & 0 & 1 \\
        -\lambda & -a_{\pm} & 0 
    \end{pmatrix}.
\end{equation*}
For each fixed $\lambda \in \mathbb{C}$, the eigenvalues of 
$A_{\pm} (\lambda)$ are roots $\mu$ of the cubic polynomial 
\begin{equation} \label{gkdv-h}
h (\mu; \lambda)
= \mu^3 + a_{\pm} \mu + \lambda.
\end{equation}
For $\lambda = 0$, the roots are easily found to be
\begin{equation*}
    \mu_1 (0) = - 2 \frac{\gamma}{p},
    \quad \mu_2 (0) = 0,
    \quad \mu_3 (0) = 2 \frac{\gamma}{p},
\end{equation*}
where we've recalled that $a_{\pm} = -4 \frac{\gamma^2}{p^2}$. 
As $\lambda$ decreases from 0, the graph of $h(\mu; \lambda)$
will keep the same form but decrease so that the roots $\mu_1 (\lambda)$
and $\mu_2 (\lambda)$ will initially move toward one another while the root 
$\mu_3 (\lambda)$ increases. As $\lambda$ continues to 
decrease, the roots $\mu_1 (\lambda)$ and $\mu_2 (\lambda)$ will
coalesce at some value $\lambda = \lambda_c$ into a complex conjugate pair.  
The precise value of $\lambda_c$ isn't critical to our analysis, but 
one easily finds it to be 
\begin{equation} \label{gkdv-coalescence}
    \lambda_c = - 2 (s/3)^{3/2}.
\end{equation}
For $\lambda$ above this coalescence value (i.e., for $\lambda \in (\lambda_c, 0]$) 
we can associate an eigenvector $v_j (\lambda)$ with each $\mu_j (\lambda)$, $j \in \{1, 2, 3\}$,
\begin{equation} \label{gkdv-asymptotic-eigenvectors}
v_j (\lambda) = (1 \,\, \mu_j (\lambda) \,\, \mu_j (\lambda)^2)^T,
\quad j = 1, 2, 3.
\end{equation}
Correspondingly, it's straightforward to identify three 
linearly independent solutions of (\ref{nonhammy}) with 
$A(x; \lambda)$ as in (\ref{gkdv-A}),
\begin{equation*}
    y_j^{\pm} (x; \lambda) 
    = e^{\mu_j (\lambda) x} (v_j (\lambda) + E_j^{\pm} (x; \lambda)),
\end{equation*}
where $E_j^{\pm} (x; \lambda) = {\mathbf{O} (e^{- \alpha |x|})}$
uniformly in $\lambda$ on compact subsets of $I$ 
for some fixed constant $\alpha > 0$. 
(See, e.g., \cite{ZH1998}.) The eigenvector $v_3 (\lambda)$ 
has the same form for all $\lambda \le 0$, and up to a choice of scaling 
is $v^- (\lambda)$ from Proposition \ref{ode-proposition1}.
Likewise, up to a choice of scaling, $\eta^- (x; \lambda)$ 
is $y_3^- (x; \lambda)$.
On the other hand, once $\lambda$ decreases to the 
coalescence threshold, it becomes problematic to separate 
the solutions associated with $\mu_1 (\lambda)$ and $\mu_2 (\lambda)$,
and one of the advantages of the wedge-product formulation is 
that no such separation is necessary. 

Turning to consideration of $\tilde{A} (x; \lambda)$, we see from 
(\ref{tildeA}) that for this application 
\begin{equation*}
    \tilde{A} (x; \lambda)
    = \begin{pmatrix}
        0 & 1 & 0 \\
        -a(x) & 0 & 1 \\
        \lambda + a'(x) & 0 & 0
    \end{pmatrix},
\end{equation*}
and correspondingly 
\begin{equation} \label{gkdv-Aplus}
    \tilde{A}_{\pm} (\lambda)
    = \lim_{x \to \pm \infty} \tilde{A} (x; \lambda)
    = \begin{pmatrix}
        0 & 1 & 0 \\
        -a_{\pm} & 0 & 1 \\
        \lambda & 0 & 0
    \end{pmatrix}.
\end{equation}
According to Proposition \ref{eigenvalues-proposition}, 
$ \tilde{A}_{\pm} (\lambda)$
has eigenvalues $\tilde{\mu}_j (\lambda) = - \mu_{3-j+1} (\lambda)$, 
$j = 1, 2, 3$, and we  
readily find that the associated eigenvectors have 
the form $\tilde{\mathcal{V}}_j = (1 \,\, \tilde{\mu}_j \,\, \lambda/\tilde{\mu}_j)^T$.
In particular, $\tilde{\mu}_1 (\lambda) = - \mu_3 (\lambda)$,
and up to a choice of scaling $\tilde{\mathcal{V}}_1 (\lambda)$ 
is $\tilde{\mathcal{V}}^+ (\lambda)$ from Proposition \ref{ode-proposition2}. 

Since this application is in the setting of Proposition \ref{left-shelf-crossings}, 
we can conclude immediately that we can take $\lambda_1$ sufficiently negative so 
that the hyperplane index associated with the left shelf gives no contribution.
In addition, we see from (\ref{normalization1}) in the proof of 
Proposition \ref{evans-large-lambda} that $v^- (\lambda) \wedge \tilde{\mathcal{V}}^+ (\lambda) = 1$
for all $\lambda \le 0$, ensuring that there are no crossing points
along the bottom shelf (as discussed in Section \ref{bottom-shelf-section}). We 
note that in order to obtain this normalization, we can use 
$v^- (\lambda) = \kappa (\lambda) v_3 (\lambda)$ 
and $\tilde{\mathcal{V}}^+ (\lambda) = \kappa (\lambda) \tilde{\mathcal{V}}_1 (\lambda)$,
where 
\begin{equation} \label{gkdv-scaling}
    \kappa (\lambda)
    = 1/\sqrt{- \lambda/\mu_3 (\lambda) + 2 \mu_3 (\lambda)^2}.
\end{equation}

Next, we turn to the evaluation of $D(0)$, $D' (0)$, and $D'' (0)$, 
which we adapt from \cite{PW1992} with only 
minor changes required for the current wedge-based 
formulation. To start, we observe that since $y_3^- (x; 0)$ 
is the only solution from our basis $\{y_j^- (x; 0)\}_{j=1}^3$
that decays as $x$ tends to $-\infty$ it must be the case 
that there exists some constant $k_-$ so that 
(keeping in mind that $\eta^- (x; \lambda)$ is just a rescaling 
of $y_3^- (x; \lambda)$)
\begin{equation} \label{gkdv-etamat0}
\eta^- (x; 0) = k_-    
\begin{pmatrix}
\bar{u}' (x) \\
\bar{u}'' (x) \\
\bar{u}''' (x)
\end{pmatrix}.
\end{equation}
According to Proposition \ref{ode-proposition1}, we can write 
\begin{equation*}
    \lim_{x \to -\infty} e^{- \mu_3 (0) x} 
    k_-    
\begin{pmatrix}
\bar{u}' (x) \\
\bar{u}'' (x) \\
\bar{u}''' (x)
\end{pmatrix}
= v^- (0).
\end{equation*}
Recalling that the first component of $v^- (\lambda)$ has been 
chosen by convention to be positive and that 
$\bar{u}' (x) > 0$ for all $x < 0$, we conclude that 
$k_- > 0$. 

In addition, we need to understand the nature of $\tilde{\mathcal{Y}}^+ (x; 0)$,
which solves the ODE 
\begin{equation*}
\tilde{\mathcal{Y}}^{+\,\prime} 
= \tilde{A} (x; 0) \tilde{\mathcal{Y}}^+.
\end{equation*}
In particular, we see that if we write 
$\tilde{\mathcal{Y}}^+ = (\tilde{\mathcal{Y}}_1^+ \,\, \tilde{\mathcal{Y}}_2^+ \,\, \tilde{\mathcal{Y}}_3^+)^T$,
then 
$\tilde{\mathcal{Y}}_1^{+\,\prime} = \tilde{\mathcal{Y}}_2^{+}$,
so that 
\begin{equation} \label{y1tilde00}
\tilde{\mathcal{Y}}_1^{+\,\prime \prime} 
=  \tilde{\mathcal{Y}}_2^{+\,\prime}
= - a(x) \tilde{\mathcal{Y}}_1^+ 
+ \tilde{\mathcal{Y}}_3^+.  
\end{equation}
Differentiating once more, we see that 
\begin{equation} \label{y1tilde0}
\tilde{\mathcal{Y}}_1^{+\,\prime \prime \prime}   
= - (a(x) \tilde{\mathcal{Y}}_1^+)' 
+ \tilde{\mathcal{Y}}_3^{+\,\prime} 
= - (a(x) \tilde{\mathcal{Y}}_1^+)' 
+ a'(x) \tilde{\mathcal{Y}}_1^+ 
= - a(x) \tilde{\mathcal{Y}}_1^{+\,\prime}. 
\end{equation}
Comparing this equation with (\ref{gkdv-wave-equation}),
and recalling the definition of $a(x)$ in 
(\ref{gkdv-evalue-problem}), we see that 
$\tilde{\mathcal{Y}}_1^{+} (x; 0)$ solves 
the same equation as $\bar{u} (x)$, which is 
the only solution of this equation that lies 
left in $\mathbb{R}$. It follows that there 
exists a constant $k_+$ so that 
\begin{equation} \label{tilde-y-zero}
\tilde{\mathcal{Y}}^{+} (x; 0)
= \begin{pmatrix}
\tilde{\mathcal{Y}}_1^{+} (x; 0) \\
\tilde{\mathcal{Y}}_2^{+} (x; 0) \\
\tilde{\mathcal{Y}}_3^{+} (x; 0) 
\end{pmatrix}
=
\begin{pmatrix}
\tilde{\mathcal{Y}}_1^{+} (x; 0) \\
\tilde{\mathcal{Y}}_1^{+\,\prime} (x; 0) \\
\tilde{\mathcal{Y}}_1^{+\,\prime \prime} (x; 0) + a(x) \tilde{\mathcal{Y}}_1^{+} (x; 0) 
\end{pmatrix}
= k_+ \begin{pmatrix}
\bar{u} (x) \\
\bar{u}'(x) \\
\bar{u}''(x) + a(x) \bar{u} (x)
\end{pmatrix}.
\end{equation}
Recalling the asymptotic relation 
\begin{equation*}
    \lim_{x \to + \infty} e^{\mu_+ (0) x} \tilde{\mathcal{Y}}^{+} (x; 0)
    = \tilde{\mathcal{V}}^+ (0),
\end{equation*}
and noting that $\bar{u} (x) > 0$ for all $x > 0$, we see that 
$k_+ > 0$.

We can now directly compute 
\begin{equation*}
    \begin{aligned}
    D(0)
    &= \eta^- (x;0) \wedge \tilde{\mathcal{Y}}^+ (x; 0)
    = k_- k_+ \begin{pmatrix}
\bar{u}' (x) \\
\bar{u}'' (x) \\
\bar{u}''' (x)
\end{pmatrix}
\wedge
\begin{pmatrix}
\bar{u} (x) \\
\bar{u}'(x) \\
\bar{u}''(x) + a(x) \bar{u} (x)
\end{pmatrix} \\
&=  k_- k_+ \bar{u} (x) (\bar{u}'''(x) + a(x) \bar{u}' (x)) = 0.
    \end{aligned}
\end{equation*}

Using 
\begin{equation*}
    A_{\lambda} (x; 0)
    = \begin{pmatrix}
    0 & 0 & 0 \\
    0 & 0 & 0 \\
    -1 & 0 & 0
    \end{pmatrix},
\end{equation*}
along with the relations above for $\eta^- (x; 0)$ and 
$\tilde{\mathcal{Y}}^+ (x; 0)$, we find that 
\begin{equation*}
\Big(A_{\lambda} (x; 0) \eta^- (x; 0)\Big) \wedge \tilde{\mathcal{Y}}^+ (x; 0)
= - k_- k_+ \bar{u}'(x) \bar{u} (x),
\end{equation*}
from which we see that (using the expression for $D'(0)$ in Proposition \ref{evans-proposition-derivatives})
\begin{equation*}
    D' (0) = - k_- k_+ \int_{-\infty}^{+\infty} \frac{d}{dx} \frac{\bar{u}(x)^2}{2} dx
    = 0. 
\end{equation*}

In order to evaluate the integrals in the expression for 
$D''(0)$ in Proposition \ref{evans-proposition-derivatives}, 
we need to understand the functions $\eta_{\lambda}^- (x; 0)$
and $\tilde{\mathcal{Y}}_{\lambda}^+ (x; 0)$. Starting with 
$\eta_{\lambda}^- (x; 0)$, we recall that the first component 
of $\eta^- (x; \lambda)$, 
$\phi (x; \lambda) := \eta_1^- (x; \lambda)$,
solves the eigenvalue problem 
\begin{equation*}
- \phi^{\prime \prime \prime}    
- (a(x) \phi)' = \lambda \phi. 
\end{equation*}
Upon differentiation in $\lambda$, we see that
$\phi_{\lambda}$ solves the equation 
\begin{equation*}
- \phi_{\lambda}^{\prime \prime \prime}    
- (a(x) \phi_\lambda)' = \phi + \lambda \phi_\lambda.     
\end{equation*}
From (\ref{gkdv-etamat0}) we can write $\phi (x; 0) = k_- \bar{u}'(x)$,
so that with $\lambda = 0$ we obtain the equation 
\begin{equation} \label{gkdv-inhomogeneous1}
- \phi_{\lambda}^{\prime \prime \prime} (x; 0)   
- (a(x) \phi_\lambda (x; 0))' = k_- \bar{u}' (x).     
\end{equation}
For comparison, we note that by 
differentiating (\ref{gkdv-wave-equation}) 
in $s$ we obtain the relation 
\begin{equation} \label{gkdv-inhomogeneous}
    p \bar{u}^{p-1} \bar{u}_s + \bar{u}^p \bar{u}_s' - \bar{u}' - s \bar{u}_s' 
    = - \bar{u}_s''',
\end{equation}
which we can re-write as 
\begin{equation} \label{gkdv-baru-s}
    - \bar{u}_s''' - (a(x) \bar{u}_s)' 
    = - \bar{u}'.
\end{equation}
We see that the function $- k_- \bar{u}_s (x; s)$ solves
the inhomogeneous ODE (\ref{gkdv-inhomogeneous1}). We've already 
observed that the only left-decaying solution to the associated homogeneous 
equation is $\bar{u}' (x)$, so we must have 
\begin{equation*}
    \phi_{\lambda} (x; 0) = \beta_- \bar{u}' (x; s) - k_- \bar{u}_s (x; s)
\end{equation*}
for some constant $\beta_- \in \mathbb{R}$. 

Next, if we repeat the calculations leading to (\ref{y1tilde0})
with $\lambda \ne 0$, we obtain the relation 
\begin{equation*} 
\tilde{\mathcal{Y}}_1^{+\,\prime \prime \prime}   
+ a(x) \tilde{\mathcal{Y}}_1^{+\,\prime} 
= \lambda \tilde{\mathcal{Y}}_1^+. 
\end{equation*}
Upon differentiation in $\lambda$, we find that 
$\varphi (x; \lambda) := \partial_{\lambda} \tilde{\mathcal{Y}}_1^+ (x; \lambda)$
satisfies the equation 
\begin{equation*} 
\varphi^{\prime \prime \prime}   
+ a(x) \varphi^{\prime} 
= \lambda \varphi + \tilde{\mathcal{Y}}_1^+. 
\end{equation*}
Evaluating now at $\lambda = 0$ and recalling (\ref{tilde-y-zero}),
we see that 
\begin{equation} \label{tilde-y-lambda-derivative}
\varphi^{\prime \prime \prime} (x; 0)   
+ a(x) \varphi^{\prime} (x; 0)
= k_+ \bar{u} (x; s). 
\end{equation}
For comparison, we integrate (\ref{gkdv-baru-s}) (and 
change signs) to see that 
\begin{equation} \label{gkdv-baru-s-integrated}
   \bar{u}_s'' + a(x) \bar{u}_s 
    = \bar{u},
\end{equation}
for which the constant of integration is seen to be $0$ 
since $\bar{u} (x; s)$ tends to $0$ along with its 
derivatives as $x$ tends to $\pm \infty$.
If we now introduce an integrated variable 
\begin{equation*}
    \mathcal{U} (x; s) 
    := \int_{-\infty}^x \bar{u}_s (\xi; s) d \xi,
\end{equation*}
we can express (\ref{gkdv-baru-s-integrated}) as 
\begin{equation*}
    \mathcal{U}''' + a(x) \mathcal{U}'
    = \bar{u}.
\end{equation*}
In this way, we see that a particular solution of (\ref{tilde-y-lambda-derivative})
is $k_+ \mathcal{U} (x; s)$, and since $\bar{u} (x; s)$ is the 
only solution (up to constant multiplication) of the associated 
homogeneous equation, we can write 
\begin{equation*}
(\partial_{\lambda} \tilde{\mathcal{Y}}_1^+) (x; 0)
= \varphi (x; 0) = \beta_+ \bar{u} (x; s) + k_+ \mathcal{U} (x;s)
\end{equation*}
for some constant $\beta_+$. 

We are now in a position to readily evaluate the expression 
for $D''(0)$ in Proposition \ref{evans-proposition-derivatives}. 
For the first summand, we compute
\begin{equation*}
\begin{aligned}
    - k_+ &\int_{-\infty}^{+\infty}
    \begin{pmatrix}
    0 \\ 0 \\ \beta_- \bar{u}' (x;s) - k_- \bar{u}_s (x;s)
    \end{pmatrix}
    \wedge 
    \begin{pmatrix}
    \bar{u} (x;s) \\
    \bar{u}'(x;s) \\
    \bar{u}''(x;s) + a(x) \bar{u} (x;s)
    \end{pmatrix}
    dx \\
    &= - k_+ \int_{-\infty}^{+\infty} (\beta_- \bar{u}' (x;s) - k_- \bar{u}_s (x;s)) \bar{u} (x;s) dx
    = k_- k_+ \int_{-\infty}^{+\infty} \bar{u} (x;s) \bar{u}_s (x;s) dx,
\end{aligned}
\end{equation*}
where we've observed that $\bar{u} (x;s) \bar{u}' (x;s)$ integrates to 0. Turning to the 
second summand in the expression for $D''(0)$ in Proposition \ref{evans-proposition-derivatives},
we compute
\begin{equation*}
\begin{aligned}
    - k_- &\int_{-\infty}^{+\infty}
    \begin{pmatrix}
    0 \\ 0 \\ \bar{u}'(x;s)
    \end{pmatrix}
    \wedge 
    \begin{pmatrix}
    \beta_+ \bar{u} (x;s) + k_+ \mathcal{U} (x;s) \\ * \\ *
    \end{pmatrix}
    dx \\
    &= - k_- \int_{-\infty}^{+\infty} (\beta_+ \bar{u} (x;s) + k_+ \mathcal{U} (x;s)) \bar{u}'(x;s)  dx
    = - k_- k_+ \int_{-\infty}^{+\infty} \mathcal{U} (x;s) \bar{u}'(x;s) dx,
\end{aligned}
\end{equation*}
where the asterisks indicate terms that don't have a role in the 
calculation. 
If we integrate this last integral by parts and observe that there is no 
contribution from the boundary,
we see that it becomes precisely the above expression
\begin{equation*}
+ k_- k_+ \int_{-\infty}^{+\infty} \bar{u} (x) \bar{u}_s (x) dx.     
\end{equation*}
In total, we can write 
\begin{equation*}
    D'' (0) = k_- k_+ \frac{d}{ds} \int_{-\infty}^{+\infty} \bar{u} (x; s)^2 dx. 
\end{equation*}

Using the specification in (\ref{gkdv-wave}), we see that  
\begin{equation*}
\int_{-\infty}^{+\infty} \bar{u} (x; s)^2 dx
= \alpha (s)^2 \int_{-\infty}^{+\infty} \sech^{4/p} (\gamma (s)x) dx
= c_p \frac{\alpha (s)^2}{\gamma (s)},
\end{equation*}
where $c_p := \int_{-\infty}^{+\infty} \sech^{4/p} x dx$, and in obtaining 
this expression we've used a change of variables $y = \gamma (s) x$. 
We've seen above that $k_- > 0$ and $k_+ > 0$, and it's clear that 
$c_p > 0$, so 
\begin{equation*}
    \sgn D''(0) = \sgn \frac{d}{ds} \frac{\alpha (s)^2}{\gamma (s)}. 
\end{equation*}
In order to determine the sign on the right-hand side, it's convenient to 
write 
\begin{equation*}
    \ln \frac{\alpha (s)^2}{\gamma (s)}
    = \frac{2}{p} \ln (\frac{1}{2} s (p+1) (p+2)) - \ln (\frac{1}{2}p\sqrt{s})
    = \frac{2}{p} \ln s - \ln \sqrt{s} + C, 
\end{equation*}
where $C$ simply notes additional terms constant in $s$. Differentiating 
in $s$, we now see that 
\begin{equation*}
(\frac{\alpha (s)^2}{\gamma (s)})^{-1} \frac{d}{ds} \frac{\alpha (s)^2}{\gamma (s)}
= \frac{4-p}{2ps}.
\end{equation*}
We conclude that 
\begin{equation} \label{gkdv-evans-sign}
    \sgn D''(0) 
    = \begin{cases}
    +1 & 1 \le p < 4 \\
    -1 & 4 < p.
    \end{cases}
\end{equation}

We are now in a position to compute the hyperplane index along the right shelf. 
Our starting point for this calculation is the map 
\begin{equation*}
\begin{aligned}
    \tilde{\omega}_1^+ (x; 0)
    &= \eta^- (x; 0) \wedge \tilde{\mathcal{V}}^+ (0)
    = k_- \kappa(0) \begin{pmatrix}
    \bar{u}'(x) \\ \bar{u}''(x) \\ \bar{u}'''(x)
    \end{pmatrix}
    \wedge 
    \begin{pmatrix}
        1 \\ - \mu_3 (0) \\ 0
    \end{pmatrix} \\
    &= \frac{k_- p}{2\sqrt{2} \gamma}
    \Big{\{}2 \frac{\gamma}{p} \bar{u}''(x) + \bar{u}'''(x) \Big{\}},
\end{aligned}
\end{equation*}
where we've used the relations $\mu_3 (0) = 2 \gamma/p$ 
and $\kappa (0) = p/(2\sqrt{2}\gamma)$.

We see from this that in order to identify crossing points along the 
right shelf, we need to look for roots of the relation 
$\frac{2 \gamma}{p} \bar{u}''(x)+\bar{u}'''(x)$. Using 
(\ref{gkdv-wave}), we can compute 
\begin{equation} \label{gkdv-ubar-prime}
    \bar{u}'(x)
    = \frac{2 \alpha \gamma}{p} \sech^{\frac{2}{p}-1} (\gamma x)
    (- \sech (\gamma x) \tanh (\gamma x))
    = - \frac{2 \gamma}{p} \bar{u} (x) \tanh (\gamma x)
\end{equation}
In addition, by integrating (\ref{gkdv-wave-equation}), we obtain 
the relation 
\begin{equation*}
    \bar{u}'' = s \bar{u} - \frac{\bar{u}^{p+1}}{p+1},
\end{equation*}
and from (\ref{gkdv-wave-equation}) itself we can write 
\begin{equation*}
    \bar{u}'''
    = (s - \bar{u}^p) \bar{u}'
    = - \frac{2 \gamma}{p} \bar{u} (s - \bar{u}^p) \tanh (\gamma x).
\end{equation*}
Upon combining these relations, we see that 
\begin{equation} \label{gkdv-relation}
\frac{2 \gamma}{p} \bar{u}''(x)+\bar{u}'''(x)
= \frac{2 \gamma}{p} \bar{u} (x)
\Big(s - s \tanh (\gamma x) + \bar{u}^p \tanh (\gamma x) - \frac{\bar{u}^p}{p+1} \Big).
\end{equation}
Using now the identity 
\begin{equation*}
    \bar{u}(x)^p = \alpha^p \sech^2 (\gamma x) 
    = \alpha^p (1 - \tanh^2 (\gamma x)),
\end{equation*}
we can express the quantity in large parentheses on the 
right-hand side of (\ref{gkdv-relation}) in terms of the variable 
$z = \tanh (\gamma x) \in (-1, 1)$. Precisely, we can write this 
expression as 
\begin{equation*} \label{quantity-in-curved-brackets}
    \Psi (z) = (1-z) \Big{\{} s + \alpha^p z (1+z) - \frac{\alpha^p (1+z)}{p+1} \Big{\}}. 
\end{equation*}
As expected, since $\lambda = 0$ is an eigenvalue, we have a crossing point
in the limit as $x \to + \infty$ (corresponding with $z \to 1$), and otherwise
there is a one-to-one relationship between zeros of the polynomial 
$\Psi (z)$ and crossing points for $\tilde{\omega}_1^+ (x; 0)$ for $x \in \mathbb{R}$. 
The quantity in curved brackets in (\ref{quantity-in-curved-brackets}) 
is a quadratic in $z$, which can be expressed as 
\begin{equation} \label{gkdv-quadratic}
    \Psi_2 (z) = \alpha^p z^2 + \frac{1}{2} ps (p+2) z - \frac{ps}{2}.
\end{equation}
Writing out $\alpha^p$, we see that we're looking for roots of 
the quadratic expression 
\begin{equation*}
    \frac{1}{2} s (p+1) (p+2) z^2 + \frac{1}{2} ps (p+2) z - \frac{ps}{2}=0
    \implies (p+1) (p+2) z^2 + p (p+2) z - p = 0.
\end{equation*}
Checking the right-hand side at the values $z = -1, 0, 1$, we respectively 
obtain the values $2$, $-p$ and $2(p+1)^2$, from which we see that 
$\Psi_2 (z)$ has two real roots $z_1$ and $z_2$ on the interval $(-1,1)$, 
and moreover that $z_1 < 0 < z_2$. Correspondingly, $\tilde{\omega}_1^+ (x; 0)$ has two real roots 
$x_1$ and $x_2$ on $\mathbb{R}$, with $x_1 < 0 < x_2$. 

In order to understand possible contributions to the hyperplane index
from the limits $x \to \pm \infty$, we work with the scaled map
\begin{equation*}
    \tilde{\psi}_1^+ (x; 0)
    = \frac{\tilde{\omega}_1^+ (x; 0)}{|\eta^- (x; 0)| |\tilde{\mathcal{V}}^+ (0)|},
\end{equation*}
for which we've seen in (\ref{psi1-tilde-plus-minus-defined}) has the 
well-defined asymptotic limit
\begin{equation*}
    \tilde{\psi}_1^{+,-} (0) 
    = \lim_{x \to -\infty} \tilde{\psi}_1^+ (x; 0)
    = \frac{v^- (0) \wedge \tilde{\mathcal{V}}^+ (0)}{|v^- (0)| |\tilde{\mathcal{V}}^+ (0)|}
    =: c_0,
\end{equation*}
where according to our scaling convention $c_0$ is the positive constant 
\begin{equation*}
    c_0 = \frac{1}{|v^- (0)| |\tilde{\mathcal{V}}^+ (0)|}.
\end{equation*}
In this case, the wave $\bar{u} (x)$ is symmetric about $x = 0$, 
allowing us to see that (due to (\ref{gkdv-etamat0})) 
\begin{equation*}
    \lim_{x \to + \infty} \frac{\eta^- (x; 0)}{|\eta^- (x; 0)|}
    = - \frac{v^+ (0)}{|v^+ (0)|},
\end{equation*}
where 
\begin{equation*}
    v^+ (0) = \kappa (0)
    \begin{pmatrix}
        1 \\ - \mu_3 (0) \\ \mu_3 (0)^2
    \end{pmatrix}
    = \frac{p}{2\sqrt{2} \gamma} \begin{pmatrix}
        1 \\ - 2\gamma/p \\ 4 \gamma^2/p^2
    \end{pmatrix}.
\end{equation*}
We see that 
\begin{equation*}
\begin{aligned}
    \tilde{\psi}_1^{+,+} (0) 
    &= \lim_{x \to +\infty} \tilde{\psi}_1^+ (x; 0)
    = - \frac{v^+ (0) \wedge \tilde{\mathcal{V}}^+ (0)}{|v^+ (0)| |\tilde{\mathcal{V}}^+ (0)|} \\
    &= - \frac{ \kappa (0)^2 }{|v^+ (0)| |\tilde{\mathcal{V}}^+ (0)|}   
    \begin{pmatrix}
        1 \\ - 2\gamma/p \\ 4 \gamma^2/p^2
    \end{pmatrix}
    \wedge 
    \begin{pmatrix}
        1 \\ -2 \gamma/p \\ 0
    \end{pmatrix}
    = 0.
\end{aligned}
\end{equation*}
The graph of the function $\tilde{\psi}_1^+ (x; 0)$ is depicted 
in Figure \ref{gkdv-figure1}. 

\begin{figure}[ht]
\begin{center}
\begin{tikzpicture}
\draw[thick, <->] (-5,0) -- (5,0);
\node at (4.5,-.4) {$x$};
\draw[thick,<->] (0,-2) -- (0,2);
\node at (.15,2.5) {$\tilde{\psi}_1^+ (x; 0)$};
\draw[thick] (-3,1) .. controls (0,-4) and (1,3) .. (3,.75);
\draw[thick,->] (-3,1) .. controls (-3.5,1.6) and (-3.7,1.5) .. (-5,1.5); 
\draw[thick,->] (3,.75) .. controls (3.6,.1) .. (5,.1);
\draw[dashed,<->] (-5,1.6) -- (5,1.6);
\node at (-2.3,-.4) {$x_1$};
\node at (.6,-.4) {$x_2$};
\node at (.3,1.3) {$c_0$};
\end{tikzpicture}
\end{center}
\vspace{-.8in}
\caption{The function $\tilde{\psi}_1^+ (x; 0)$ in the case of the generalized KdV equation.} 
\label{gkdv-figure1}
\end{figure}

In order to compute the hyperplane index along the right shelf, 
we also need to specify $\tilde{\omega}_2^+ (x; 0)$, and for this 
we have considerable flexibility, particularly in our choice of 
the constant matrix $M$. For reasons that will become clear 
just below, we will find it convenient to take 
\begin{equation} \label{gkdv-M}
    M = \begin{pmatrix}
0 & 1 & 1/\sqrt{s} \\
1 & 0 & 0 \\
1 & 0 & - 1/s
\end{pmatrix}.
\end{equation}
We now set 
\begin{equation*}
    \tilde{\omega}_2^+ (x; 0)
     = \eta_- (x; 0) \wedge \tilde{\mathcal{V}}_M (0),
\end{equation*}
where $\tilde{\mathcal{V}}_M (0)$ is an eigenvector
of the asymptotic matrix $\tilde{\mathcal{A}}^+ (0)$
(from (\ref{tilde-mathcal-A-asymptotic})) associated
with the eigenvalue $-\mu_3 (0)$, namely
\begin{equation*}
    \tilde{\mathcal{V}}_M^+ (0)
    = \begin{pmatrix}
        -s/\mu_3 (0) + \sqrt{s} \\ -\mu_3 (0) + \sqrt{s} \\ 1
    \end{pmatrix}
    = \begin{pmatrix}
        0 \\ 0 \\ 1
    \end{pmatrix}.
\end{equation*}
It follows that 
\begin{equation*}
    \tilde{\omega}_2^+ (x; 0)
    = k_- \begin{pmatrix}
        \bar{u}' (x) \\ \bar{u}'' (x) \\ \bar{u}'''(x)
    \end{pmatrix}
    \wedge 
    \begin{pmatrix}
        0 \\ 0 \\ 1
    \end{pmatrix}
    = k_- \bar{u}'(x)
    = - \frac{2k_- \gamma}{p} \bar{u} (x) \tanh (\gamma x).
\end{equation*}

As with $\tilde{\omega}_1^+ (x; 0)$, in order to accommodate the limits 
$x \to \pm \infty$, we will work with the scaled map
\begin{equation}
    \tilde{\psi}_2^+ (x; 0)
    = \frac{\tilde{\omega}_2^+ (x; 0)}{|\eta^- (x; 0)| |\tilde{\mathcal{V}}^+ (0)|},
\end{equation}
for which we have the limits 
\begin{equation*}
\begin{aligned}
    \tilde{\psi}_2^{+,-} (0)
    &= \lim_{x \to -\infty} \tilde{\psi}_2^+ (x; 0)
    = \frac{v^- (0) \wedge \tilde{\mathcal{V}}_M^+ (0)}{|v^- (0)| |\tilde{\mathcal{V}}^+ (0)|} \\
    &= \frac{p}{2\sqrt{2}\gamma} 
    \begin{pmatrix}
        1 \\ 2 \gamma/p \\ 4 \gamma^2/p^2
    \end{pmatrix}
    \wedge 
     \begin{pmatrix}
        0 \\ 0 \\ 1
    \end{pmatrix}
    = \frac{p}{2\sqrt{2}\gamma} > 0,
\end{aligned}
\end{equation*}
and likewise 
\begin{equation*}
\begin{aligned}
    \tilde{\psi}_2^{+,+} (0)
    &= \lim_{x \to +\infty} \tilde{\psi}_2^+ (x; 0)
    = - \frac{v^+ (0) \wedge \tilde{\mathcal{V}}_M^+ (0)}{|v^+ (0)| |\tilde{\mathcal{V}}^+ (0)|} \\
    &= - \frac{p}{2\sqrt{2}\gamma} 
    \begin{pmatrix}
        1 \\ - 2 \gamma/p \\ 4 \gamma^2/p^2
    \end{pmatrix}
    \wedge 
     \begin{pmatrix}
        0 \\ 0 \\ 1
    \end{pmatrix}
    = - \frac{p}{2\sqrt{2}\gamma} < 0.
\end{aligned}
\end{equation*}
In total, we see that $\tilde{\psi}_2^+ (x; 0)$ is as depicted in Figure 
\ref{gkdv-figure2}.

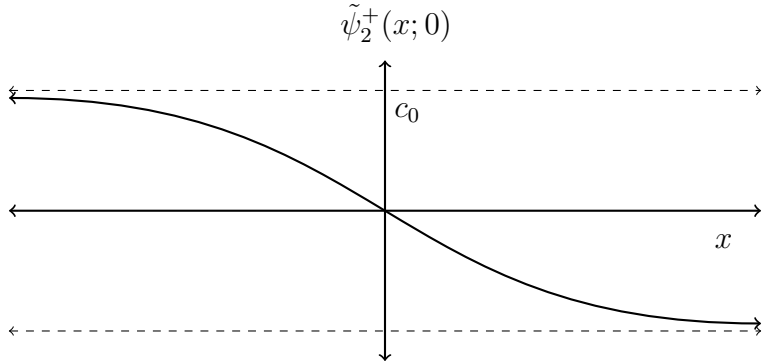
\begin{figure}[ht]
\begin{center}
\begin{tikzpicture}
\draw[thick, <->] (-5,0) -- (5,0);
\node at (4.5,-.4) {$x$};
\draw[thick,<->] (0,-2) -- (0,2);
\node at (.15,2.5) {$\tilde{\psi}_2^+ (x; 0)$};
\draw[thick,<->] (-5,1.5) .. controls (0,1.5) and (0,-1.5) .. (5,-1.5); 
\draw[dashed,<->] (-5,1.6) -- (5,1.6);
\draw[dashed,<->] (-5,-1.6) -- (5,-1.6);
\node at (.3,1.3) {$c_0$};
\end{tikzpicture}
\end{center}
\caption{The function $\tilde{\psi}_2^+ (x; 0)$ in the case of the generalized KdV equation.} 
\label{gkdv-figure2}
\end{figure}

Before computing the hyperplane index on the right shelf, we verify that we have 
invariance. First, the only roots of $\tilde{\psi}_1^+ (x; 0)$ are the values
designated as $x_1$ and $x_2$ above, and these satisfy $x_1 < 0 < x_2$. By 
contrast, the only real root of $\tilde{\psi}_2^+ (x; 0)$ is $x = 0$, so 
there is no value $x \in \mathbb{R}$ for which both $\tilde{\psi}_1^+ (x; 0)$
and $\tilde{\psi}_2^+ (x; 0)$ vanishe. In addition, since $\tilde{\psi}_2^{+, \pm} (0) \ne 0$,
invariance is not lost at the asymptotic endstates. 

In order to compute the hyperplane index along the right 
shelf, we need only identify all conjugate points on the interval 
$[-\infty,+\infty]$ (allowing $\pm \infty$ to serve asymptotically as
crossing points, though $-\infty$ has already been ruled out 
by our analysis of the bottom shelf) and assign a direction to each. 
For any 
finite crossing point $x_*$, this direction is determined 
by the ratio 
\begin{equation*}
    \frac{\tilde{\psi}_1^{+\,\prime} (x_*; 0)}{\tilde{\psi}_2^+ (x_*; 0)},
\end{equation*}
while for the crossing point at $+ \infty$ we will work directly 
with the rotation described by $p^+ (x; 0)$ as $x$ tends toward $+\infty$. 
For the former, we see immediately from Figure \ref{gkdv-figure1} that 
$\tilde{\psi}_1^{+ \, \prime} (x_1; 0) < 0$ 
and $\tilde{\psi}_1^{+\,\prime} (x_2; 0) > 0$, 
while from Figure \ref{gkdv-figure2} we see that 
$\tilde{\psi}_2^+ (x_1; 0) > 0$ and $\tilde{\psi}_2^+ (x_2; 0) < 0$. 
We can conclude that a direction of $-1$ should be assigned to each 
of these crossings. 

For the asymptotic crossing as $x \to + \infty$, we observe from 
Figure \ref{gkdv-figure1} that for all $x > 0$ sufficiently large 
we have $\tilde{\psi}_1^+ (x; 0) > 0$, while from 
Figure \ref{gkdv-figure2} we see that for all $x > 0$, 
$\tilde{\psi}_2^+ (x; 0) < 0$. This pairing
places the point $p^+ (x; 0)$ in the second quadrant, from which 
we see that the approach of $p^+ (x; 0)$ to the point $(-1,0)$
will be in the positive (i.e., in the counterclockwise) direction,
giving a contribution of $+1$ to the hyperplane index. Combining 
this with the negative crossings at $x_1$ and $x_2$, 
we conclude that in this case, the hyperplane index along the right shelf
satisfies the relation 
\begin{equation*}
    \ind (\mathpzc{g} (\cdot; 0), \mathpzc{h}^+ (0); [-\infty, + \infty])
    = - 1.
\end{equation*}

According to Theorem \ref{main-theorem}, we can conclude that {\it if} we 
have invariance along the top shelf then the  
count $\mathcal{N}_{\#} ((-\infty, 0])$ of non-positive eigenvalues
of (\ref{gkdv-evalue-problem}) satisfies the relation 
\begin{equation} \label{gkdv-preliminary-conclusion}
\mathcal{N}_{\#} ((-\infty, 0])
\ge \Big| \ind (\mathpzc{g} (\cdot; 0), \mathpzc{h}^+ (0); [-\infty, + \infty]) - \mathfrak{m} \Big|
= |-1 - \mathfrak{m}|. 
\end{equation}

In order to count the number of strictly negative eigenvalues, 
we would like to replace $\mathcal{N}_{\#} ((-\infty, 0])$ with 
$\mathcal{N}_{\#} ((-\infty, 0))$, and we have seen in Section \ref{evans-section}
how this can be accomplished via the Evans function. For the current
example, we have the two cases specified in (\ref{gkdv-evans-sign}). 
First, for $1\le p < 4$, we have $D'' (0) > 0$, from which we see that 
for $\lambda < 0$ sufficiently close to $0$ we must have 
$D (\lambda) > 0$ and consequently (via (\ref{evans-omega1-close})) 
$\tilde{\psi}_1^+ (x; \lambda) > 0$ for $x$ sufficiently large. 
In addition, we've seen that for the 
current application $\tilde{\psi}_2^+ (x; 0) < 0$ for $x > 0$. 
This pairing $\tilde{\psi}_1^+ (x; \lambda) > 0$ and 
$\tilde{\psi}_2^+ (x; 0) < 0$ places $p^+ (x; \lambda)$ in the second quadrant,
indicating that as $\lambda$ increases to $0$ $p (x; \lambda)$
rotates into $(-1,0)$ in the positive (i.e., the counterclockwise) 
direction, thus incrementing the hyperplane index by $+1$. 
As in the discussion of (\ref{detailed1}), this allows us to refine 
(\ref{gkdv-preliminary-conclusion}) to the statement 
\begin{equation*} 
\mathcal{N}_{\#} ((-\infty, 0)) + 1
\ge |1 + \mathfrak{m}|. 
\end{equation*}
Since we may have $\mathfrak{m} = 0$, this relation is consistent with 
spectral stability (i.e., consistent with the count $\mathcal{N}_{\#} ((-\infty, 0)) = 0$). 
Although we cannot conclude spectral stability from 
this calculation, we note that spectral stability is known to hold
in this case. (See p. 50 of \cite{PW1992} for discussion and 
references.)

On the other hand, for $p > 4$, we  have $D''(0) < 0$, with 
everything else as before, and the same
considerations described just above determine that in this case 
the relation (\ref{gkdv-preliminary-conclusion}) can be
refined to 
\begin{equation} \label{gkdv-conclusion2}
\mathcal{N}_{\#} ((-\infty, 0)) 
\ge |1 + \mathfrak{m}|. 
\end{equation}
If the boundary invariant $\mathfrak{m}$ is an even number, as suggested
by Lemma \ref{invariance-lemma} (see also Remark \ref{mathfrak-m-remark}
just below), then we can conclude that there is an unstable eigenvalue 
in this case, and so $\bar{u} (x)$ must be spectrally unstable. 

In summary, we expect $\bar{u} (x)$ to be stable for $1 \le p < 4$ and 
unstable for $p > 4$. This is precisely the conclusion of \cite{PW1992},
obtained there in the following way. For $1 \le p < 4$, we've seen 
that $D (\lambda) > 0$ for $\lambda < 0$ sufficiently close to $0$,
and we also know from Proposition \ref{evans-large-lambda} that 
$D(\lambda) \to +1$ as $\lambda \to - \infty$. This arrangement
is consistent with an absence of real roots of $D(\lambda)$ on 
$(-\infty, 0)$, and so consistent with the case of spectral stability.
(Of course, an even number of eigenvalues is possible, so no positive
conclusion can be reached based on this calculation.)
Likewise, if $p > 4$, then $D''(0) < 0$ and so $D (\lambda) < 0$
for $\lambda < 0$ sufficiently close to $0$. Since we still have 
the limit
$D(\lambda) \to 1$ as $\lambda \to \infty$, this arrangement 
guarantees that $D(\lambda)$ has a least one real root on the interval 
$(-\infty, 0)$, so we certainly have spectral instability. 

\begin{remark} \label{mathfrak-m-remark}
As noted in Lemma \ref{invariance-lemma}, under fairly 
general conditions we have that $\mathfrak{m}$ is an 
even number. For the current application, we can 
verify this rigorously by combining the analysis 
of \cite{PW1992} described just above with our 
analysis of the remaining shelves of the Maslov box. 
\end{remark}

At this point, we have carried out the full analysis required to reach 
our conclusions, but in order to illustrate how the method is working, 
we provide numerically generated depictions of the Maslov box 
for two cases, one stable and one unstable. For these calculations,
we will work with 
\begin{equation*}
    \tilde{\omega}_1^+ (x; \lambda)
    = \eta^- (x; \lambda) \wedge \tilde{\mathcal{V}}^+ (\lambda),
\end{equation*}
where we recall that 
\begin{equation*}
\tilde{\mathcal{V}}^+ (\lambda)
= \kappa(\lambda) \begin{pmatrix}
    1 \\ - \mu_3 (\lambda) \\ - \lambda/\mu_3 (\lambda) 
\end{pmatrix},
\end{equation*}
with $\kappa (\lambda)$ serving as the scaling constant specified
in (\ref{gkdv-scaling}). We have, then,  
\begin{equation} \label{gkdv-tilde-omega1}
    \tilde{\omega}_1^+ (x; \lambda) 
    = \kappa (\lambda) 
    \Big{\{} - \frac{\lambda}{\mu_3 (\lambda)} \eta_1^- (x; \lambda)
    + \mu_3 (\lambda) \eta_2^- (x; \lambda) + \eta_3^- (x; \lambda) 
    \Big{\}}.
\end{equation}

Likewise, we set 
\begin{equation*}
    \tilde{\omega}_2^+ (x; 0)
     = \eta^- (x; \lambda) \wedge \tilde{\mathcal{V}}_M (\lambda),
\end{equation*}
where $\tilde{\mathcal{V}}_M (\lambda)$ is an eigenvector
of the asymptotic matrix $\tilde{\mathcal{A}}_+ (\lambda)$
(from (\ref{tilde-mathcal-A-asymptotic})) associated
with the eigenvalue $-\mu_3 (\lambda)$, namely
\begin{equation*}
    \tilde{\mathcal{V}}_M^+ (\lambda)
    = \begin{pmatrix}
        -s/\mu_3 (\lambda) + \sqrt{s} \\ -\mu_3 (\lambda) + \sqrt{s} \\ 1
    \end{pmatrix},  
\end{equation*}
where no specific normalization is required. It follows that 
\begin{equation} \label{gkdv-tilde-omega2}
    \tilde{\omega}_2^+ (x; \lambda)
    = k_- \Big{\{} \eta_1^- (x; \lambda) 
    - \eta_2^- (x; \lambda) (-\mu_3 (\lambda) + \sqrt{s})
    + \eta_3^- (x; \lambda) (-s/\mu_3 (\lambda) + \sqrt{s})\Big{\}}.
\end{equation}
 
Using (\ref{gkdv-tilde-omega1}) and (\ref{gkdv-tilde-omega2}), we can now
generate spectral curves throughout a fixed Maslov box
by numerically generating $\eta^- (x; \lambda)$ throughout the 
box. As an example of the stable case, we will carry this out for 
$p = 7/2$ and $s = 1/2$, using $[-7, 0] \times [-5, 5]$ as the 
(truncated) Maslov box. (See Figure \ref{kdvbox3-figure}.) We've seen in our
analytic calculation that each crossing on the right shelf 
gives a contribution of $-1$ to the hyperplane index. In 
addition, we've seen that there is an additional contribution 
of $+1$ obtained in the limit as $x \to + \infty$, but this is 
never picked up on any box truncated in the $x$-direction. Finally,
by using the Evans function, we were able to show that for $c$ 
sufficiently large the tracking point $p^+ (c; \lambda)$
rotates in the clockwise direction as $\lambda$ decreases from 
$0$. We see that as $x$ increases toward $c$, $p^+ (x; 0)$
rotates toward $(-1,0)$ in the counterclockwise direction without
ever arriving at $(-1,0)$, and then when $x = c$ and $\lambda$ 
is decreased from $0$, the point $p^+ (c; \lambda)$ 
rotates back in the clockwise direction so that the corner point
at $x = c$ and $\lambda = 0$ does not increment the hyperplane
index. In total, we can conclude that in this case 
$\mathfrak{m} = -2$. In order to have such a value for 
$\mathfrak{m}$, there must be at least one point at which 
invariance is lost in the open box $(-7,0) \times (-5,5)$, 
and since one condition for loss of invariance is 
$\tilde{\omega}_1^+ (x; \lambda) = 0$,
this point must occur along the spectral curve (depicted 
in red in Figure \ref{kdvbox3-figure}). Numerically 
searching along this curve for zeros of $\tilde{\omega}_2^+ (x; \lambda)$,
we find that invariance seems to be lost at about $x = .348$
and $\lambda = -1.706$ (working with increments $.001$ in both 
$x$ and $\lambda$). At first glance, it may seem that the turnaround
point of the spectral curve at $(\lambda, x) \cong (-3.4,0)$
is a likely candidate for the point at which invariance is lost,
but this depends entirely on the choice of 
$\tilde{\omega}_2^+ (x; \lambda)$. In particular, as we've 
seen in the proof of Lemma \ref{invariance-lemma}, the 
point of invariance can always be changed by changing 
the choice of $\tilde{\omega}_2^+ (x; \lambda)$. 

\begin{figure}[ht] 
\begin{center}\includegraphics[%
  width=11cm,
  height=8.2cm]{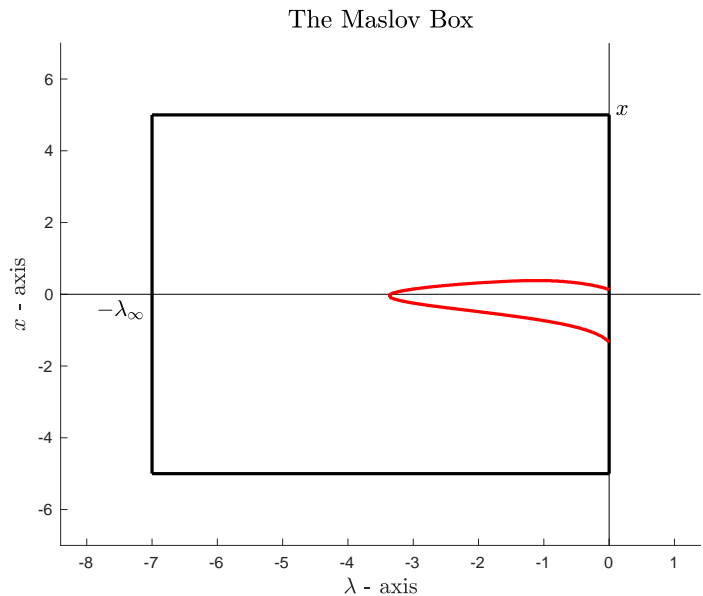}\end{center}
\caption{Maslov Box associated with (\ref{gkdv-wave-equation}) for 
$s = 1/2$ and $p = 7/2$. \label{kdvbox3-figure}}
\end{figure}

For the unstable case, we'll again take $s = 1/2$, this time 
with $p = 9/2$. Proceeding similarly as in the previous case, 
we numerically generate the spectral curves for the Maslov
box $[-7, 0] \times [-5, 5]$. (See Figure \ref{kdvbox2b-figure}.)  
In this case, we see a compressed spectral curve in the upper 
right corner of the Maslov box, and in order to clarify the behavior 
there, we include a second Maslov box on $[-1, 0] \times [-5, 5]$.
(See Figure \ref{kdvbox2a-figure}.) As in the previous case with $p = 7/2$,
the crossings along the right shelf each contribute $-1$ to the 
hyperplane index. As $x$ increases to a sufficiently large value, 
the tracking point $p^+ (x; 0)$ moves toward 
$(-1,0)$ in the counterclockwise direction. In addition, we have 
seen from our analysis of the Evans function that for $c$ sufficiently 
large, $p^+ (c; \lambda)$ rotates away from $(-1, 0)$ in the 
counterclockwise direction as $\lambda$ decreases from 0. 
In this way, we see that for $c$ sufficiently large, 
the hyperplane index increments by $+1$ 
as $p^+ (x; \lambda)$ traverses the corner at $x = c$ 
and $\lambda = 0$. Finally, there is a contribution of $-1$ 
from the eigenvalue at $\lambda = - .0959$ (with an increment
in the calculation of $.0001$). As with the case with $p = 7/2$,
we can conclude that $\mathfrak{m} = -2$. Once again, we see 
that invariance is lost for at least one point, and computing
numerically we approximate this point as $x = - .286$ and 
$\lambda = - 4.563$ (with a step size of $.001$ in both 
$x$ and $\lambda$). 

\begin{figure}[ht]  
\begin{center}\includegraphics[%
  width=11cm,
  height=8.2cm]{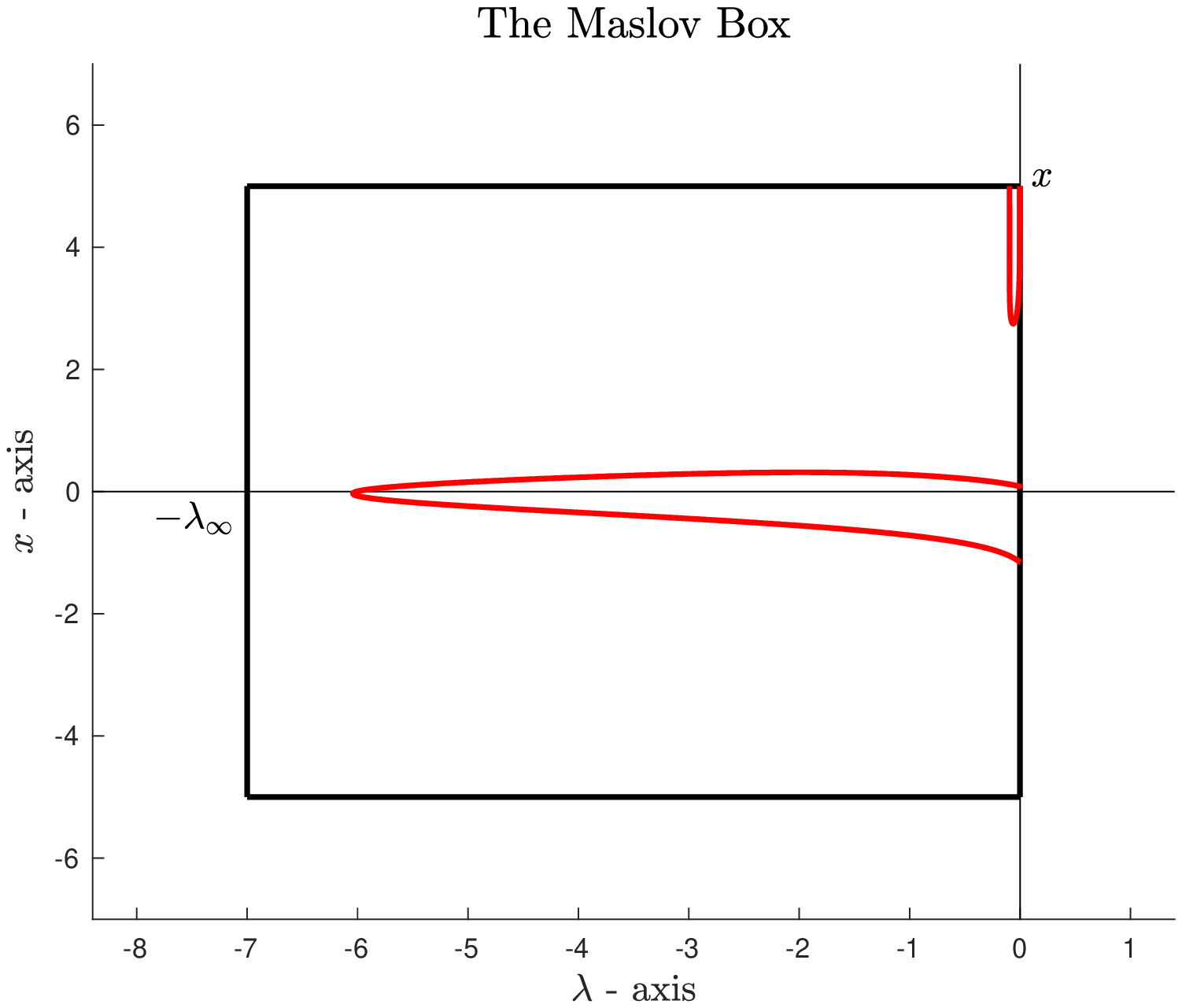}\end{center}
\caption{Maslov Box associated with (\ref{gkdv-wave-equation}) for 
$s = 1/2$ and $p = 9/2$. \label{kdvbox2b-figure}}
\end{figure}

\begin{figure}[ht]  
\begin{center}\includegraphics[%
  width=12cm,
  height=8.2cm]{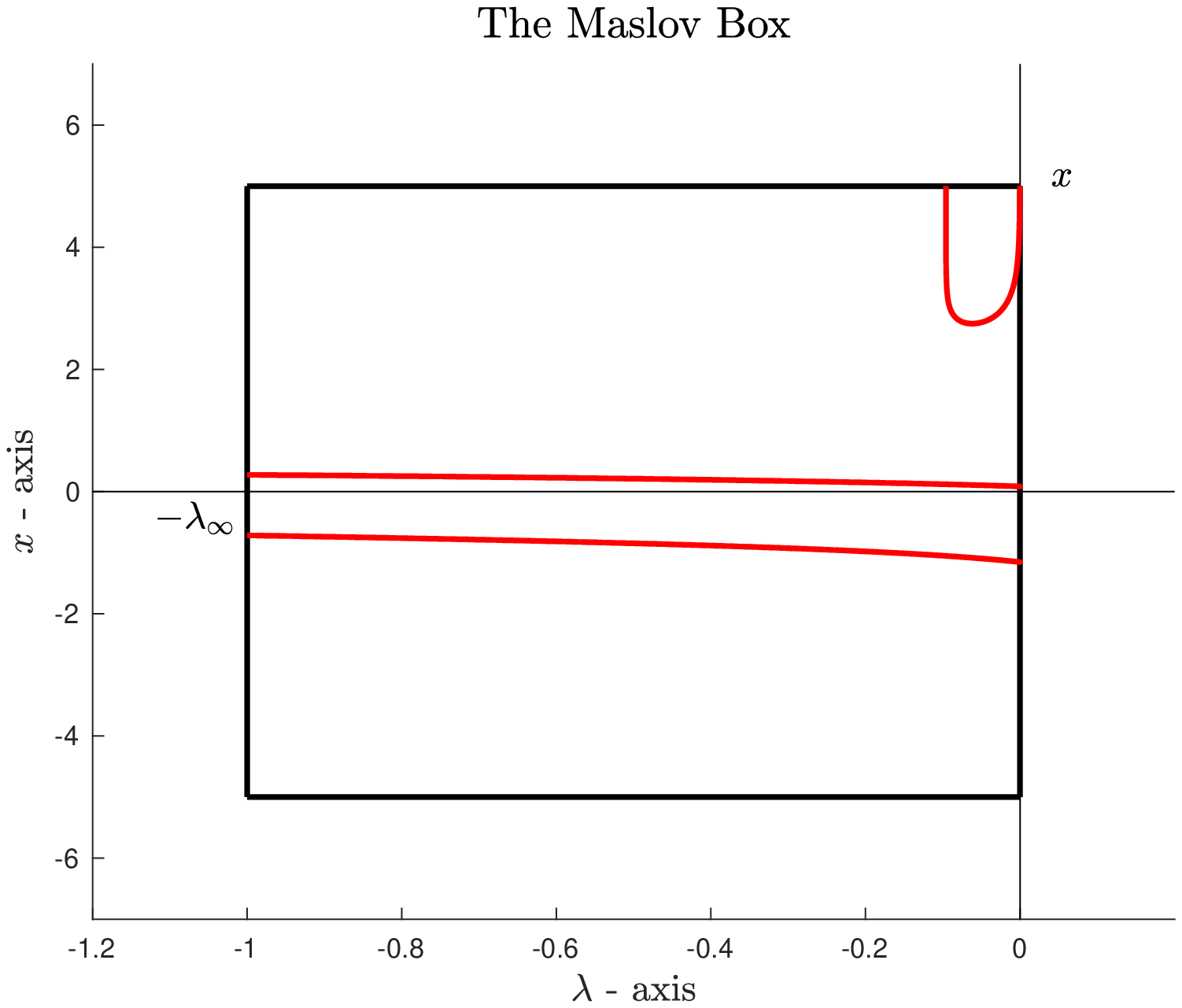}\end{center}
\caption{Additional detail for the Maslov Box associated with (\ref{gkdv-wave-equation}) for 
$s = 1/2$ and $p = 9/2$. \label{kdvbox2a-figure}}
\end{figure}

\subsection{The Korteweg-de Vries-Burgers Equation}
\label{kdvb-section}

For our second application, we consider the KdV-Burgers equation
\begin{equation} \label{kdvb-equation}
u_t + u u_x = u_{xx} + \nu u_{xxx},    
\end{equation}
where $\nu \in \mathbb{R}$ is fixed. It's known
(see, e.g., \cite{PSW1993})
that there exist stationary solutions 
$\bar{u} (x)$ for (\ref{kdvb-equation}), which 
satisfy the asymptotic conditions 
\begin{equation} \label{kdvp-endstates}
    \lim_{x \to \pm \infty} \bar{u} (x) 
    = \mp 1.
\end{equation}
Such solutions satisfy the ODE
\begin{equation} \label{kdvb-wave-equation}
    \nu \bar{u}''' + \bar{u}''
    - \bar{u} \bar{u}' = 0,
\end{equation}
which we can integrate to 
\begin{equation} \label{kdvb-wave-equation-integrated}
    \nu \bar{u}'' + \bar{u}'
    - \frac{1}{2} (\bar{u}^2 - 1) = 0.
\end{equation}
For $|\nu| \le 1/4$ these solutions 
are known to be monotonic, while for $|\nu| > 1/4$ they 
are known to oscillate as $x$ tends to 
$+ \infty$. Such a wave, generated numerically, 
is depicted in Figure \ref{kdvb-wave-figure} for $\nu = 10$. For 
specificity, we will take $\nu > 0$ throughout 
our calculations, noting that the case $\nu < 0$
can be addressed similarly. 

\begin{figure}[ht] 
\begin{center}\includegraphics[%
  width=11cm,
  height=8.2cm]{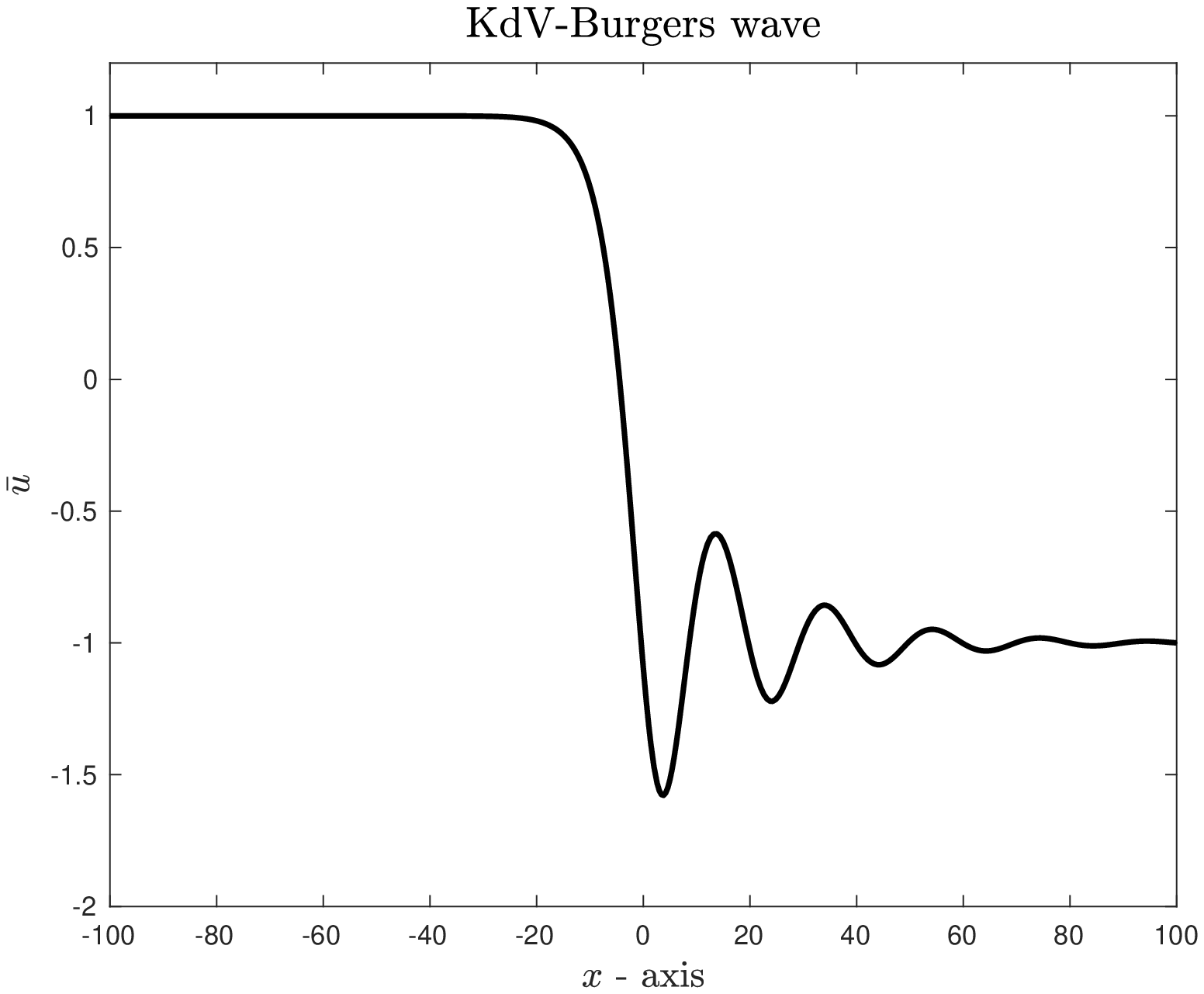}\end{center}
\caption{Stationary solution $\bar{u} (x)$ for (\ref{kdvb-equation}) with $\nu = 10$. \label{kdvb-wave-figure}}
\end{figure}

\begin{remark} \label{PSW1993-remark}
    In Theorem 1.1 of \cite{PSW1993}, the authors consider a more general 
    equation than (\ref{kdvb-wave-equation-integrated}), and with a different
    scaling, namely 
    \begin{equation} \label{PSW1993-equation}
        - c \phi + \frac{1}{p+1} \phi^{p+1} + \phi'' = \alpha \phi',
    \end{equation}
    where $p \ge 1$ and $c$ and $\alpha$ are taken to be positive 
    constants. Our equation (\ref{kdvb-wave-equation-integrated}) can 
    be obtained from (\ref{PSW1993-equation}) in the case $\nu > 0$
    by setting $c = 1$, $\alpha = 1/\sqrt{\nu}$, and 
    \begin{equation*}
        \bar{u} (x) = - \phi (-x/\sqrt{\nu}) + 1,
    \end{equation*}
    and for $\nu < 0$ by setting $c = 1$, $\alpha = 1/\sqrt{-\nu}$, and 
     \begin{equation*}
        \bar{u} (x) = \phi (x/\sqrt{-\nu}) - 1.
    \end{equation*}
\end{remark}

If we linearize (\ref{kdvb-equation}) about the wave 
$\bar{u} (x)$, we obtain the eigenvalue problem 
\begin{equation} \label{kdvb-evp}
    L \phi
    = - \nu \phi''' - \phi'' + (\bar{u} (x) \phi)'
    = \lambda \phi,
\end{equation}
where the sign has again been chosen so that eigenvalues 
with a negative real part signify spectral instability. 
According to (\ref{kdvp-endstates}),
the associated asymptotic problems are 
\begin{equation} \label{kdvb-asymptotic-problem}
- \nu \phi''' - \phi'' \pm \phi' = \lambda \phi.    
\end{equation}
According to \cite{He1981, KP2013}, we can understand the 
essential spectrum of $L$ by looking for solutions 
to (\ref{kdvb-asymptotic-problem}) of the form
$\phi (x; k) = e^{i k x}$. We find 
\begin{equation*}
    \sigma_{\ess} (L)
    = \{\lambda \in \mathbb{C}: \lambda = i \nu k^3 + k^2 \pm ik, \, k \in \mathbb{R}\},
\end{equation*}
from which it's clear that the essential spectrum of $L$ is confined
to the right complex half-plane along with $\lambda = 0$.

In order to express (\ref{kdvb-evp}) in our standard form 
(\ref{nonhammy}), we set $y_1 = \phi$, $y_2 = \phi'$, and 
$y_3 = \nu \phi''$, so that 
\begin{equation} \label{kdvb-system}
    y' = A (x; \lambda) y,
    \quad
    A (x; \lambda) = 
    \begin{pmatrix}
    0 & 1 & 0 \\
    0 & 0 & 1/\nu \\
    \bar{u}'(x) - \lambda & \bar{u} (x) & - 1/\nu
    \end{pmatrix},
\end{equation}
with the corresponding asymptotic matrices
\begin{equation*}
    A_{\pm} (\lambda)
    := \lim_{x \to \pm \infty} A (x; \lambda)
    = \begin{pmatrix}
    0 & 1 & 0 \\
    0 & 0 & 1/\nu \\
    -\lambda & \mp 1 & - 1/\nu
    \end{pmatrix}.
\end{equation*}
The eigenvalues of $A_{\pm} (\lambda)$ are easily seen to be roots
$\mu$ of the function
\begin{equation} \label{kdvb-h}
    h_{\pm} (\mu; \lambda)
    = \mu^3 + \frac{1}{\nu} \mu^2 \pm \frac{1}{\nu} \mu + \frac{\lambda}{\nu}.
\end{equation}
For $\lambda = 0$, the roots of $h_{\pm} (\mu; 0)$ are readily computed,
\begin{equation*}
    \mu = 0, \quad
    \frac{1}{2 \nu} (-1 -  \sqrt{1 \mp 4 \nu}), \quad
    \frac{1}{2 \nu} (-1 +  \sqrt{1 \mp 4 \nu}).
\end{equation*}
Recalling that we're taking $\nu > 0$, we see that for $h_- (\mu; 0)$
these roots are naturally ordered as 
\begin{equation*}
\mu_1^- (0) = \frac{1}{2 \nu} (-1 -  \sqrt{1 + 4 \nu}) < 0, \quad
\mu_2^- (0) = 0, \quad 
\mu_3^- (0) = \frac{1}{2 \nu} (-1 +  \sqrt{1 + 4 \nu}) > 0.
\end{equation*}
As $\lambda$ decreases from $0$, the graph of the cubic function
$h_- (\mu; 0)$  will lower so that the values $\mu_1^- (\lambda)$
and $\mu_2^- (\lambda)$ will approach one another, while $\mu_3^- (\lambda)$
will increase. As $\lambda$ continues to 
decrease, the roots $\mu_1^- (\lambda)$ and $\mu_2^- (\lambda)$ will
coalesce at some value $\lambda = \lambda_c$ into a complex conjugate pair.  
For $\lambda$ above this coalescence value (i.e., for $\lambda \in (\lambda_c, 0]$) 
we can associate an eigenvector $v_j^- (\lambda)$ with each $\mu_j^- (\lambda)$, $j \in \{1, 2, 3\}$,
\begin{equation} \label{kdvb-asymptotic-eigenvectors-left}
v_j^- (\lambda) = (1 \,\, \mu_j^- (\lambda) \,\, \nu \mu_j^- (\lambda)^2)^T,
\quad j = 1, 2, 3.
\end{equation}
The eigenvector $v_3^- (\lambda)$ has the same form 
for all $\lambda \le 0$, and up to a scaling factor is 
the eigenvector $v^- (\lambda)$ from Proposition
\ref{ode-proposition1}. 

Proceeding similarly for $h_+ (\mu; \lambda)$, we first observe that in 
this case $\re \sqrt{1-4\nu} < 1$, so we have the ordering 
\begin{equation} \label{kdv-right-eigenvalues}
\mu_1^+ (0) = \frac{1}{2 \nu} (-1 -  \sqrt{1 - 4 \nu}), \quad
\mu_2^+ (0) = \frac{1}{2 \nu} (-1 +  \sqrt{1 - 4 \nu}), \quad
\mu_3^+ (0) = 0,
\end{equation}
where $\mu_1^+ (0)$ and $\mu_2^+ (0)$ both have negative real part. 
If $\nu > 1/4$, then $\mu_1^+ (0)$ and $\mu_2^+ (0)$ will comprise
a complex conjugate pair, while if $0 < \nu < 1/4$, then 
$\mu_1^+ (\lambda)$ and $\mu_2^+ (\lambda)$
will coalesce into such a pair as $\lambda$ decreases. In either 
case, the value $\mu_3^+ (\lambda)$ will remain real and increasing
as $\lambda$ decreases from 0. (The borderline case $\nu = 1/4$
requires additional work and won't be considered in the current 
analysis.)   

We next consider $\tilde{A} (x; \lambda)$ (from (\ref{tildeA})), 
which in this case is 
\begin{equation} \label{kdvb-tildeA}
    \tilde{A} (x; \lambda)
    = \begin{pmatrix}
    1/\nu & 1/\nu & 0 \\
    \bar{u} (x) & 0 & 1 \\
    \lambda - \bar{u}' (x) & 0 & 0
    \end{pmatrix},
\end{equation}
with the corresponding asymptotic matrix 
\begin{equation} \label{kdvb-Atilde-asymptotic}
    \tilde{A}_+ (\lambda)
    = \lim_{x \to + \infty} \tilde{A} (x; \lambda)
    = \begin{pmatrix}
    1/\nu & 1/\nu & 0 \\
    -1 & 0 & 1 \\
    \lambda & 0 & 0
    \end{pmatrix}.
\end{equation}
The unique left-most eigenvalue of $\tilde{A}_+ (\lambda)$
is $\tilde{\mu}_1^+ (\lambda) = - \mu_3^+ (\lambda)$, with 
associated eigenvector 
\begin{equation} \label{kdvb-tildev-alternative}
\tilde{\mathcal{V}}_1^+ (\lambda)
= \begin{pmatrix}
1 \\ - \nu \mu_3^+ (\lambda)-1 \\ \nu \mu_3^+ (\lambda)^2 + \mu_3^+ (\lambda) + 1
\end{pmatrix},
\end{equation}
which up to a scaling factor is $\tilde{\mathcal{V}}^+ (\lambda)$ from 
Remark \ref{eigenvector-remark}. 

At this point, we have the pieces in place to verify that for any 
choice of $\lambda_{\infty} > 0$ the hyperplane index along the bottom shelf for the interval 
$[- \lambda_{\infty}, 0]$ gives no contribution. Namely, using the development
of Section \ref{bottom-and-left}, we can verify this
assertion if we can show that for all $\lambda \le 0$
\begin{equation} \label{kvdb-sought}
    v^- (\lambda) \wedge \tilde{\mathcal{V}}^+ (\lambda) 
    \ne 0.
\end{equation}
(In contrast with the previous case, since $A_- (\lambda) \ne A_+ (\lambda)$
here, relation (\ref{kvdb-sought}) isn't immediate.) 
Using our expressions just above for $v^- (\lambda)$ and 
$\tilde{\mathcal{V}}^+ (\lambda)$, with $\kappa (\lambda)$
and $\tilde{\kappa} (\lambda)$ serving respectively as scaling 
factors for $v^- (\lambda)$ and $\tilde{\mathcal{V}}^+ (\lambda)$, 
we can compute this wedge product to be 
\begin{equation*}
    \begin{aligned}
    v^- (\lambda) \wedge \tilde{\mathcal{V}}^+ (\lambda)
    &= \kappa (\lambda) \tilde{\kappa} (\lambda)
    \begin{pmatrix}
    1 \\ \mu_3^- (\lambda) \\ \nu \mu_3^- (\lambda)^2
    \end{pmatrix}
    \wedge
    \begin{pmatrix}
    1 \\ - \nu \mu_3^+ (\lambda)-1 \\ \nu \mu_3^+ (\lambda)^2 + \mu_3^+ (\lambda) + 1
    \end{pmatrix} \\
    &= \kappa (\lambda) \tilde{\kappa} (\lambda)
    \Big{\{} \nu \mu_3^+ (\lambda)^2 + \mu_3^+ (\lambda) + 1 
    - \mu_3^- (\lambda) ( - \nu \mu_3^+ (\lambda)-1) + \nu \mu_3^- (\lambda)^2 \Big{\}} \\
    &= \kappa (\lambda) \tilde{\kappa} (\lambda)
    \Big{\{} \nu \mu_3^+ (\lambda)^2 + \mu_3^+ (\lambda) + 1  
    + \nu \mu_3^- (\lambda) \mu_3^+ (\lambda) + \mu_3^- (\lambda) + \nu \mu_3^- (\lambda)^2 \Big{\}}.
    \end{aligned}
\end{equation*}
Recalling that $\mu_3^- (\lambda) > 0$ for all $\lambda \le 0$ and $\mu_3^+ (\lambda) \ge 0$
for all $\lambda \le 0$, we see that $v^- (\lambda) \wedge \tilde{\mathcal{V}}^+ (\lambda) > 0$ 
for all $\lambda \le 0$, verifying (\ref{kvdb-sought}). 

Before moving on, we observe that since $A_- (\lambda) \ne A_+ (\lambda)$
in this case, we cannot conclude immediately from Proposition \ref{left-shelf-crossings} 
that there are no crossings along the left shelf. Nonetheless, such a conclusion 
can be drawn from an argument
based on energy methods. Since that argument has a different flavor than the 
current considerations, it has been placed in an appendix. 

Next, we turn to the evaluation of $D(0)$ and $D' (0)$ 
(it will turn out that $D'' (0)$ isn't required). 
To this end, we first observe that since $A_- (\lambda)$ only 
has one positive eigenvalue, the solution $\eta^- (x; \lambda)$
described in Proposition \ref{ode-proposition1} is (up to a multiplicative
constant) the only solution of (\ref{kdvb-system})
that decays as $x$ tends to $-\infty$. In this way, we see that 
that there exists some constant $k_-$ so that 
\begin{equation} \label{etamat0}
\eta^- (x; 0) = k_-    
\begin{pmatrix}
\bar{u}' (x) \\
\bar{u}'' (x) \\
\nu \bar{u}''' (x)
\end{pmatrix}.
\end{equation}
According to Proposition \ref{ode-proposition1}, we can write 
\begin{equation*}
    \lim_{x \to -\infty} e^{- \mu_3^- (0) x} 
    k_-    
\begin{pmatrix}
\bar{u}' (x) \\
\bar{u}'' (x) \\
\nu \bar{u}''' (x)
\end{pmatrix}
=   \lim_{x \to -\infty} e^{- \mu_3^- (0) x} \eta^- (x; 0)
= v^- (\lambda).
\end{equation*}
Recalling that the first component of $v^- (\lambda)$ is positive 
and $\bar{u}' (x) < 0$ for $x \ll 0$, we conclude that 
$k_- < 0$. 

In addition, we need to understand the nature of $\tilde{\mathcal{Y}}^+ (x; 0)$,
which solves the ODE 
\begin{equation} \label{kdvb-tilde-system0}
\tilde{\mathcal{Y}}^{+\,\prime} 
= \tilde{A} (x; 0) \tilde{\mathcal{Y}}^+,
\end{equation}
where $\tilde{A} (x; \lambda)$ as as in 
(\ref{kdvb-Atilde-asymptotic}). 
Proceeding similarly as in (\ref{y1tilde00}) and (\ref{y1tilde0}),
we find that the first component of $\tilde{\mathcal{Y}}^+ (x; 0)$
satisfies the equation
\begin{equation} \label{kdvb-y1tilde}
    \nu \tilde{\mathcal{Y}}_1^{+\, \prime \prime \prime} - \tilde{\mathcal{Y}}_1^{+\,\prime \prime}
    - \bar{u} (x) \tilde{\mathcal{Y}}_1^{+\,\prime} = 0.
\end{equation}
It's clear that one solution of this equation is $\tilde{\mathcal{Y}}_1^+ (x; 0) \equiv 1$,
and upon substitution of this component into the full system (\ref{kdvb-tilde-system0})
we see that one family of solutions is 
\begin{equation*}
    \tilde{\mathcal{Y}}^+ (x; 0)
    = k_+ \begin{pmatrix}
    1 \\
    -1 \\
    - \bar{u} (x)
    \end{pmatrix},
\end{equation*}
for some constant $k_+$. 
Moreover, since all other (linearly independent) solutions to (\ref{kdvb-tilde-system0}) grow at 
exponential rate as $x \to + \infty$, this must indeed be the solution 
$\tilde{\mathcal{Y}}^+ (x; 0)$ we're seeking. Recalling from Proposition 
\ref{ode-proposition2} that 
\begin{equation*}
    \lim_{x \to \infty} e^{- \mu_3^+ (0) x} \tilde{\mathcal{Y}}^+ (x; 0)
    = \tilde{\mathcal{V}}^+ (0),
\end{equation*}
we see that we must have $k_+ > 0$. We can now directly compute 
\begin{equation*}
    \begin{aligned}
    D(0)
    &= \eta^- (x;0) \wedge \tilde{\mathcal{Y}}^+ (x; 0)
    = k_- k_+ \begin{pmatrix}
\bar{u}' (x) \\
\bar{u}'' (x) \\
\nu \bar{u}''' (x)
\end{pmatrix}
\wedge
\begin{pmatrix}
1 \\
-1 \\
- \bar{u}(x)
\end{pmatrix} \\
&=  k_- k_+ \Big{\{} \nu \bar{u}'''(x) + \bar{u}''(x) - \bar{u} (x) \bar{u}'(x) \Big{\}} = 0,
    \end{aligned}
\end{equation*}
where the final equality holds by the specification of $\bar{u} (x)$ 
as a stationary solution to (\ref{kdvb-equation}) (i.e., from 
(\ref{kdvb-wave-equation-integrated})). 

Using 
\begin{equation*}
    A_{\lambda} (x; 0)
    = \begin{pmatrix}
    0 & 0 & 0 \\
    0 & 0 & 0 \\
    -1 & 0 & 0
    \end{pmatrix},
\end{equation*}
along with the relations above for $\eta^- (x; 0)$ and 
$\tilde{\mathcal{Y}}^+ (x; 0)$, we find that 
\begin{equation*}
\Big(A_{\lambda} (x; 0) \eta^- (x; 0)\Big) \wedge \tilde{\mathcal{Y}}^+ (x; 0)
= - k_- k_+ \bar{u}'(x),
\end{equation*}
from which we conclude (using the expression for $D'(0)$ in Proposition \ref{evans-proposition-derivatives})
\begin{equation*}
    D' (0) = - k_- k_+ \int_{-\infty}^{+\infty} \bar{u}' (x) dx
    = - k_- k_+ (u_+ - u_-) = 2 k_- k_+ < 0.
\end{equation*}

We now compute the hyperplane index along the right shelf.
First, we will detect crossing points with 
the function 
\begin{equation*}
\begin{aligned}
    \tilde{\omega}_1^+ (x; 0)
    &= \eta^- (x; 0) \wedge \tilde{\mathcal{V}}^+ (0)
    = k_- \tilde{\kappa} (0)   
\begin{pmatrix}
\bar{u}' (x) \\
\bar{u}'' (x) \\
\nu \bar{u}''' (x)
\end{pmatrix}
\wedge
\begin{pmatrix}
1 \\ - \nu \mu_3^+ (0)-1 \\ \nu \mu_3^+ (0)^2 + \mu_3^+ (0) + 1
\end{pmatrix} \\
&=  k_-  \tilde{\kappa} (0)   
\begin{pmatrix}
\bar{u}' (x) \\
\bar{u}'' (x) \\
\nu \bar{u}''' (x)
\end{pmatrix}
\wedge
\begin{pmatrix}
1 \\ -1 \\ 1
\end{pmatrix} 
= k_- \tilde{\kappa} (0) \Big( \nu \bar{u}'''(x) + \bar{u}''(x) + \bar{u}'(x) \Big).
\end{aligned}
\end{equation*}
We recall that $\nu \bar{u}'''(x) + \bar{u}''(x) = \bar{u} (x) \bar{u}'(x)$,
so that 
\begin{equation} \label{kdvb-tilde-omega1}
  \tilde{\omega}_1^+ (x; 0)
  =  k_-  \tilde{\kappa} (0)   \bar{u}'(x) (\bar{u}(x) + 1). 
\end{equation}

For $0 < \nu < 1/4$, in which case $\bar{u} (x)$ decreases 
monotonically from $1$ to $-1$, there are no crossing points 
$x_* \in \mathbb{R}$ (though there is an asymptotic crossing point
at $+ \infty$). More interesting, for $\nu > 1/4$, oscillations 
lead to crossings at each critical point of $\bar{u} (x)$ and 
also at each value $x_* \in \mathbb{R}$ so that $\bar{u} (x_*) = -1$. 

In order to compute the hyperplane index along the right shelf, 
we also need to specify $\tilde{\omega}_2^+ (x; 0)$, and for this 
we have considerable flexibility, particularly in our choice of 
the constant matrix $M$. For reasons that will become clear 
just below, we will find it convenient to take 
\begin{equation} \label{kdvb-tilde-omega2-matrix}
    M = \begin{pmatrix}
    1 & 0 & 0 \\
    1 & 1 & 1 \\
    0 & 0 & 1
    \end{pmatrix}.
\end{equation}
We now set 
\begin{equation*}
    \tilde{\omega}_2^+ (x; 0)
     = \eta_- (x; 0) \wedge \tilde{\mathcal{V}}_M (0),
\end{equation*}
where $\tilde{\mathcal{V}}_M (0)$ is an eigenvector
of the asymptotic matrix $\tilde{\mathcal{A}}_+ (0)$
associated with the eigenvalue $-\mu_3 (0)$, where  
$\tilde{\mathcal{A}}_+ (0)$ is computed via 
(\ref{tilde-mathcal-A-asymptotic})) from the matrix
\begin{equation*}
    \mathcal{A}_+ (0)
    = M A_+ (0) M^{-1}
    = \begin{pmatrix}
    -1 & 1 & -1 \\
    0 & 0 & 0 \\
    1 & -1 & 1-1/\nu
    \end{pmatrix}.
\end{equation*}
We find 
\begin{equation*}
    \tilde{\mathcal{V}}_M^+ (0)
    = - \begin{pmatrix}
        - \mu_3^+ (0) \\ -\mu_3^+ (0) - 1/\nu \\ \mu_3^+ (0)^2 - (1-1/\nu) \mu_3^+ (0) 
    \end{pmatrix}
    = \begin{pmatrix}
        0 \\ 1/\nu \\ 0
    \end{pmatrix},
\end{equation*}
from which it follows that 
\begin{equation} \label{kdvb-tilde-omega20}
    \tilde{\omega}_2^+ (x; 0)
    = k_- \begin{pmatrix}
        \bar{u}' (x) \\ \bar{u}'' (x) \\ \bar{u}'''(x)
    \end{pmatrix}
    \wedge 
    \begin{pmatrix}
        0 \\ 1/\nu \\ 0
    \end{pmatrix}
    = - (k_-/\nu) \bar{u}''(x).  
\end{equation}

For the ensuing discussion, it will be necessary to understand 
$\bar{u}'' (x)$ for arbitrarily large values of $x$. For this, 
we first observe that for the eigenvalues $\mu_1^+ (0)$
and $\mu_2^+ (0)$ from (\ref{kdv-right-eigenvalues}) the associated
eigenvectors can be chosen to be
\begin{equation} \label{kdv-right-eigenvectors}
    v_i^+ (0) =
    \begin{pmatrix}
    1 \\ \mu_i^+ (0) \\ \nu \mu_i^+ (0)^2
    \end{pmatrix}, \quad i = 1, 2.
\end{equation}
For $0 < \nu < 1/4$, the eigenvalues $\mu_1^+ (0)$
and $\mu_2^+ (0)$ are real and distinct, 
satisfying $\mu_1^+ (0) < \mu_2^+ (0) < 0$, and we can readily 
construct individual solutions 
\begin{equation*}
    y_i^+ (x; 0) = e^{\mu_i^+ (0) x} (v_i^+ (0) + E_i^+ (x; 0)),
    \quad i = 1, 2,
\end{equation*}
where $E_i^+ (x; 0)$ decays at exponential rate in $x$ as $x$ tends
to $+\infty$. Since $\eta^- (x; 0)$ solves (\ref{kdvb-system})
and decays at exponential rate as $x$ tends to $+\infty$, and 
additionally since $\mu_3^+ (0) = 0$, there must exist constants
$C_1$ and $C_2$ so that 
\begin{equation} \label{kdvb-upp1}
    \eta^- (x; 0) = k_- \begin{pmatrix}
    \bar{u}' (x) \\ \bar{u}'' (x) \\ \nu \bar{u}'' (x)
    \end{pmatrix}
    = C_1 e^{\mu_1^+ (0) x} (v_1^+ (0) + E_1^+ (x; 0))
    + C_2 e^{\mu_2^+ (0) x} (v_2^+ (0) + E_2^+ (x; 0)).
\end{equation}
In addition, upon integration of the first component on $(x, \infty)$, we obtain 
the relation 
\begin{equation} \label{kdvb-ratio1}
    k_- (\bar{u} (x) + 1)
    = C_1 e^{\mu_1^+ (0) x} (\frac{1}{\mu_1^+ (0)} + \mathbf{O} (e^{-\alpha |x|}))
    + C_2 e^{\mu_2^+ (0) x} (\frac{1}{\mu_2^+ (0)} + \mathbf{O} (e^{-\alpha |x|})),
\end{equation}
for some fixed $\alpha > 0$. Combining these observations, we can 
conclude that for $C_2 = 0$ we have the limits
\begin{equation} \label{kdvb-ratio2}
    \lim_{x \to + \infty} \frac{\bar{u}''(x)}{\bar{u}' (x)}
    = \mu_1^+ (0); \quad
     \lim_{x \to + \infty} \frac{\bar{u}'(x)}{\bar{u} (x) + 1}
    = \mu_1^+ (0),
\end{equation}
while for $C_2 \ne 0$ the same relations hold with $\mu_1^+ (0)$
replaced by $\mu_2^+ (0)$. 

We are now in a position to compute the hyperplane index on the right shelf
for $0 < \nu < 1/4$. We have already seen that there are no crossing
points $x_* \in \mathbb{R}$, so we only need to understand the nature of 
the asymptotic crossing point as $x$ tends to $+ \infty$. For this, 
we work with the scaled map
\begin{equation*}
    \tilde{\psi}_1^+ (x; 0)
    = \frac{\tilde{\omega}_1^+ (x; 0)}{|\eta^- (x; 0)| |\tilde{\mathcal{V}}^+ (0)|},
\end{equation*}
which according to (\ref{psi1-tilde-plus-minus-defined}) has the 
well-defined asymptotic limit
\begin{equation*}
    \tilde{\psi}_1^{+,-} (0) 
    = \lim_{x \to -\infty} \tilde{\psi}_1^+ (x; 0)
    = \frac{v^- (0) \wedge \tilde{\mathcal{V}}^+ (0)}{|v^- (0)| |\tilde{\mathcal{V}}^+ (0)|}
    > 0,
\end{equation*}
where the final inequality follows from our verification of (\ref{kvdb-sought}). 
From (\ref{kdvb-upp1}), we see that 
\begin{equation*}
    \lim_{x \to + \infty} \frac{\eta^- (x; \lambda)}{|\eta^- (x; \lambda)|}
    = \begin{cases}
        v_2^+ (0)/|v_2^+ (0)| & C_2 \ne 0 \\
        v_1^+ (0)/|v_1^+ (0)| & C_2 = 0. 
    \end{cases}
\end{equation*}
In either case, 
\begin{equation*}
v_i^+ (0) \wedge \tilde{\mathcal{V}}^+ (0)
= \tilde{\kappa} (0) \begin{pmatrix}
    1 \\ \mu_i^+ (0) \\ \nu \mu_i^+ (0)^2
    \end{pmatrix}
\wedge
\begin{pmatrix}
    1 \\ -1 \\1
\end{pmatrix}
= \tilde{\kappa} (0) (\mu \mu_i^+ (0)^2 + \mu_i^+ (0) + 1) = 0,
\end{equation*}
where the final equality holds because $\mu_1^+ (0)$
and $\mu_2^+ (0)$ are the non-zero roots of 
$h_+ (\mu;0)$. In this way, we see that a crossing 
must be associated with the asymptotic limit as 
$x$ tends to $+ \infty$. In order to understand the 
sign associated with this crossing, we consider directly
the signs of $\tilde{\psi}_1^+ (x; 0)$ and 
$\tilde{\psi}_2^+ (x; 0)$ as $x$ tends to $+ \infty$. 

First, since $k_- < 0$ and $\bar{u}' (x) < 0$
for all $x \in \mathbb{R}$, we see from (\ref{kdvb-tilde-omega1}) that in 
this case $\tilde{\psi}_1^+ (x;0) > 0$ for all $x \in \mathbb{R}$. In 
addition, we see from (\ref{kdvb-ratio1}) and (\ref{kdvb-ratio2}) that 
$\bar{u}'' (x) > 0$ for $x$ sufficiently large, so from 
(\ref{kdvb-tilde-omega20}) we see that $\tilde{\psi}_2^+ (x; 0) > 0$ 
for all $x$ sufficiently large. With $\tilde{\psi}_1^+ (x;0)$ and 
$\tilde{\psi}_2^+ (x;0)$ both positive, $p^+(x;0)$ lies in the third 
quadrant and so approaches $(-1,0)$ in the clockwise direction 
as $x$ tends to $+\infty$. 
According to our convention, the hyperplane index does not increment 
in this case, and we can conclude that for $0 < \nu < 1/4$, we have 
\begin{equation*}
\ind (\mathpzc{g} (\cdot; 0), \mathpzc{h}^+ (0); [- \infty, + \infty])    
= 0.
\end{equation*}
According to Theorem \ref{main-theorem}, we can conclude that {\it if} we 
have invariance along the top shelf in this case then the  
count $\mathcal{N}_{\#} ((-\infty, 0])$ of non-positive eigenvalues
of (\ref{kdvb-evp}) satisfies the inequality 
\begin{equation*}
\mathcal{N}_{\#} ((-\infty,0]) \ge |\mathfrak{m}|.     
\end{equation*}

As in Section \ref{gkdv-section}, we can use information about the 
Evans function to obtain an estimate for $\mathcal{N}_{\#} ((-\infty,0))$
rather than $\mathcal{N}_{\#} ((-\infty,0])$. We've seen that 
for (\ref{kdvb-evp}) we have $D(0) = 0$ and $D' (0) < 0$, from which we can conclude that 
for $x$ sufficiently large we must have $\tilde{\psi}_1^+ (x; \lambda) > 0$
for all $\lambda < 0$ sufficiently close to $0$. In addition, we've seen that for the 
current application $\tilde{\psi}_2^+ (x; 0) > 0$ for $x \gg 0$.
This pairing $\tilde{\psi}_1^+ (x; \lambda) > 0$ and 
$\tilde{\psi}_2^+ (x; 0) > 0$ places $p^+ (x; \lambda)$ in the third 
quadrant,
signifying that as $\lambda$ increases to $0$ $p^+ (x; \lambda)$
rotates into $(-1,0)$ in the clockwise direction, and the 
hyperplane index is not incremented. This allows us to refine 
(\ref{gkdv-preliminary-conclusion}) to the statement 
\begin{equation*} 
\mathcal{N}_{\#} ((-\infty, 0))
\ge |\mathfrak{m}|. 
\end{equation*}
Since we may have $\mathfrak{m} = 0$, this relation is consistent with 
stability (though does not imply stability). 

We now turn to the interesting case $\eta > 1/4$, for which oscillations 
in the wave $\bar{u} (x)$ suggest the possible onset of instability. 
Our primary interest in this example is determining why, from the current
geometric point of view, such 
instability doesn't occur. As usual, we begin by identifying 
all conjugate points on $\mathbb{R}$. We see from (\ref{kdvb-tilde-omega1}) 
that these are values $x_* \in \mathbb{R}$
for which either $\bar{u}'(x_*) = 0$ or $\bar{u} (x_*) = -1$. The first
such point, which we will denote $x_1$, occurs when $\bar{u} (x)$
first crosses the horizontal line at $-1$. If we order further crossings
as the sequence $x_1 < x_2 < x_3 < ...$, we see from Figure \ref{gkdv-figure1}
that $\bar{u}' (x_2) = 0$, $\bar{u} (x_3) = -1$, $\bar{u}' (x_4) = 0$,
and so on, with an infinite number of crossings in total. 

\begin{remark} \label{secondPSW1993-remark}
    For $\nu > 1/4$, the wave $\bar{u} (x)$ corresponds with a connection 
    in the $(\bar{u}, \bar{u}')$ phase plane from a saddle point at 
    $(1, 0)$ to a stable spiral at $(-1, 0)$, ensuring the qualitative
    properties described here. See Theorem 1.1 in \cite{PSW1993} for 
    details. 
\end{remark}

In order to assign directions to these crossing points, we express
(\ref{kdvb-wave-equation-integrated}) as 
\begin{equation} \label{{kdvb-wave-equation-integrated}}
    \nu \bar{u}'' (x) + \bar{u}'(x) = \frac{1}{2} (\bar{u} (x)^2 - 1).
\end{equation}
For $x_1$, we see that 
\begin{equation*}
 \nu \bar{u}'' (x_1) = - \bar{u}'(x_1) < 0,     
\end{equation*}
and subsequently the sign alternates for $x_3$, $x_5$ etc. (i.e., 
$ \nu \bar{u}'' (x_3) = - \bar{u}'(x_3) > 0$, 
$ \nu \bar{u}'' (x_5) = - \bar{u}'(x_5) < 0$ and so on). 
Likewise, for $x_2$ we see that 
\begin{equation*}
 \nu \bar{u}'' (x_2) = \frac{1}{2} (\bar{u} (x_2)^2 - 1) > 0,      
\end{equation*}
with the sign again alternating for $x_2$, $x_4$, etc. Combining 
these observations, we conclude that 
\begin{equation*}
    \sgn \tilde{\omega}_2^+ (x_i;0)
    = \sgn \bar{u} '' (x)
    = \begin{cases}
    +1 & i = 2, 3, 6, 7, \dots \\
    -1 & i = 1, 4, 5, 8, 9, \dots.
    \end{cases}
\end{equation*}

Likewise, 
\begin{equation*}
    \tilde{\omega}_1^{+\,\prime} (x; 0)
    = \frac{k_-}{\sqrt{3}} \Big(\bar{u}''(x) (1+\bar{u} (x)) + \bar{u}'(x)^2 \Big),
\end{equation*}
and so
\begin{equation*}
 \tilde{\omega}_1^{+\,\prime} (x_i; 0)
 = \frac{k_-}{\sqrt{3}} \bar{u}'(x_i)^2 < 0,
 \quad i = 1, 3, 5, \dots,
\end{equation*}
with also 
\begin{equation*}
 \tilde{\omega}_1^{+\,\prime} (x_i; 0)
 = \frac{k_-}{\sqrt{3}} \bar{u}''(x_i) (1+\bar{u} (x_i)) > 0,
 \quad i = 2, 4, 6, \dots,
\end{equation*}
Putting these observations together, we see that 
\begin{equation*}
    \sgn \frac{\tilde{\omega}_1^{+\,\prime} (x_i; 0)}{\tilde{\omega}_2^+ (x_i;0)}
    = \begin{cases}
    +1 & i = 1, 2, 5, 6, \dots \\
    -1 & i = 3, 4, 7, 8, \dots.
    \end{cases}
\end{equation*}
We can conclude that as $x$ increases the value of 
\begin{equation*}
\ind (\mathpzc{g} (\cdot; 0), \mathpzc{h}^+ (0); (-\infty, x])    
\end{equation*}
cycles among the values $\{0, 1, 2\}$, starting with $0$
(for $x < x_1$). 
The full hyperplane index doesn't exist, but the cancellation 
in this calculation suggests that every spectral curve
that enters the Maslov box through the right shelf 
also exits through the right shelf. These considerations
suggest that even for $\nu > 1/4$ the wave $\bar{u} (x)$
might be spectrally stable, and indeed the additional numerical
calculations carried out below bear this out. 

As with the application considered in Section \ref{gkdv-section}, we 
finish this section with a numerical evaluation of the spectral 
curves associated with (\ref{kdvb-evp}). For this calculation,
we will work with the functions 
\begin{equation*}
\begin{aligned}
    \tilde{\omega}_1^+ (x; \lambda)
    &= \eta^- (x; \lambda) \wedge \tilde{\mathcal{V}}^+ (\lambda) \\
    &=\tilde{\kappa} (\lambda)
    \Big{\{} (\nu \mu_3^+ (\lambda)^2 + \mu_3^+ (\lambda) + 1) \eta_1^- (x; \lambda)
    + (\nu \mu_3^+ (\lambda)+1) \eta_2^- (x; \lambda) + \eta_3^- (x; \lambda)\Big{\}}
\end{aligned}
\end{equation*}
and 
\begin{equation*}
\begin{aligned}
    \tilde{\omega}_2^+ (x; \lambda)
    &= \eta^- (x; \lambda) \wedge \tilde{\mathcal{V}}_M^+ (\lambda) 
    = - \Big{\{} (\mu_3^+ (\lambda)^2 - (1-1/\nu) \mu_3^+ (\lambda)) \eta_1^- (x; \lambda) \\
    &\quad \quad + (\mu_3^+ (\lambda) + 1/\nu) \eta_2^- (x; \lambda) 
    - \mu_3^+ (\lambda) \eta_3^- (x; \lambda)\Big{\}},
\end{aligned}
\end{equation*}
where $\tilde{\mathcal{V}}^+ (\lambda)$ is as above and 
\begin{equation*}
    \tilde{\mathcal{V}}_M^+ (\lambda)
    =
    \begin{pmatrix}
    \mu_3^+ (\lambda) \\
     \mu_3^+ (\lambda) + 1/\nu \\
      - \mu_3^+ (\lambda)^2 + (1 - 1/\nu) \mu_3^+ (\lambda)
    \end{pmatrix},
\end{equation*}
is an eigenvector associated with the eigenvalue $- \mu_3^+ (\lambda)$
for the matrix $\tilde{\mathcal{A}}_+ (\lambda)$ defined via (\ref{tildeA}) 
from 
\begin{equation*}
    \mathcal{A}_+ (\lambda)
    = M A_+(\lambda) M^{-1}
    = \begin{pmatrix}
    -1 & 1 & -1 \\
    -\lambda & 0 & 0 \\
    1 - \lambda & -1 & 1 - 1/\nu
    \end{pmatrix}.
\end{equation*}

Using $\tilde{\omega}_1^+ (x; \lambda)$ and $\tilde{\omega}_2^+ (x; \lambda)$, 
we can now generate the spectral curves on any truncated Maslov box by numerically 
computing $\eta^- (x; \lambda)$. As expected, if this is done for any $\nu \in (0, 1/4)$, 
no spectral curves enter the Maslov box, and so there are no spectral curves to 
depict. In addition, as $x$ increases to some sufficiently large value $c$, 
the tracking point $p^+ (x; 0)$ rotates toward $(-1,0)$ in the clockwise 
direction (without reaching it), and as $\lambda$ decreases from $0$ $p^+ (c; \lambda)$
rotates away from $(-1,0)$ in the counterclockwise direction. In this way, we see that 
for $c$ sufficiently large, there is no contribution to the Maslov box for any $\nu \in (0,1/4)$.
As a specific case, calculations were carried out for $\nu = 1/8$, and it was confirmed
that for the Maslov box $[-5,0] \times [-20, 20]$ there are no crossings on the boundary,
so $\mathfrak{m} = 0$. This includes no crossings on the top shelf, indicating no 
unstable eigenvalues. 

The cases $\nu > 1/4$ are more interesting. We have already seen in our analytic calculations
that for all $\nu > 1/4$, there are in fact an infinite number of crossings on the right 
shelf at a sequence of values $\{x_i\}_{i=1}^{\infty}$ such that 
$\lim_{i \to \infty} x_i = \infty$. Correspondingly, we expect to find an infinite 
number of spectral curves entering and exiting the Maslov box along the right 
shelf. Since our numerical calculations will be truncated, they won't confirm this 
expectation, but they provide evidence that the spectral curves are precisely as 
expected over the window of investigation. We will carry out calculations for two 
specific cases, $\nu = 2$ and $\nu = 5$. These seem to give an indication of how 
the picture varies as $\nu$ decreases below 2 and increases above 5. 

First, for $\nu = 2$, the spectral curves entering through the right shelf 
are seen numerically to be contained in the vertical strip associated with 
the $\lambda$ interval $[-.02, 0]$. In Figure \ref{kdvb_box2-figure1}, 
spectral curves in the Maslov box $[-.02, 0] \times [-22, 22]$ are depicted.
In order to see that the crossings continue, we also provide a second figure
depicting the Maslov box $[-.0002, 0] \times [-22, 22]$ in which the upper spectral curves 
from Figure \ref{kdvb_box2-figure1} are more fully resolved, and two additional
spectral curve becomes apparent. See Figure \ref{kdvb_box2-figure2}.  

\begin{figure}[ht] 
\begin{center}\includegraphics[%
  width=11cm,
  height=8cm]{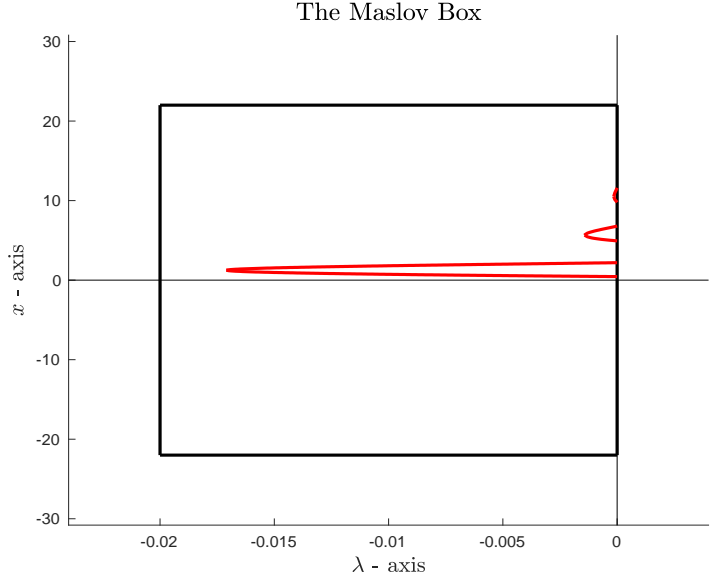}\end{center}
\caption{Maslov Box associated with (\ref{kdvb-equation}) for 
$\nu = 2$. \label{kdvb_box2-figure1}}
\end{figure}

\begin{figure}[ht] 
\begin{center}\includegraphics[%
  width=11cm,
  height=8cm]{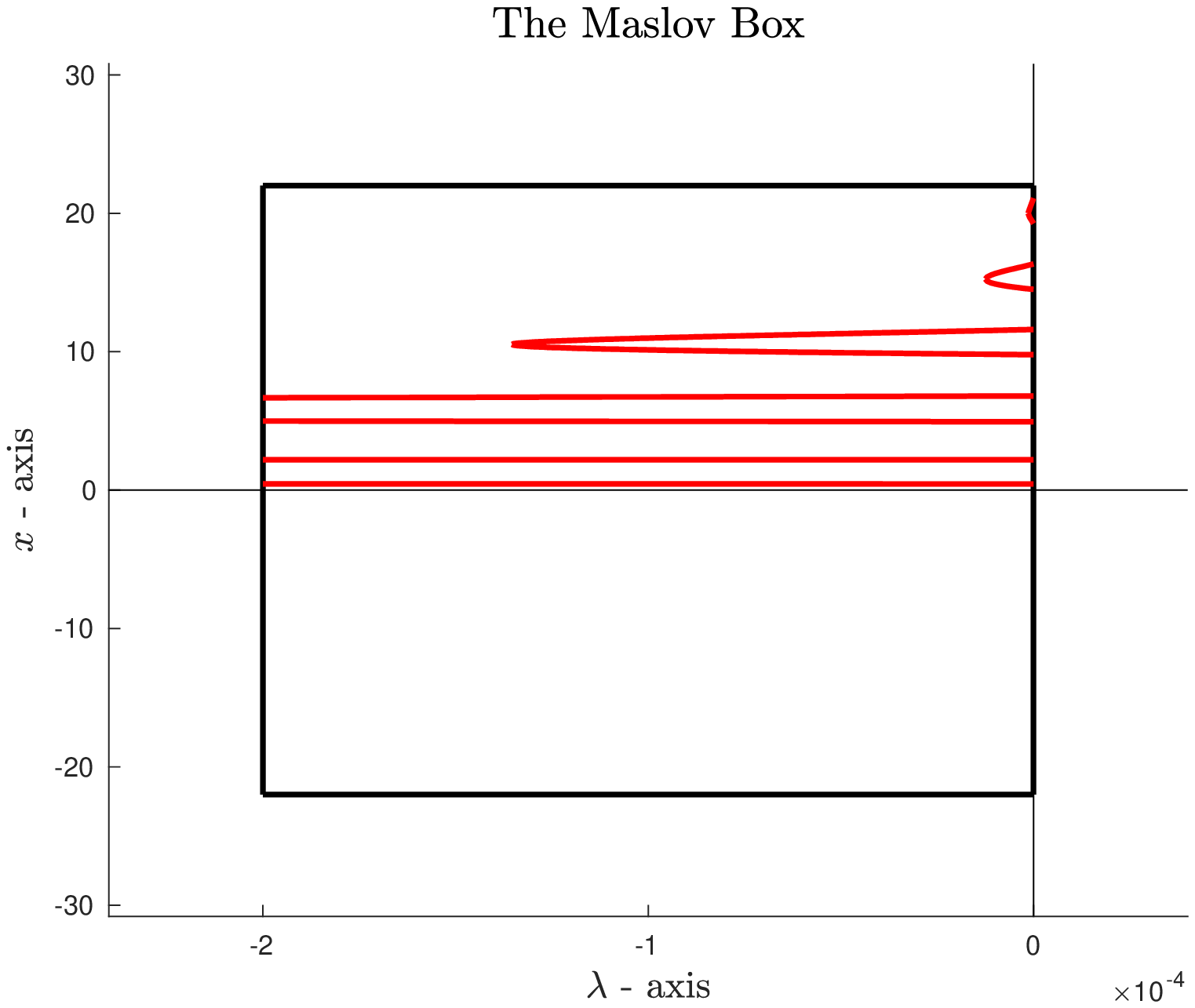}\end{center}
\caption{Maslov Box associated with (\ref{kdvb-equation}) for 
$\nu = 2$, including additional spectral curves. \label{kdvb_box2-figure2}}
\end{figure}

As a point of comparison, we also numerically generate spectral curves 
for the case $\nu = 5$. As $\nu$ increases, the oscillations of 
the stationary solution $\bar{u} (x; \nu)$ spread out, so we 
expect the spectral curves to be farther apart. This is indeed 
the case, as show in Figure \ref{kdvb_nu5-figure}. As $\nu$
continues to increase, we expect the spectral curves to be spaced
farther apart and to go farther into the Maslov box. 

\begin{figure}[ht] 
\begin{center}\includegraphics[%
  width=11cm,
  height=8cm]{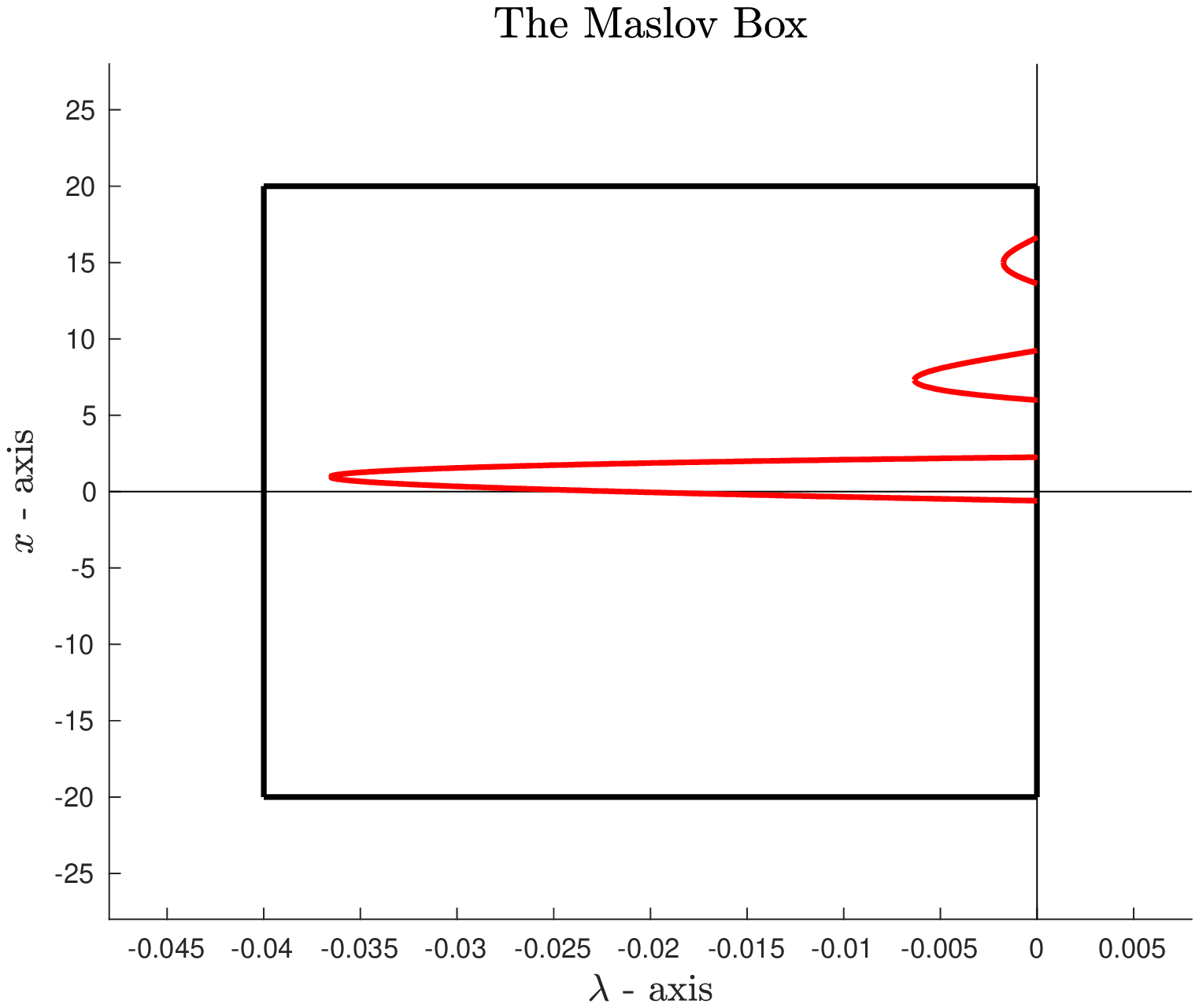}\end{center}
\caption{Maslov Box associated with (\ref{kdvb-equation}) for 
$\nu = 5$. \label{kdvb_nu5-figure}}
\end{figure}

\appendix

\section{Appendix}

In this appendix, we include a verification of the claim in 
Section \ref{kdvb-section} that in our analysis of the 
KdV-Burgers equation, there are no intersections along 
the left shelf. First, in order to have an intersection 
along the left shelf, there must be some value $s \in \mathbb{R}$
so that 
\begin{equation*}
    \eta^- (s; \lambda_1) \wedge \tilde{\mathcal{V}}^+ (\lambda_1)
    = 0.
\end{equation*}
I.e., $\lambda_1$ must be an eigenvalue for the half-line 
problem 
\begin{equation} \label{kdvb-halfline}
\begin{aligned}
    - \nu \phi''' &- \phi'' + (\bar{u} (x) \phi)' 
    = \lambda \phi, \quad x \in (- \infty, s) \\
    - \frac{\lambda}{\mu_3^+ (\lambda)} \phi (s)
    &- (\frac{\lambda}{\mu_3^+ (\lambda)^2} + \frac{1}{\mu_3^+ (\lambda)}) \phi' (s)
    + \nu \phi''(s) = 0,
\end{aligned}
\end{equation}
where in formulating this boundary condition we have 
used $\tilde{\mathcal{V}}^+ (\lambda)$
from (\ref{kdvb-tildev-alternative}) to write 
\begin{equation*}
\begin{aligned}
    \begin{pmatrix}
        \phi (s) \\ \phi' (s) \\ \nu \phi''(s)
    \end{pmatrix}
    &\wedge
    \begin{pmatrix}
        1 \\ - \nu \mu_3^+ (\lambda) - 1 \\ \nu \mu_3^+ (\lambda)^2 + \mu_3^+ (\lambda) +1
    \end{pmatrix} \\
    &= (\nu \mu_3^+ (\lambda)^2 + \mu_3^+ (\lambda) +1) \phi (s)
    + (\nu \mu_3^+ (\lambda) + 1) \phi' (s) + \nu \phi''(s) = 0,
\end{aligned}
\end{equation*}
then rearranged terms using with the eigenvalue relation 
\begin{equation*}
    \nu \mu_3^+ (\lambda)^3 + \mu_3^+ (\lambda)^2
    + \mu_3^+ (\lambda) = - \lambda
\end{equation*}
(i.e., $h_+ (\mu_3^+; \lambda) = 0$) with $h_+$ as in (\ref{kdvb-h}).

We proceed via a standard energy argument, which begins with the assumption
that for some $\lambda \le 0$ there exists a solution $\phi (x)$ to 
(\ref{kdvb-halfline}) so that $\phi (x) \to 0$ as $x \to - \infty$, 
necessarily at exponential rate under our assumptions. 
We multiply (\ref{kdvb-halfline}) by $\phi$
and integrate over $(-\infty,s)$ to obtain the relation
\begin{equation} \label{kdvb-halfline-integrated}
    - \nu \int_{-\infty}^s \phi''' \phi dx
    - \int_{-\infty}^s \phi'' \phi dx
    + \int_{-\infty}^s (\bar{u} (x) \phi)' \phi dx
    = \lambda \int_{-\infty}^s \phi^2 dx.
\end{equation}
For the first summand on the left-hand side of (\ref{kdvb-halfline-integrated}),
we can integrate by parts to write 
\begin{equation} \label{kdvb-summand1}
    \begin{aligned}
    - \nu \int_{-\infty}^s \phi''' \phi dx
    &= - \nu \phi''(s) \phi (s) + \nu \int_{-\infty}^s \phi'' \phi' dx \\
    &= - \nu \phi''(s) \phi (s) + \frac{\nu}{2} \phi' (s)^2.
    \end{aligned}
\end{equation}
Likewise, for the second summand on the left-hand side of (\ref{kdvb-halfline-integrated}),
we can integrate by parts to write 
\begin{equation} \label{kdvb-summand2}
    - \int_{-\infty}^s \phi'' \phi dx 
    = - \phi' (s) \phi (s) + \int_{- \infty}^s {\phi'}^2 dx,
\end{equation}
and for the third we similarly obtain the relation 
\begin{equation} \label{kdvb-summand3}
\int_{-\infty}^s (\bar{u} (x) \phi)' \phi dx
= \bar{u} (s) \phi (s)^2 - \int_{-\infty}^s \bar{u} (x) \phi \phi' dx.
\end{equation}

Next, we obtain lower bounds on the right-hand sides of the 
expressions (\ref{kdvb-summand1}), (\ref{kdvb-summand2}),
and (\ref{kdvb-summand3}). Starting with (\ref{kdvb-summand3}), 
we'll set 
\begin{equation*}
    C := \sup_{x \in \mathbb{R}} |\bar{u} (x)|,
\end{equation*}
which is bounded due to continuity of $\bar{u} (x)$ and the 
endstate conditions. In addition, for a value $\epsilon > 0$ to 
be chosen sufficiently small below, we note the standard inequality
\begin{equation} \label{young-inequality}
    |\phi| |\phi'| \le \frac{1}{2 \epsilon} |\phi|^2 
    + \frac{\epsilon}{2} |\phi'|^2.
\end{equation}
This allows us to express the lower bound 
\begin{equation*}
- \int_{-\infty}^s \bar{u} (x) \phi \phi' dx
\ge - \frac{C}{2 \epsilon} \int_{-\infty}^s \phi^2 dx
- \frac{C \epsilon}{2} \int_{-\infty}^s {\phi'}^2 dx. 
\end{equation*}
Next, in order to estimate $\bar{u} (s) \phi (s)^2$, we 
observe that $\phi (s)^2$ can be expressed as 
\begin{equation*}
    \phi (s)^2 = \int_{-\infty}^s \frac{d}{dx} \phi^2 dx
    =  \int_{-\infty}^s 2 \phi \phi' dx,
\end{equation*}
which we can combine with (\ref{young-inequality}) to see 
that 
\begin{equation} \label{inequality-on-phi-squared}
\phi (s)^2
\le \frac{1}{\epsilon} \int_{-\infty}^s \phi^2 dx
+ \epsilon \int_{-\infty}^s {\phi'}^2 dx.
\end{equation}
Upon combining these observations, we obtain the 
inequality 
\begin{equation*}
\int_{-\infty}^s (\bar{u} (x) \phi)' \phi dx
\ge - \frac{3 C}{2\epsilon} \int_{-\infty}^s \phi^2 dx
- \frac{3 C \epsilon}{2} \int_{-\infty}^s {\phi'}^2 dx. 
\end{equation*}

Next, for the right-hand side of (\ref{kdvb-summand2}), we 
can use (\ref{young-inequality}) once again (with a new 
constant $\delta$ in place of $\epsilon$) to write 
\begin{equation*}
     - \phi' (s) \phi (s) + \int_{-\infty}^s {\phi'}^2 dx 
     \ge - \frac{1}{2 \delta} |\phi|^2 
    - \frac{\delta}{2} |\phi'|^2 
     + \int_{- \infty}^s {\phi'}^2 dx,
\end{equation*}
and subsequently we can use 
(\ref{inequality-on-phi-squared}) to obtain the 
inequality 
\begin{equation*}
     - \int_{-\infty}^s \phi'' \phi dx 
     \ge - \frac{1}{2 \epsilon \delta} \int_{- \infty}^s \phi^2 dx
     - \frac{\epsilon}{2\delta} \int_{- \infty}^s {\phi'}^2 dx
    - \frac{\delta}{2} |\phi'|^2 
     + \int_{- \infty}^s {\phi'}^2 dx.
\end{equation*}

Last, for (\ref{kdvb-summand1}), we use the boundary condition in 
(\ref{kdvb-halfline}) to write 
\begin{equation*}
    \nu \phi'' (s) 
    =  \frac{\lambda}{\mu_3^+ (\lambda)} \phi (s)
    + (\frac{\lambda}{\mu_3^+ (\lambda)^2} + \frac{1}{\mu_3^+ (\lambda)}) \phi' (s),
\end{equation*}
from which we see that 
\begin{equation*}
- \nu \phi''(s) \phi (s) =    
- \frac{\lambda}{\mu_3^+ (\lambda)} \phi (s)^2
    - (\frac{\lambda}{\mu_3^+ (\lambda)^2} + \frac{1}{\mu_3^+ (\lambda)}) \phi' (s) \phi (s).
\end{equation*}
The sign of the first summand on the right-hand side is beneficial, but the 
second summand on the right-hand side will take some work to control. First, 
similarly as with (\ref{young-inequality}), given any $\eta > 0$, we have the 
inequality 
\begin{equation*}
|(\frac{\lambda}{\mu_3^+ (\lambda)^2} + \frac{1}{\mu_3^+ (\lambda)}) \phi' (s) \phi (s)|
\le \frac{1}{2 \eta} |(\frac{\lambda}{\mu_3^+ (\lambda)^2} + \frac{1}{\mu_3^+ (\lambda)})|^2 \phi (s)^2
+ \frac{\eta}{2} \phi'(s)^2.
\end{equation*}

Recalling that $\mu_3^+ (\lambda)$ solves the polynomial 
equation 
\begin{equation*}
    \nu \mu^3 + \mu^2 + \mu + \lambda = 0,
\end{equation*}
we see that the behavior of $\mu_3^+ (\lambda)$
for $\lambda \ll 0$ can be characterized by the limit 
\begin{equation*}
    \lim_{\lambda \to - \infty} \frac{\nu \mu_3^+ (\lambda)^3}{\lambda} = - 1. 
\end{equation*}
In particular, 
\begin{equation*}
\lim_{\lambda \to - \infty} \frac{|\lambda|^2/|\mu^+_3 (\lambda)|^4}{|\lambda|/|\mu^+_3 (\lambda)|}
= \lim_{\lambda \to - \infty} \frac{|\lambda|}{|\mu^+_3 (\lambda)|^3} = \nu.
\end{equation*}
Given any $\kappa > 0$, we can find $\Lambda \gg 0$ sufficiently large so that
\begin{equation*}
    \frac{|\lambda|^2}{|\mu^+_3 (\lambda)|^4}
    \le (\nu + \kappa) \frac{|\lambda|}{|\mu^+_3 (\lambda)|},
    \quad \forall\, \lambda < - \Lambda.
\end{equation*}
In addition, since 
$\lambda/\mu_3^+ (\lambda)^2 < - 1/\mu_3^+ (\lambda) < 0$,
we have the simple inequality 
 \begin{equation*}
|\frac{\lambda}{\mu_3^+ (\lambda)^2} + \frac{1}{\mu_3^+ (\lambda)}|
\le |\frac{\lambda}{\mu_3^+ (\lambda)^2}|,
\end{equation*}
and this allows us to express the estimate
\begin{equation*}
\frac{1}{2 \eta} |\frac{\lambda}{\mu_3^+ (\lambda)^2} + \frac{1}{\mu_3^+ (\lambda)}|^2
\le \frac{1}{2 \eta}  |\frac{\lambda}{\mu_3^+ (\lambda)^2}|^2
\le \frac{1}{2 \eta} (\nu + \kappa) \frac{|\lambda|}{|\mu^+_3 (\lambda)|}.
\end{equation*}
We now see that 
\begin{equation*}
- \nu \phi''(s) \phi (s)
\ge - \frac{\lambda}{\mu_3^+ (\lambda)} \phi (s)^2
- \frac{1}{2 \eta} (\nu + \kappa) \frac{|\lambda|}{|\mu^+_3 (\lambda)|} \phi (s)^2
- \frac{\eta}{2} \phi'(s)^2. 
\end{equation*}
At this point, we choose $\eta = \frac{3}{4} \nu$ and $\kappa = \frac{1}{4} \nu$, for 
which the inequality becomes 
\begin{equation*}
- \nu \phi''(s) \phi (s)
\ge - \frac{\lambda}{\mu_3^+ (\lambda)} \phi (s)^2
- \frac{5}{6} \frac{|\lambda|}{|\mu^+_3 (\lambda)|} \phi (s)^2
- \frac{3}{8} \nu \phi'(s)^2 \ge - \frac{3}{8} \nu \phi'(s)^2.
\end{equation*}

If we now put all of these inequalities together, we obtain the 
lower bound 
\begin{equation*}
\begin{aligned}
\lambda \int_{-\infty}^s \phi^2 dx
&\ge - \frac{3}{8} \nu \phi'(s)^2 + \frac{\nu}{2} \phi'(s)^2
- \frac{1}{2 \epsilon \delta} \int_{- \infty}^s \phi^2 dx
     - \frac{\epsilon}{2\delta} \int_{- \infty}^s {\phi'}^2 dx \\
    &- \frac{\delta}{2} \phi' (s)^2 
     + \int_{- \infty}^s {\phi'}^2 dx 
     - \frac{3 C}{2\epsilon} \int_{-\infty}^s \phi^2 dx
- \frac{3 C \epsilon}{2} \int_{-\infty}^s {\phi'}^2 dx. 
\end{aligned}
\end{equation*}
Finally, we will complete the calculation by making judicious
choices for $\epsilon$ and $\delta$ based on the values
of $\nu$ and $C$. For this we'll require 
\begin{equation*}
\begin{aligned}
    \frac{3}{8} \nu + \frac{\epsilon}{2} &\le \frac{\nu}{2} \\
    \frac{\epsilon}{2\delta} + \frac{3 C \epsilon}{2} &\le 1. 
\end{aligned}
\end{equation*}
To be concrete, we take the specific values 
\begin{equation*}
    \epsilon = \min \{\frac{1}{4}\nu, \frac{1}{3C}\}
    \quad \textrm{and} \quad 
    \delta = \epsilon.
\end{equation*}
It follows that 
\begin{equation*}
    \lambda \int_{-\infty}^s \phi^2 dx
    \ge - \Big(\frac{1}{2 \epsilon \delta} + \frac{3 C}{2\epsilon} \Big) 
    \int_{- \infty}^s \phi^2 dx,
\end{equation*}
from which we see that any eigenvalue $\lambda \in \mathbb{R}$ of 
(\ref{kdvb-halfline}) must satisfy the inequality 
\begin{equation*}
    \lambda \ge - \Big(\frac{1}{2 \epsilon \delta} + \frac{3 C}{2\epsilon} \Big),
\end{equation*}
where $\epsilon$ and $\delta$ are the fixed constants chosen above. By choosing
$\lambda_1$ below this threshold, we can ensure that there are no crossings 
on the left shelf.

{\em author email: phoward@tamu.edu}


\begin{thebibliography}{99}

\bibitem{AGJ1990} J. Alexander, R. Gardner, and C. K. R. T. Jones, 
{\it A topological invariant arising in the stability analysis of 
traveling waves}, J. Reine Angew. Math. {\bf 410} (1990) 167-212.

\bibitem{BCCJM2022} T. J. Baird, P. Cornwell, G. Cox, C. Jones, and R. Marangell,
{\it Generalized Maslov indices for non-Hamiltonian systems},
SIAM J. Math. Anal. {\bf 54} (2022) 1623-1668. 

\bibitem{BCJLMS2018} M. Beck, G. Cox, C. K. R. T. Jones, 
Y. Latushkin, K. McQuighan, and A. Sukhtayev, {\it Instability of pulses 
in gradient reaction-diffusion systems: a symplectic approach},
Philos. Trans. Roy. Soc. A {\bf 376} (2018), no. 2117, 20170187, 
20 pp. 

\bibitem{BJ1995} A.\ Bose and C.\ K.\ R.\ T.\ Jones,
{\em Stability of the in-phase traveling wave solution in a pair 
of coupled nerve fibers,}
Indiana U. Math. J. {\bf 44} (1995) 189 -- 220.

\bibitem{BM2013} M.\ Beck and S.\ Malham, 
{\em Computing the Maslov index for large systems}, 
Proc. Amer. Math. Soc. {\bf 143} (2015), no. 5, 2159–2173.

\bibitem{BM2022} M. Beck and J. Jaquette, {\it Validated spectral stability
via conjugate points}, SIAM J. Appl. Dyn. Sys. {\bf 21} (2022) 366--404.

\bibitem{CB2015} F.\ Chardard and T.\ J.\ Bridges,
{\em Transversality of homoclinic orbits, the Maslov index, and 
the symplectic Evans function}, Nonlinearity {\bf 28} (2015)
77--102. 

\bibitem{CDB2009} F.\ Chardard, F.\ Dias and T.\ J.\ Bridges, 
{\em Computing the Maslov index of solitary waves, Part 1: 
Hamiltonian systems on a four-dimensional phase space}, 
Phys. D {\bf 238} (2009) 1841 -- 1867.

\bibitem{CDB2011} F.\ Chardard, F.\ Dias and T.\ J.\ Bridges, 
{\em Computing the Maslov index of solitary waves, Part 2: 
Phase space with dimension greater than four}, 
Phys. D {\bf 240} (2011) 1334 -- 1344.

\bibitem{CH2007} C--N. Chen and Xijun Hu, {\em Maslov index for homoclinic 
orbits of Hamiltonian systems}, Ann. Inst. H. Poincar\'e Anal. Nonlin\'eaire
{\bf 24} (2007) 589--603.

\bibitem{CH2014} C--N. Chen and Xijun Hu, {\em Stability analysis for standing
pulse solutions to FitzHugh--Nagumo equations}, Calculus of Variations
and Partial Differential Equations {\bf 49} (2014) 827--845.

\bibitem{Chardard2009} F.\ Chardard, 
{\em Stability of Solitary Waves}, Doctoral thesis, Centre de Mathematiques et de 
Leurs Applications, 2009. Advisors: T.\ J.\ Bridges and F. Dias. 

\bibitem{CJ2018} P. Cornwell and C. K. R. T. Jones,
{\em On the existence and stability of fast traveling waves in a doubly diffusive FitzHugh--Nagumo system},
SIAM Journal on Applied Dynamical Systems {\bf 17} (2018) 754-787

\bibitem{CJ2020} P. Cornwell and C. K. R. T. Jones,
{\em A stability index for traveling waves in activator-inhibitor systems}
Proceedings of the Royal Society of Edinburgh: Section A Mathematics 
{\bf 150} (2020) 517--548.

\bibitem{Evans1972a} J. W. Evans, {\it Nerve axon equations I: Linear Approximations},
Indiana Univ. Math. J. {\bf 21} (1972) 877-955.

\bibitem{Evans1972b} J. W. Evans, {\it Nerve axon equations II: Stability at Rest},
Indiana Univ. Math. J. {\bf 22} (1972) 75-90.

\bibitem{Evans1972c} J. W. Evans, {\it Nerve axon equations III: Stability of the Nerve Impulse},
Indiana Univ. Math. J. {\bf 22} (1972) 577-594.

\bibitem{Evans1975} J. W. Evans, {\it Nerve axon equations IV: The Stable and Unstable Impulse},
Indiana Univ. Math. J. {\bf 24} (1975) 1169-1190.

\bibitem{Furutani2004} K.\ Furutani, {\em Fredholm-Lagrangian-Grassmannian and the Maslov index,} 
Journal of Geometry and Physics {\bf 51} (2004) 269 -- 331.

\bibitem{GZ1998} R. Gardner and K. Zumbrum, {\it The Gap Lemma and Geometric Criteria for 
Instability of Viscous Shock Profiles}, Comm. Pure. Appl. Math. {\bf 51} (1998) 789-847.

\bibitem{He1981} D. Henry, 
{\em Geometric theory of semilinear parabolic equations}, 
Lect. Notes Math. \textbf{840}, Springer-Verlag, Berlin-New York, 1981.

\bibitem{HLS2018} P. Howard, Y. Latushkin, and A. Sukhtayev,
{\it  The Maslov and Morse indices for Schr\"odinger operators on $\mathbb{R}$}, 
Indiana U. Mathematics Journal {\bf 67} (2018) 1765-1815.

\bibitem{Howard2022a} P. Howard, {\it Renormalized oscillation theory for 
regular linear non-Hamiltonian systems}, Comm. Pure Appl. Anal. {\bf 21} (2022)
4311--4345.

\bibitem{Howard2022b} P. Howard, {\it The Maslov index and spectral counts for 
linear Hamiltonian systems on $\mathbb{R}$}, to appear in J. Dynamics and Differential
equations.

\bibitem{HS2018} P. Howard and A. Sukhtayev, {\it Renormalized oscillation theory
for linear Hamiltonian systems on $[0, 1]$ via the Maslov index},
J. Dynamics and Differential Equations, 
DOI: 10.1007/s10884-021-10121-2.

\bibitem{HS2022} P. Howard and A. Sukhtayev, {\it Renormalized oscillation theory 
for singular linear Hamiltonian systems}, J. Functional Analysis {\bf 283} (2022).

\bibitem{J1988a}  C.\ K.\ R.\ T.\ Jones,
{\em Instability of standing waves for nonlinear Schr\"odinger-type equations},
Ergodic Theory Dynam. Systems {\bf 8} (1988) 119 -- 138.

\bibitem{J1988b}  C.\ K.\ R.\ T.\ Jones,
{\em An instability mechanism for radially symmetric standing
waves of a nonlinear Schr\"odinger equation},
J. Differential Equations {\bf 71} (1988) 34 -- 62.

\bibitem{JM2012} C.\ K.\ R.\ T.\ Jones and R.\ Marangell,
{\em The spectrum of travelling wave solutions to the Sine-Gordon
equation}, Discrete and Cont. Dyn. Sys. {\bf 5} (2012) 925 -- 937.

\bibitem{KP2013} T. Kapitula and K. Promislow, 
{\em Spectral and dynamical stability of nonlinear waves}, 
Springer, New York, 2013.

\bibitem{Morse1934} H. C. M. Morse, {\it The calculus of variations in the large},
AMS Coll. Publ. {\bf 18} (1934).

\bibitem{PSW1993} R. L. Pego, P. Smereka, and M. I. Weinstein,
{\it Oscillatory instability of traveling waves for a KdV-Burgers 
equation}, Physica D {\bf 67} (1993) 45-65.

\bibitem{PW1992} R. L. Pego and M. I. Weinstein,
{Eigenvalues, and instabilities of solitary waves},
Phil. Trans. R. Soc. Lond. A {\bf 340} (1992) 47-94.

\bibitem{Sturm1836} C. Sturm, {\it M\'emoire sur les \'equations diff\'erentielles lin\'eaires du second ordre},
J. math. pures appl. {\bf 1} (1836) 106-186.

\bibitem{ZH1998} K. Zumbrun and P. Howard, {\it Pointwise semigroup methods 
and stability of viscous shock waves}, Indiana U. Math. J. {\bf 47} (1998) 
741-871. See also the errata for this paper: Indiana U. Math. J. {\bf 51}
(2002) 1017--1021. 

\end{thebibliography}
\end{document}